\documentclass[12pt]{amsart}
\usepackage[margin=1.0in]{geometry}
\usepackage{amsmath,amssymb,amsfonts,graphicx,color,fancyhdr,psfrag,comment,enumerate}
\usepackage[latin1]{inputenc}
\usepackage[hyperpageref]{backref}
\usepackage[colorlinks=true, pdfstartview=FitV, linkcolor=blue, citecolor=blue, urlcolor=blue]{hyperref}
\usepackage{tcolorbox,longfbox}
\allowdisplaybreaks

\usepackage{epstopdf}
\usepackage{mathrsfs}
\usepackage{epsfig}
\usepackage{hyperref}
\usepackage{ifthen}
\usepackage{float}
\usepackage{enumerate}
\usepackage{mathtools}
\usepackage[bottom]{footmisc}
\usepackage{tikz}
\usetikzlibrary{matrix,shapes,arrows,positioning,chains}
\usepackage{amssymb,amsmath,amsfonts}
\usepackage{bbm}

\numberwithin{equation}{section}

\DeclareMathOperator\supp{supp}

\newtheorem{theorem}{Theorem}[section]
\newtheorem{proposition}[theorem]{Proposition}
\newtheorem{lemma}[theorem]{Lemma}
\newtheorem{definition}[theorem]{Definition}

\newtheorem{remark}[theorem]{Remark}
\newtheorem{corollary}[theorem]{Corollary}

\newtheorem{convention}[theorem]{Convention}
\newtheorem{assumption}[theorem]{Standing assumption}

\DeclareMathOperator*{\esssup}{ess\,sup}

\begin{document}
\title[Grushin Pseudo-multipliers]{Sparse bounds for pseudo-multipliers associated to Grushin operators, II}

\author[S. Bagchi, R. Basak, R. Garg and A. Ghosh]
{Sayan Bagchi \and Riju Basak \and Rahul Garg \and Abhishek Ghosh}

\address[S. Bagchi]{Department of Mathematics and Statistics, Indian Institute of Science Education and Research Kolkata, Mohanpur--741246, West Bengal, India.}
\email{sayansamrat@gmail.com}

\address[R. Basak]{Department of Mathematics, Indian Institute of Science Education and Research Bhopal, Bhopal--462066, Madhya Pradesh, India.}
\email{riju.basak@iiserb.ac.in}

\address[R. Garg]{Department of Mathematics, Indian Institute of Science Education and Research Bhopal, Bhopal--462066, Madhya Pradesh, India.}
\email{rahulgarg@iiserb.ac.in}

\address[A. Ghosh]{Tata Institute of Fundamental Research, Centre For Applicable Mathematics, Bangalore--560065, Karnataka, India.} 
\email{abhi170791@gmail.com}



\begin{abstract}
In this article, we establish pointwise sparse domination results for Grushin pseudo-multipliers corresponding to various symbol classes, as a continuation of our investigation initiated in \cite{BBGG-1}. As a consequence, we deduce quantitative weighted estimates for these pseudo-multipliers.
\end{abstract}
\maketitle

\tableofcontents

\section{Introduction} \label{sec:intro}
This article is a continuation of our work \cite{BBGG-1} where we have considered pseudo-multipliers associated to Grushin operators and studied sparse operator bounds for pseudo-multipliers with symbols satisfying Mihlin--H\"ormander type conditions. In this article, we  analyse other symbol classes associated to Grushin operators and study their quantitative weighted boundedness. Before discussing our main results, let us first review some well known results regarding pseudo-differential operators on the Euclidean space which serve as the motivation for our investigation. 
\subsection{Pseudo-differential operators on the Euclidean space} \label{subsec:intro-euclidean-pseudo-diff}
The pseudo-differential operator $m(x, D)$ associated with a function $m \in L^{\infty} \left( \mathbb{R}^{n} \times \mathbb{R}^{n} \right) $ is defined by 
$$m(x, D)f(x) := \int_{\mathbb{R}^{n}} m(x, \xi) \widehat{f}(\xi) e^{i x \cdot \xi} \, d\xi, $$
for Schwartz class functions $f$ on $\mathbb{R}^{n}$, where $\widehat{f}$ denotes the Fourier transform of $f$ which is defined by $\widehat{f}(\xi) = (2\pi)^{-n/2} \int_{\mathbb{R}^{n}} f(x) e^{- i x \cdot \xi} \, dx.$ Recall the following class of symbols. For any $\sigma \in \mathbb{R},$ and $ \rho, \delta \geq 0$, $S^{\sigma}_{\rho,\delta} \left( \mathbb{R}^{n} \right)$ denote the set consisting of all  functions $m \in C^\infty \left( \mathbb{R}^{n} \times \mathbb{R}^{n} \right)$ such that for all $\alpha, \beta \in \mathbb{N}^{n}$, 
$$ \left| \partial_{x}^\beta \partial_\xi^\alpha m(x, \xi)\right| \lesssim_{\alpha, \beta} (1+|\xi|)^{\sigma - \rho |\alpha| + \delta |\beta|}.$$ 

Pseudo-differential operators associated with the symbol class $S^{0}_{1,0} \left( \mathbb{R}^{n} \right)$ belong to the realm of Calder\'on--Zygmund operators and they are bounded on $ L^p(\mathbb{R}^{n}) $ for $1 <p < \infty$ (see \cite{SteinHarmonicBook93}) and also they are of weak type $(1, 1)$. For others symbol classes we recall first the following fundamental result of Fefferman.

\begin{theorem}[Fefferman \cite{Fefferman-Israel-Journal}]
Fix $0< a<1$ and $0\leq \delta<1-a$. 
\begin{itemize}
\item For $m\in S^{-na/2}_{1-a, \delta} \left( \mathbb{R}^{n} \right)$, the operator $m(x, D)$ extends to a bounded operator from the Hardy space $H^1(\mathbb{R}^n)$ to $L^1(\mathbb{R}^n)$. Also, $m(x, D)$ admits a bounded extension from $L^{\infty}(\mathbb{R}^n)$ to $BMO(\mathbb{R}^n)$.

\item In general, for any $\sigma \leq 0$ and $m \in S^{\sigma}_{1-a, \delta}$, the operator $m(x, D)$ is bounded on $L^{p}(\mathbb{R}^n)$ provided $|\frac{1}{p}-\frac{1}{2}| \leq - \sigma/na$.
\end{itemize}
\end{theorem}

We also refer to \cite{Alvarez-Hounie} for more general results in this direction. Since we are primarily interested in weighted estimates for pseudo-differential operators, we first mention the work of Miller  \cite{Miller-pseudodifferential-TAMS1982} where the author established weighted $L^p$-estimates for pseudo-differential operators belonging to the symbol class $S^{0}_{1, 0}(\mathbb{R}^n)$. Also, recall the following result of Chanillo and Torchinsky \cite{Chanillo-Torchinsky-weighted-pseudodifferential} where the authors obtained weighted estimates for pseudo-differential operators with symbols belonging to $S^{-na/2}_{1-a, \delta} (\mathbb{R}^n)$ with $0\leq \delta < 1-a < 1$ using suitable estimates for the Fefferman--Stein sharp maximal function $\mathcal{M}^{\sharp}$ acting on $m(x, D)$. 

Let $\mathcal{M}$ stand for the uncentered Hardy--Littlewood maximal function on $\mathbb{R}^n$ and for $1<r<\infty$, $\mathcal{M}_{r}f : = (\mathcal{M}(|f|^r))^{1/r}$.

\begin{theorem}[Chanillo--Torchinsky \cite{Chanillo-Torchinsky-weighted-pseudodifferential}] \label{thm:Chanillo-Torchinsky}
For $m \in S^{-na/2}_{1-a, \delta} \left( \mathbb{R}^{n} \right)$ with $0 \leq \delta < 1-a < 1$, 
$$\mathcal{M}^{\sharp}( m(\cdot, D) f) (x) \lesssim \mathcal{M}_2 f (x), $$
for all $f \in C_c^\infty \left( \mathbb{R}^{n} \right)$. Consequently, for any $2 < p < \infty$ and $w \in A_{p/2}$, the operator $ m(x, D)$ maps $L^p \left( \mathbb{R}^{n}, w \right)$ to itself.
\end{theorem} 

Operators $m(x, D)$ with $m \in S^{-na/2}_{1-a, \delta} \left( \mathbb{R}^{n} \right)$ are more singular in nature and have close connections with strongly singular integral operators studied by Fefferman \cite{Fefferman-Strongly-Acta-Math-1970}. Motivated by the above result, Michalowski, Rule and Staubach \cite{Michalowski-Rule-Staubach-Canad2012} studied $ m(x, D)$ corresponding to $m \in S^{-na}_{1-a, \delta} \left( \mathbb{R}^{n} \right)$. Increasing the decay from $-na/2$ to $-na$, it was proved that $m(x, D)$ is bounded from $L^p \left( \mathbb{R}^{n}, w \right)$ to itself for any $1< p < \infty$ and $w \in A_{p}$. In fact, they made use of some sharp kernel estimates and proved their result by proving the following estimate. 
\begin{theorem}[Michalowski--Rule--Staubach \cite{Michalowski-Rule-Staubach-Canad2012}] \label{thm:Michalowski-Rule-Staubach}
Let $m \in S^{-na}_{1-a, \delta} \left( \mathbb{R}^{n} \right)$ for $0< a < 1$ and $0 \leq \delta \leq 1-a$. Then, for any $1 < p < \infty$ we have 
$$ \mathcal{M}^{\sharp}( m(\cdot, D) f) (x) \lesssim \mathcal{M}_p f (x), $$
for all $f \in C_c^\infty \left( \mathbb{R}^{n} \right)$. Consequently, for any $1 < p < \infty$ and $w \in A_p$, the operator $ m(x, D)$ maps $L^p \left( \mathbb{R}^{n}, w \right)$ to itself.
\end{theorem} 

Recently, Beltran and Cladek \cite{Beltran-Cladek-sparse-pseudodifferential} established sparse domination for pseudo-differential operators on $\mathbb{R}^{n}$. Their work recover not only the previously known weighted boundedness results for classes $S^{- na/2}_{1-a, \delta} \left( \mathbb{R}^{n} \right)$ and $S^{-na}_{1-a, \delta} \left( \mathbb{R}^{n} \right)$, but it also establishes quantitative weighted estimates for operators corresponding to symbols coming from the intermediate classes. Namely, it follows from their results that for $m \in S^{\sigma}_{1-a, \delta}$ with $-na \leq \sigma \leq -na/2$, the operators $m(x, D)$ are bounded on $L^p  \left( \mathbb{R}^{n}, w \right)$ for $r_{e} < p < \infty$ and $w \in A_{p/r_{e}}$ where $r_{e} = - na /  \sigma $. Moreover, they could also prove weighted estimates for symbol classes with lesser decay, namely, for $S^{\sigma}_{1-a, \delta}  \left( \mathbb{R}^{n} \right)$ with $ -na/2 < \sigma < 0$. But, we do not concern our self with those results, and for interested readers we refer Corollary $4.2$ in \cite{Beltran-Cladek-sparse-pseudodifferential}.   

The main contribution of the present article is in establishing some analogous results in the context of Grushin pseudo-multipliers. Our approach also relies on the sparse domination technique. Note that in the context of the Grushin operator, we need not have the adjoint operator $m(x, G_{\varkappa})^*$ to be a pseudo-multiplier operator, a fact crucially used in \cite{Beltran-Cladek-sparse-pseudodifferential}. Therefore, our focus in the present work is to obtain linear sparse forms, which in fact provide a stronger pointwise control. In \cite{Beltran-Cladek-sparse-pseudodifferential}, linear sparse forms were shown only in the case of $\sigma = -na$. We succeed in establishing linear sparse forms in the range $-na \leq \sigma < -na/2$. We also have some weighted boundedness result for $\sigma = -na/2$, which we shall explain in Sections \ref{subsec:about-chanillo-type-class} and \ref{Subsec:intro-joint-functional}. 

In order to present our results in detail, let us first recall some preliminaries.


\subsection{Grushin pseudo-multipliers} \label{subsec:intro-symb-class}
We start with recalling some preliminaries. For more details, we refer \cite{BBGG-1}. Let us denote the points in $\mathbb{R}^{n_1 + n_2}$ as $x = (x^{\prime}, x^{\prime \prime}) \in \mathbb{R}^{n_1} \times \mathbb{R}^{n_2}$. For each $\varkappa \in \mathbb{N}_+$, we consider Grushin operator $G_{\varkappa} = - \Delta_{x^{\prime}} - V_{\varkappa} \left(x^{\prime} \right) \Delta_{x^{\prime \prime}}$, 
where $V_{\varkappa} \left(x^{\prime} \right)$ is either $|x^{\prime}|^{2 \varkappa}$ or $\sum_{j = 1}^{n_1} {x_j^{\prime}}^{2 \varkappa}$. 

Using inverse Fourier transform in the last variable, for a dense class of functions, we can write 
$$ G_{\varkappa} f (x) = \int_{\mathbb{R}^{n_2}} e^{-i \lambda \cdot x^{\prime \prime}} \left( H_{\varkappa} (\lambda) f^{\lambda} \right) (x^{\prime}) \, d\lambda, $$
where, for $\lambda \neq 0$, $H_{\varkappa}(\lambda) = - \Delta_{x^{\prime}} + |\lambda|^2 |x^{\prime}|^{2 \varkappa}$ or $H_{\varkappa}(\lambda) = - \Delta_{x^{\prime}} + |\lambda|^2 \sum_{j = 1}^{n_1} {x_j^{\prime}}^{2 \varkappa} $, depending on the choice of $V_{\varkappa} (x^{\prime})$ and $f^\lambda(x^{\prime}) = (2 \pi)^{- n_2} \int_{\mathbb{R}^{n_2}} f (x^{\prime}, x^{\prime \prime}) e^{i \lambda \cdot x^{\prime \prime}} \, d x^{\prime \prime}$. 

In particular, when $\varkappa = 1$, for each $\lambda \neq 0$, operators $H_{1}(\lambda) = H(\lambda) = - \Delta_{x^{\prime}} + |\lambda|^2 |x^{\prime}|^2$ are the scaled Hermite operators on $\mathbb{R}^{n_1}$ and we simply denote $G_{1}$ by $G$ itself.

Using spectral decomposition of $H_{\varkappa}(\lambda)$ we have the following representation 
$$G_{\varkappa} f (x) = \int_{\mathbb{R}^{n_2}} e^{-i \lambda \cdot x^{\prime \prime}} \sum_{k \in \mathbb{N}} |\lambda|^{\frac{2}{\varkappa + 1}} \nu_{\varkappa, k} \left( f^{\lambda}, h^{\lambda}_{\varkappa, k} \right) h^{\lambda}_{\varkappa, k} (x^\prime) \, d\lambda,$$ 
where $ \{ h_{\varkappa, k} : k \in \mathbb{N} \}$ is a complete orthonormal basis of $L^2 \left( \mathbb{R}^{n_1} \right)$ such that $H_{\varkappa} (1) h_{\varkappa, k} = \nu_{\varkappa, k} h_{\varkappa, k}$ with $0 < \nu_{\varkappa, 1} \leq \nu_{\varkappa, 2} \leq \nu_{\varkappa, 3} \leq \ldots$ and $\lim_{k \to \infty} \nu_{\varkappa, k} = \infty$, and for each $k \in \mathbb{N}$ and $\lambda \neq 0$, $$ h^{\lambda}_{\varkappa, k} (x^\prime):= |\lambda|^{\frac{n_1}{2 (\varkappa + 1)}} h_{\varkappa, k} \left( |\lambda|^{\frac{1}{\varkappa + 1}} x^\prime \right).$$
For $\varkappa = 1 $, the eigen functions $h_{\varkappa, k}$ are the well-known scaled Hermite functions. 

Let us now define pseudo-multipliers associated to Grushin operators.

\begin{definition}
Given $m \in L^\infty \left( \mathbb{R}^{n_1 + n_2} \times \mathbb{R}_+ \right) $, the Grushin pseudo-multiplier $m(x, G_{\varkappa})$ is defined densely on $L^2 \left( \mathbb{R}^{n_1 + n_2} \right)$ by 
\begin{equation} \label{Gru-pseudo}
m(x, G_{\varkappa}) f(x) := \int_{\mathbb{R}^{n_2}} e^{-i \lambda \cdot x^{\prime \prime}} \sum_{k \in \mathbb{N}} m \left( x, |\lambda|^{\frac{2}{\varkappa + 1}} \nu_{\varkappa, k} \right) \left( f^{\lambda}, h^{\lambda}_{\varkappa, k} \right) h^{\lambda}_{\varkappa, k} (x^\prime) \, d\lambda. 
\end{equation}
\end{definition}
Throughout this article, $K_{m (x, G_{\varkappa})}$ will denote the integral kernel of the Grushin pseudo-multiplier operator $m (x, G_{\varkappa})$. Also, $m(G_{\varkappa})$ will simply denote the Grushin multiplier corresponding to a bounded function $m$ defined on $\mathbb{R}_{+}$. 

Recall that Grushin operator $G_{\varkappa}$ can be expressed as a negative sum of $X^2_{j}$'s and $X^2_{\alpha, k}$'s where $X_{j}$ and $X_{\alpha, k}$ are first order gradient vector fields defined as follows
\begin{align} \label{first-order-grad}
X_{j} = \frac{\partial}{\partial x_j^{\prime}} \quad \textup{and} \quad X_{\alpha, k} = {x^{\prime}}^{\alpha} \frac{\partial}{\partial x_k^{\prime \prime}},
\end{align}
for $1 \leq j \leq n_1$, $1 \leq k \leq n_2$, and $\alpha \in \mathbb{N}^{n_1}$ with $|\alpha| = \varkappa$. 

Let us denote by $X$ the first order gradient vector field 
\begin{align} \label{first-order-grad-vector}
X := (X_j, X_{\alpha,k})_{1 \leq j \leq n_1, \, 1 \leq k \leq n_2, \, |\alpha| = \varkappa}. 
\end{align}

We consider symbol classes $\mathscr{S}^{\sigma}_{\rho, \delta}(G_{\varkappa})$, defined as follows.
\begin{definition} \label{grushin-symb-old-def1}
For any $\sigma \in \mathbb{R}$ and $\rho, \delta \geq 0$, we  say that a function $m \in C^\infty \left(\mathbb{R}^{n_1 + n_2} \times \mathbb{R}_+ \right)$, belongs to the symbol class $\mathscr{S}^{\sigma}_{\rho, \delta}(G_\varkappa)$ if it satisfies the following estimate: 
\begin{align} \label{def-grushin-symb-old}
\left|X^\Gamma \partial_{\eta}^l m(x, \eta) \right| \lesssim_{\Gamma, l} (1+\eta)^{\frac{\sigma}{2}- (1 + \rho) \frac{l}{2} + \delta \frac{|\Gamma|}{2}} 
\end{align} 
for all $\Gamma \in \mathbb{N}^{n_{0}}$ and $l \in \mathbb{N}$, where $n_0 = n_1 + n_2 \binom{\varkappa + n_1 - 1}{n_1 - 1}$. \end{definition}

For $N \in \mathbb{N}$, we define the seminorm for the above symbol classes as follows:
\begin{align} \label{def-grushin-symb-seminorm}
\| m \|_{\mathscr{S}^{\sigma, N}_{\rho,\delta}} = \sup_{|\Gamma| + l \leq N} \sup_{x, \eta}  (1+\eta)^{-\frac{\sigma}{2} + (1 + \rho) \frac{l}{2} - \delta \frac{|\Gamma|}{2}} \left| X^\Gamma \partial_{\eta}^l m(x, \eta) \right|. 
\end{align} 
For convenience, we will use a shortened notation $\| m \|_{\mathscr{S}^\sigma_{\rho,\delta}}$ where the number of derivatives $N$ should be clear from the context where it is used. 

In \cite{Bagchi-Garg-1}, an analogue of the Calder\'on-Vaillancourt theorem for the Grushin operator $G = G_1$ was shown to be true for symbol classes $\mathscr{S}^{0}_{\rho, \delta}(G)$ with $0 \leq \delta < \rho \leq 1$. In \cite{BBGG-1}, we were primarily concerned with symbol classes $\mathscr{S}^{0}_{1, \delta}(G_\varkappa), \, \delta<1$, and we established appropriate sparse bounds with an emphasis on the number of derivatives of symbol functions. 

The present article is dedicated to study sparse bounds and quantitative weighted estimates for the symbol classes $\mathscr{S}^{-Qa/q}_{1-a, \delta}(G_\varkappa)$ for $0 < a  < 1$, $\delta \leq 1-a$, and $1 \leq q < 2$. 

We also prove weighted estimates for pseudo-multipliers corresponding to the symbol class $\mathscr{S}^{-Qa/2}_{1-a, \delta}(\boldsymbol{L}, \boldsymbol{U})$, with $0 < a  < 1$, $\delta \leq 1-a$, associated with the joint functional calculus of the Grushin operator $G$. 

In the next few subsections we describe our main results.

\subsection{Main results for classes \texorpdfstring{$\mathscr{S}^{ - Qa/q  }_{1-a, \delta}(G_{\varkappa})$}{} with \texorpdfstring{$1\leq q< 2$}{}} \label{subsec:intro-main-results}

Before discussing our results, let us record here that unless otherwise stated, throughout this article we work under the following convention on parameters. 

\begin{convention}
We always take $0 < a < 1, \, 0 \leq \delta \leq 1-a $, and $\, 0 \leq \delta \leq \rho < 1$. 
\end{convention} 

We also have the following standing assumption, which plays an important role in our analysis. 
\begin{assumption} \label{assumption-weighted-L2}
We assume that for each $m \in \mathscr{S}^{0}_{\rho, \delta}(G_\varkappa)$, the operator 
\begin{align*}
m ( \cdot, G_\varkappa): L^2 ( |B(\cdot, 1)|^{ \mathfrak{b} } ) \to L^2 (|B(\cdot, 1)|^{\mathfrak{b}}) 
\end{align*}
is bounded for all $0 \leq \mathfrak{b} < 1$, with the operator norm bound depending only on $\mathfrak{b}$ and the symbol seminorm $\| m \|_{\mathscr{S}^0_{\rho, \delta}}$. 
\end{assumption}

The following remark is in order.

\begin{remark} \label{rem:justifying-assumption-weighted-L2} 
Standing assumption \ref{assumption-weighted-L2} may at first look artificial but we will show in Proposition \ref{prop:weighted-CV} that given any $m \in \mathscr{S}^{0}_{\rho, \delta}(G)$, with $\delta < \rho$, the operator $m(x, G)$ maps $L^2(|B(\cdot, 1)|^{\mathfrak{b}})$ to itself for all $\, \mathfrak{b} \geq 0$. 

In general, since even an analogue of the Calder\'on-Vaillancourt type theorem is not known, we have to assume throughout this article that for any symbol $ m \in \mathscr{S}^{0}_{\rho, \delta} (G_\varkappa)$, the pseudo-multiplier operator $m(x, G_\varkappa)$ is bounded on $ L^2 ( |B(\cdot,  1)|^{\mathfrak{b}})$ for all $ 0 \leq \mathfrak{b} < 1 $.  
\end{remark}

We now present the following theorem regarding the end-point boundedness of pseudo-multipliers associated to the symbol class $\mathscr{S}^{-Qa}_{1 - a, \delta}(G_\varkappa)$. This result is not only quintessential in our subsequent proofs of sparse domination, it is also very important in its own right. Let $H^1_{G_\varkappa} \left( \mathbb{R}^{n_1 + n_2} \right)$ denote the Hardy space associated to the Grushin operator $G_\varkappa$ which we will explain in detail in Subsection~\ref{Hardy-interpolation}.
\begin{theorem} \label{thm:unweightedLp:pseudo}
Given $m \in \mathscr{S}^{-Qa}_{1 - a, \delta}(G_\varkappa)$, the operator $m(x, G_\varkappa)$ is bounded from $H^1_{G_\varkappa} \left( \mathbb{R}^{n_1 + n_2} \right)$ to $L^1 \left( \mathbb{R}^{n_1 + n_2} \right)$. Consequently, $m(x, G_\varkappa)$ is $L^p$-bounded for all $1 < p < 2$. 
\end{theorem}

Since $( \mathbb{R}^{n_1 + n_2}, d_{G_\varkappa}, |\cdot|)$ is a homogeneous space, in order to develop sparse domination results we rely on the dyadic structure provided by Christ's dyadic grid $\mathcal{S}$. For more relevant details, we refer to Subsection $2.5$ of \cite{BBGG-1}. We say a collection of measurable sets $S \subset \mathcal{S}$ to be a $\eta$-sparse family (for some $0 < \eta < 1$) if for every member $\mathcal{Q}\in {S}$ there exists a set $E_{\mathcal{Q}} \subseteq \mathcal{Q}$ such that $|E_{Q}| \geq \eta |\mathcal{Q}|$. Corresponding to a sparse family ${S}$ and $1 \leq r < \infty$, we define the sparse operator as follows: 
\begin{align} \label{def:Sparse-operator} 
\mathcal{A}_{r, S} f(x) = \sum_{\mathcal{Q} \in S} \left( \frac{1}{|\mathcal{Q}|} \int_{\mathcal{Q}}|f|^r \right)^{1/r} \chi_{\mathcal{Q}}(x).
\end{align} 
We simply write $\mathcal{A}_{S}$ for $\mathcal{A}_{1, S}$. The precise quantitative estimates for sparse operators in terms of the $A_p$ characteristic follow from Proposition $4.1$ in \cite{Lorist-pointwaise-sparse2021}. 
\begin{theorem}
\label{thm:Main-Sparse-Result}
Let $ \sigma = - Qa/q $ for some $ 1 \leq q < 2$. Given $ m \in \mathscr{S}^{\sigma}_{1 - a, \delta}(G_{\varkappa}) $, for every compactly supported function and every $ q < r < \infty $, there exists a sparse family ${S}$ such that
\begin{align}
|m(x, G_\varkappa)f(x)| \lesssim \mathcal{A}_{r, S} f(x)
\end{align}
holds true a.e. $x \in \mathbb{R}^{n_1 + n_2}$.
\end{theorem}

We employ the following grand maximal truncated operator in our proof of Theorem \ref{thm:Main-Sparse-Result}: For any linear operator $T$, and any $s > 0$, we consider the grand maximal truncated operator defined by 
\begin{align} \label{def:grand-maximal-truncated-operator}
\mathcal{M}^{\#}_{T, s} f(x) = \sup_{B: B \ni x} \esssup \limits_{y,z \in B} \left| T(f \chi_{\mathbb{R}^{n_1 + n_2}\setminus s B})(y)-T(f \chi_{\mathbb{R}^{n_1 + n_2}\setminus s B})(z) \right|, 
\end{align}
where the supremum is taken over all balls $B$ containing the point $x$.

It is known that the sparse domination results for $T$ follow once we have an appropriate end-point boundedness of $\mathcal{M}^{\#}_{T, s}$. More precisely, for a sublinear operator $T$, if we have 
\begin{itemize}
\item $T$ is of weak type $(p, p)$ for some $1 \leq p < \infty$, 

\item $\mathcal{M}^{\#}_{T, s}$ is weak type $(q, q)$ for some $1 \leq q < \infty$, 
\end{itemize}
and $s \geq \frac{3 C_{0}^{2}}{\delta}$, then there is an $0 < \eta < 1$ such that for every compactly supported bounded measurable function $f$, there exist an $\eta$-sparse family $S \subset \mathcal{S}$ such that for almost every $x \in \mathbb{R}^{n_1 + n_2}$, we have 
\begin{align*}
|T f(x)| \lesssim_{S, s} C_T \mathcal{A}_{r, S} f(x), 
\end{align*} 
where $r=\max\{p, q\}$ and $C_T = \|T\|_{L^p \to L^{p, \infty}} + \| \mathcal{M}^{\#}_{T, s} \|_{L^q \to L^{q, \infty}}$. 

The above methodology was first shown by Lerner and Ombrosi on Euclidean spaces in \cite{Lerner-Ombrosi-pointwaise-sparse2020}, and the same was extended to spaces of homogeneous type by Lorist in \cite{Lorist-pointwaise-sparse2021}.

Concluding weighted estimates from sparse operator bounds is standard by now and hence we do not provide details and rather conclude the following weighted estimates. We refer Proposition $4.1$ in \cite{Lorist-pointwaise-sparse2021} for more details. 

\begin{theorem} \label{thm:Main-results}
Let $\sigma = - Qa/q$ for some $1 \leq q < 2$. Given $m \in \mathscr{S}^{\sigma}_{1-a, \delta}(G_{\varkappa})$, the operator $m(x, G_\varkappa)$ is bounded on $L^p(w)$ to itself for $w \in A_{p/q}(\mathbb{R}^{n_1 + n_2})$ and for all $q < p < \infty$.
\end{theorem}


\subsection{About classes \texorpdfstring{$ \mathscr{S}^{-Qa/2}_{1-a, \delta}(G_{\varkappa})$}{}} 
\label{subsec:about-chanillo-type-class}
In the previous subsection, we explained sparse domination and weighted estimates for pseudo-multipliers with symbols coming from $\mathscr{S}^{-{Qa}/{q}}_{1-a, \delta}(G_{\varkappa})$ classes with $1 \leq q < 2$. It turns out that our approach is not well suited for the class $\mathscr{S}^{-{Qa}/{2}}_{1-a, \delta}(G_{\varkappa})$. More precisely, in the proof of Theorem \ref{thm:Main-Sparse-Result}, we make a crucial use of the weighted Plancherel estimates of the integral kernels of the following form
\begin{align} \label{Motivation-Plancherel} 
|B(x,R^{-1})| \int_{\mathbb{R}^{n_1 + n_2}} (1+Rd(x,y))^{2\mathfrak{r}} \left| K_{m (x, G_{\varkappa})}(x, y) \right|^2  dy \quad \lesssim_{\mathfrak{r}, \epsilon} \sup_{x_{0}}  \| m(x_0, R^2 \cdot) \|^2_{W^{\infty}_{\mathfrak{r} + \epsilon}},
\end{align}
for every every $\mathfrak{r}, \epsilon > 0$ and any bounded Borel function $m : \mathbb{R}^{n_1 + n_2} \times \mathbb{R} \rightarrow \mathbb{C}$ whose support in the last variable is in $[0, R^2]$ for any $R > 0$. For a detailed proof of this estimate, we refer Lemma $4.3$ in \cite{BBGG-1}. 

We would like to note the fact that in the Euclidean setting it is possible to replace the norm ${W^{\infty}_{\mathfrak{r} + \epsilon}}$ by ${W^{\infty}_{\mathfrak{r}}}$ in right hand side of \eqref{Motivation-Plancherel} using Hausdorff--Young theorem. While working with symbols $m\in \mathscr{S}^{-{Qa}/{q}}_{1-a, \delta}(G_{\varkappa})$ with $1 \leq q < 2$, we are able to establish sparse domination even with the presence of the norm ${W^{\infty}_{\mathfrak{r} + \epsilon}}$, but difficulties arising from Grushin metric prohibit us to do the same for the class $\mathscr{S}^{-{Qa}/{2}}_{1-a, \delta}(G_{\varkappa})$ even if we assume ${W^{\infty}_{\mathfrak{r}}}$ norm in the right hand side of \eqref{Motivation-Plancherel}. For more details, we refer to Remark \ref{rem:about-chanillo-type-class}. 

However, we are able to employ the machinery of the Fefferman--Stein sharp maximal function and good-$\lambda$-inequalities to conclude weighted boundedness for pseudo-multiplier operator for the symbol class $\mathscr{S}^{-Qa/2}_{1-a, \delta}(\boldsymbol{L}, \boldsymbol{U})$ associated to the joint functional calculus of $G$, under some more assumptions on the symbol function. We discuss this set-up and state our weighted boundedness result in that context in the next subsection. 


\subsection{Joint functional calculus and results for classes \texorpdfstring{$\mathscr{S}^{-Qa/2}_{1-a, \delta}(\boldsymbol{L}, \boldsymbol{U})$}{}} 
\label{Subsec:intro-joint-functional}

Once again, we follow the terminology of \cite{BBGG-1} (see Subsection $2.2$ in \cite{BBGG-1}). 

Let us consider the following family of operators: 
\begin{align} \label{eq:operatorsLandU}
L_{j} = (-i \partial_{x_j'})^2 + {x_j^\prime}^{2} \sum_{k=1}^{n_2} (-i\partial_{x_k''})^2 \quad \textup{and} \quad U_k = - i \partial_{x_k''},
\end{align} 
for $j= 1, 2, \ldots, n_1$ and $k = 1, 2, \ldots, n_2$. The operators $L_1, L_2 \ldots L_{n_1}, U_1, U_2, \ldots , U_{n_2}$ are essentially self adjoint on $C_c^\infty (\mathbb{R}^{n_1 + n_2})$ and their spectral resolutions commute. The same is true for all polynomials in $L_1, L_2 \ldots L_{n_1}, U_1, U_2, \ldots , U_{n_2}$. Hence, they admit a joint functional calculus on $L^2(\mathbb{R}^{n_1 + n_2})$ in the sense of the spectral theorem. 

Let us write $\boldsymbol{L} = (L_1, L_2,\ldots, L_{n_1})$, $\boldsymbol{U} = (U_1, U_2, \ldots, U_{n_2})$ and $ \tilde{1} = (1, 1, \ldots, 1) \in \mathbb{R}^{n_1}$. Now, given a function $m \in L^\infty \left( \mathbb{R}^{n_1 + n_2} \times (\mathbb{R}_+)^{n_1} \times (\mathbb{R}^{n_2} \setminus \{0\}) \right)$, the pseudo-multiplier operator $m(x, \boldsymbol{L}, \boldsymbol{U})$ is (densely) defined on $L^2(\mathbb{R}^{n_1 + n_2})$ by 
\begin{align} \label{def:joint-Gru-pseudo}
m(x, \boldsymbol{L}, \boldsymbol{U}) f(x) : = \int_{\mathbb{R}^{n_2}} e^{-i \lambda \cdot x^{\prime \prime}} \sum_{\mu \in \mathbb{N}^{n_1}} m \left( x, (2 \mu + \tilde{1}) |\lambda|, \lambda \right) \left( f^{\lambda}, \Phi^{\lambda}_{\mu} \right) \Phi^{\lambda}_{\mu} (x^\prime) \, d\lambda, 
\end{align} 
where $\Phi^{\lambda}_{\mu}$ are the scaled Hermite functions. 

\begin{definition} \label{def:joint-symb-def1}
For any $\sigma \in \mathbb{R}$ and $\rho, \delta \geq 0$, we define the symbol class $\mathscr{S}^{\sigma}_{\rho, \delta}(\boldsymbol{L}, \boldsymbol{U})$ to be the collection of all $m \in C^\infty \left(\mathbb{R}^{n_1 + n_2} \times (\mathbb{R}_+)^{n_1} \times (\mathbb{R}^{n_2} \setminus \{0\}) \right)$ which satisfy the following estimate: 
\begin{align} \label{def:grushin-symb}
\left|X^\Gamma \partial_{\tau}^{\theta} \partial_{\kappa}^{\beta} m(x, \tau, \kappa) \right| \lesssim_{\Gamma, \theta, \beta} (1 + |\tau| + |\kappa|)^{\frac{\sigma}{2} - \frac{(1 + \rho)}{2}(|\theta| + |\beta|) +  \frac{\delta}{2} |\Gamma|} 
\end{align} 
for all $\Gamma \in \mathbb{N}^{n_1 + n_1 n_2}$, $\theta \in \mathbb{N}^{n_1}$ and $\beta \in \mathbb{N}^{n_2}$. 
\end{definition}

Similar to \eqref{def-grushin-symb-seminorm}, for $N \in \mathbb{N}$, we define the seminorm for the above symbol classes as follows:
\begin{align} \label{def-grushin-joint-symb-seminorm}
\| m \|_{\mathscr{S}^{\sigma, N}_{\rho,\delta}} = \sup_{|\Gamma| + |\theta| + |\beta| \leq N} \sup_{x, \tau, \kappa} (1 + |\tau| + |\kappa|)^{- \frac{\sigma}{2} + \frac{(1 + \rho)}{2}(|\theta| + |\beta|) -  \frac{\delta}{2} |\Gamma|} \left|X^\Gamma \partial_{\tau}^{\theta} \partial_{\kappa}^{\beta} m(x, \tau, \kappa) \right|, 
\end{align} 
and here also for convenience we will use a shortened notation $\| m \|_{\mathscr{S}^\sigma_{\rho,\delta}}$ without specifying the number of involved derivatives. 

In \cite{Bagchi-Garg-1}, first and third authors proved the following $L^2$-boundedness result. 
\begin{theorem}[\cite{Bagchi-Garg-1}] \label{thm:L2-joint-calc}
Let $m \in \mathscr{S}^{0}_{\rho, \delta}(\boldsymbol{L}, \boldsymbol{U})$. 
\begin{enumerate}
\item If $\delta < \rho$, then $m(x, \boldsymbol{L}, \boldsymbol{U})$ extends to a bounded operator on $L^2(\mathbb{R}^{n_1 + n_2})$. 

\item For $m \in \mathscr{S}^0_{\rho, \rho} (\boldsymbol{L}, \boldsymbol{U})$, with $0 \leq \rho < 1$, if we further assume that 
\begin{align} \label{def:grushin-symb-vanishing-0-condition}
\lim_{\kappa \to 0} \partial_{\kappa}^{\beta} m(x, \tau, \kappa) = 0, 
\tag{CancelCond} 
\end{align} 
for all $\beta \in \mathbb{N}^{n_2}$ with $|\beta| \leq 4 N_0 = 4 \left( \lfloor\frac{Q}{4}\rfloor + 1 \right)$, then $m(x, \boldsymbol{L}, \boldsymbol{U})$ extends to a bounded operator on $L^2(\mathbb{R}^{n_1 + n_2})$.
\end{enumerate}
\end{theorem}

In \cite{BBGG-1}, we studied sparse bound result for symbol classes $\mathscr{S}^{0}_{1, \delta}(\boldsymbol{L}, \boldsymbol{U})$, under the same cancellation condition \eqref{def:grushin-symb-vanishing-0-condition} on symbols. In this paper, as a continuation, we have the following weighted bounded results for symbol classes $\mathscr{S}^{-Qa/2}_{1-a, \delta}(\boldsymbol{L}, \boldsymbol{U})$. 

\begin{theorem} \label{thm:joint-fun-Channilo-type}
Let $m \in \mathscr{S}^{-Qa/2}_{1-a, \delta}(\boldsymbol{L}, \boldsymbol{U})$ be such that it satisfies the cancellation condition \eqref{def:grushin-symb-vanishing-0-condition} for all $\beta \in \mathbb{N}^{n_2}$ with $|\beta| \leq \max \left\{ 4 \left( \lfloor\frac{Q}{4}\rfloor + 1 \right), \lfloor \frac{Q}{2} + \frac{1}{1-a} \rfloor + 1 \right\} $. Then, the operator $m( x,\boldsymbol{L}, \boldsymbol{U})$ is bounded on $L^p(w)$ for all $w \in A_{p/2}(\mathbb{R}^{n_1 + n_2})$ and $2 < p < \infty$.
\end{theorem}


\subsection{Organisation of the paper} \label{subsec:intro-organisation}
Concluding weighted $L^p$ estimates has always been an important theme of research in the study of pseudo-differential operators. In this article, employing modern tools like sparse operators and others, we prove series of weighted estimates for Grushin pseudo-multipliers. 

We have organised the article as follows.
\begin{itemize}
\item In Section \ref{sec:prelim}, we recall the relevant preliminary details. Here, we prove some basic results which are not only essential to establish our main results but are also important in their own rights. In Subsection \ref{Hardy-interpolation}, we list out an interpolation result using the Hardy space as an end-point. We also need a suitable Hardy--Littlewood--Sobolev inequality for $G_\varkappa$, and we prove the same (see Theorem \ref{theorem:Sobolev-embedding}) in Subsection \ref{subsec:HLS-inequality}. Finally, in Proposition \ref{prop:weighted-CV}, we prove a weighted analogue of Calder\'on-Vaillancourt type theorem for Grushin pseudo-multipliers. 

\item Section~\ref{sec:Kernel-estimates} is dedicated to recalling or proving a number of kernel estimates for pseudo-multipliers associated with $G_{\varkappa}$ as well as for $(\boldsymbol{L}, \boldsymbol{U})$ assuming the cancellation condition \eqref{def:grushin-symb-vanishing-0-condition}. Most of the results in this section are a consequence of kernel estimates of \cite{BBGG-1}.

\item We develop results concerning unweighted boundedness of pseudo-multipliers corresponding to various symbol classes in Section \ref{sec:unweighted-boundedness}. In particular, the end-point $H^1$-$L^1$ boundedness for $\mathscr{S}^{-Qa}_{1-a, \delta}(G_{\varkappa})$ class symbols is addressed in Subsection \ref{subsec:H1-L1 boundedness}. In establishing this result, we make use of suitable kernel estimates, which we prove in Lemma \ref{lem:unweightedLp-kernel}. Let us mention that our proofs in Subsection \ref{subsec:H1-L1 boundedness} are inspired by the work of \'Alvarez--Hounie \cite{Alvarez-Hounie}, however, due to the non-Euclidean nature of the Grushin metric we had to overcome many difficulties. Subsequently, in Subsection \ref{subsec:Lp-boundedness-intermediatary-classes}, as an application of the interpolation between classes $\mathscr{S}^{-Qa}_{1-a, \delta}(G_{\varkappa})$ and $\mathscr{S}^{-Qa/2}_{1-a, \delta}(G_{\varkappa})$, we obtain unweighted $L^p$-boundedness for symbol classes  $\mathscr{S}^{-Qa/q}_{1-a, \delta}(G_{\varkappa})$ with $1 < q < 2$. 

\item We prove our sparse domination results in Section \ref{sec:sparse-domination-results}. In Subsection \ref{subsec: estimates-grand-maximal}, we prove an end-point boundedness for the grand maximal truncated operator $\mathcal{M}^{\sharp}_{T, s}$ which is among the key ingredients in our proof of sparse domination, using which we conclude the proof of Theorem \ref{thm:Main-Sparse-Result} in Subsection \ref{subsec:Proof-thm:Main-Sparse-Result}.

\item Section \ref{sec:joint-funct-calculus-weighted-boundedness} deals with the analysis related to the symbol class $\mathscr{S}^{-Qa/2}_{1-a, \delta}(\boldsymbol{L}, \boldsymbol{U})$. We prove a pointwise domination of the Fefferman--Stein maximal function in Subsection \ref{subsec:Proof-thm:Chanilloy-type}, and the article culminates with the proof of Theorem \ref{thm:joint-fun-Channilo-type} in Subsection \ref{subsec:Proof-thm:joint-fun-Channilo-type}. 
\end{itemize}

\noindent \textbf{Notation:} For any pair of positive real numbers $A_1$ and $A_2$, by the expression $A_1 \lesssim A_2$ we mean $A_1 \leq C A_2$ for some $C > 0$. We write $A_1 \lesssim_{\epsilon} A_2 $ whenever the implicit constant $C$ may depend on a parameter $\epsilon$. The notation $A_1 \sim A_2$ stands for $A_1 \lesssim A_2$ and $A_2 \lesssim A_1$. For a general vector $\tau = ( \tau_1, \ldots, \tau_{n_1}) \in \mathbb{R}^{n_1}$, we write $|\tau|_1 = \sum_{j=1}^{n_1} |\tau_j|$ and $|\tau| = \left( \sum_{j=1}^{n_1} |\tau_j|^2 \right)^{1/2}$. Whenever it is obvious that $\mu \in \mathbb{N}^{n_1}$, by abuse of notation, we write $|\mu|$ in pace of $|\mu|_1$. 


\section{Preliminaries and basic results} \label{sec:prelim}

We start with recalling the control distance function $\tilde{d}$ associated with the sub-Riemannian structure of the Grushin operator $G_\varkappa$. For details, we refer to Section $2.1$ of \cite{BBGG-1}. The asymptotic description of $\tilde{d}$ is known to be:
\begin{align*} 
\tilde{d} (x,y) \sim d(x,y) := \left|x^{\prime} - y^\prime \right| + 
\begin{cases}
\frac{\left|x^{\prime \prime} - y^{\prime \prime}\right|}{\left(\left|x^{\prime} \right| + \left|y^\prime \right|\right)^{\varkappa}} &\textup{ if } \left|x^{\prime \prime} - y^{\prime \prime}\right|^{1/(1+\varkappa)} \leq \left|x^{\prime} \right| + \left|y^\prime \right| \\
\left|x^{\prime \prime} - y^{\prime \prime}\right|^{1/(1+\varkappa)} &\textup{ if } \left|x^{\prime \prime} - y^{\prime \prime}\right|^{1/(1+\varkappa)} \geq \left|x^{\prime} \right| + \left|y^\prime \right|. 
\end{cases}
\end{align*}

Since $\tilde{d}$ is a metric, it follows that $d$ is a quasi-metric, that is, there exists a constant $C_0 > 1$ such that for all $x, y, z \in \mathbb{R}^{n_1 + n_2}$, 
\begin{align} \label{def:quasi-constant}
d(x,y) \leq C_0 \left( d(x,z) + d(z,y) \right)    
\end{align}

Equipped with the Lebesgue measure $|\cdot|$, it is known that $(\mathbb{R}^{n_1 + n_2}, d)$ is a space of homogeneous type with homogeneous dimension $Q:=n_1+(1 + \varkappa) n_2$. We refer Proposition $5.1$ of \cite{RobinsonSikoraDegenerateEllipticOperatorsGrushinTypeMathZ2008} for more details. Let us now recall the following mean-value estimate that we need in our subsequent kernel estimates. 

\begin{lemma}[Lemma 2.2, \cite{BBGG-1}]  \label{lem:general-grushin-Mean-value}
There exist constants $C_{1, \varkappa}, C_{2, \varkappa} > 0$ (depending also on $n_1$ and $n_2$) such that for any ball $B(x_0, r)$ and points $x, y \in B(x_0, r)$, there exists a $\tilde{d}$-length minimising curve $\gamma_0 : [0, 1] \to B(x_0, C_{1, \varkappa} \, r)$ joining $x$ to $y$, and 
\begin{align*} 
\left| f(x) - f(y) \right| \leq & C_{2, \varkappa} \, d(x, y) \int_{0}^{1} \left|Xf(\gamma_0(t)) \right| \, dt 
\end{align*}
for any $f \in C^1 \left( B(x_0, C_{1, \varkappa} \, r) \right)$.
\end{lemma} 

Let us state another convention that we shall follow throughout the article. 

\begin{convention} 
For a given symbol function $m (x, \eta)$ defined on $\mathbb{R}^{n_1 + n_2} \times \mathbb{R}_+$, we shall use the same notation $m (x, \eta)$ for its extension to $\mathbb{R}^{n_1 + n_2} \times \mathbb{R}$ with $m (x, \eta) = 0$ whenever $\eta \notin \mathbb{R}_+$. A similar remark holds for symbol functions $m (x, \tau)$ and $m (x, \tau, \kappa)$ defined on $\mathbb{R}^{n_1 + n_2} \times (\mathbb{R}_+)^{n_1}$ and $\mathbb{R}^{n_1 + n_2} \times (\mathbb{R}_+)^{n_1} \times \left( \mathbb{R}^{n_2} \setminus \{0\} \right)$ respectively. 
\end{convention} 

Before moving on, let us mention that inn our analysis we mostly decompose the spectrum of pseudo-multipliers into dyadic pieces. In that direction, we choose and fix $\psi_0 \in C_c^\infty((-2,2))$ and $\psi_1 \in C_c^\infty((1/2,2))$ such that $0 \leq \psi_0 (\eta), \psi_1 (\eta) \leq 1$, and 
\begin{align} \label{decompose-dyadic-spectral-symbol-frequency}
\sum_{j=0}^\infty \psi_j(\eta) = 1
\end{align}
for all $\eta \geq 0$, where $\psi_j (\eta) = \psi_1 \left( 2^{-(j-1)} \eta \right)$ for $j \geq 2.$ 


\subsection{Hardy spaces and interpolation}
\label{Hardy-interpolation}
The theory of Hardy spaces corresponding to Grushin operators is well developed and we recall it according to our requirement. For $\varkappa=1$, the Hardy space for the Grushin operator $G$ was studied in \cite{Dziubanski-Jotsaroop-Hardy-BMO-Grushin} and subsequently more general results were established in  \cite{Preisner-Sikora-Yan-Hardy-BMO-harmonic-weights-2019}. 

Let $H^1_{G_\varkappa}$ denote the Hardy space associated to the Grushin operator $G_\varkappa$, defined as follows: 
$$ H^1_{G_\varkappa}:=\{f\in L^1(\mathbb{R}^{n_1 + n_2}): \mathscr{M}_{G_\varkappa}f\in L^1(\mathbb{R}^{n_1 + n_2})\}, $$
where $\mathscr{M}_{G_\varkappa}f(x)=\sup_{t>0}|e^{-tG_{\varkappa}}f(x)|$ is the maximal function associated to the heat semigroup and the norm is defined as $\|f\|_{H^1_{G_\varkappa}}:=\|\mathscr{M}_{G_\varkappa}f\|_{L^1(\mathbb{R}^{n_1 + n_2})}$. 

Recall that the heat kernel $e^{-t G_{\varkappa}}(x, y)$ satisfy the following two-sided Gaussian bounds (see, for example, Theorem $2.1$ in \cite{DziubanskiSikoraLieApproach}): There exist constants $c, c' > 0$ such that 
\begin{align} \label{Both-sided-bounds}
\left| B(x, \sqrt{t}) \right|^{-1} \exp \left( \frac{-c'}{t} d(x,y)^2 \right) \lesssim \left|e^{-t G_{\varkappa}}(x, y)\right| \lesssim \left| B(x, \sqrt{t}) \right|^{-1} \exp \left( \frac{-c}{t} d(x,y)^2 \right),    
\end{align} 

In practice, for $1 \leq q < \infty$, in the spirit of \cite{Duong-Yan-Hardy-BMO-Interpolation-CPAM-2005, Preisner-Sikora-Yan-Hardy-BMO-harmonic-weights-2019}, we consider the atomic Hardy space ${H^1_{at, q}}$ as well as BMO spaces on the homogeneous space $\left( \mathbb{R}^{n_1 + n_2}, d, | \cdot | \right)$ as follows. 

\begin{definition} \label{def:atoms}
We call a function $a$ to be a $(1, q)$-atom if there exists a ball $B$ such that $supp(a) \subseteq B$, $\|a\|_{L^q} \leq |B|^{-1/{q'}}$ and the cancellation condition $\int_{B} a(x) \ dx = 0$ holds. 
\end{definition}

\begin{definition} \label{def:atomic-Hardy}
The space $H^1_{at, q} \left( \mathbb{R}^{n_1 + n_2} \right)$ consists of all functions $f \in L^1(\mathbb{R}^{n_1 + n_2})$ such that $f=\sum_{k}\lambda_{k} a_{k}$, where $a_{k}$'s are $(1, q)$-atoms and $\lambda_k$'s are complex numbers with $\sum_{k}|\lambda_{k}|<\infty$. Furthermore, we define 
\begin{align} \label{def:atomic-hardy-norm}
\|f\|_{H^1_{at, q}} := \inf \sum_{k}|\lambda_{k}| < \infty, 
\end{align} 
where the infimum is taken over all such representations of $f$.
\end{definition}

Since the heat semigroup $\{e^{-tG_{\varkappa}}\}_{t>0}$ is \textit{conservative}, that is, $\int_{\mathbb{R}^{n_1 + n_2}}e^{-tG_{\varkappa}}(x, y)\ dy=1 $ for all $t>0$ and $x \in \mathbb{R}^{n_1 + n_2} $, invoking Theorem A in \cite{Preisner-Sikora-Yan-Hardy-BMO-harmonic-weights-2019} we have that the spaces $H^1_{G_\varkappa}$ and $H^1_{at,2}\left( \mathbb{R}^{n_1 + n_2} \right)$ coincide with their norms $\|\cdot\|_{H^1_{G_\varkappa}}$ and $\|\cdot\|_{H^1_{at, 2}}$ being equivalent. 

Next, we say that a function $f$ (upto constant differences) belongs to $BMO_{q} \left( \mathbb{R}^{n_1 + n_2} \right)$ if 
\begin{align} \label{def:BMO-norm}
\|f\|_{BMO_q} & := \sup_{B} \inf_{c \in \mathbb{C}} \left( \frac{1}{|B|} \int_{B} \left| f(x) - c \right|^q dx \right)^{1/q} \\ 
\nonumber & \sim \sup_{B} \left( \frac{1}{|B|} \int_{B} \left| f(x) -  \frac{1}{|B|} \int_{B} f \right|^q dx \right)^{1/q} < \infty, 
\end{align} 
where the supremum is taken over all balls $B$. 

Further, following the work of \cite{Duong-Yan-Hardy-BMO-Interpolation-CPAM-2005}, we define the space $BMO_{G_{\varkappa}, q} \left( \mathbb{R}^{n_1 + n_2} \right)$ to be the equivalence classes of functions $f$ (upto constant differences) for which 
\begin{align} \label{def:BMO-norm-approximate-identities}
\| f \|_{BMO_{G_{\varkappa}, q}} & := \sup_{B} \left( \frac{1}{|B|} \int_{B} \left| f(x) - e^{- r_B^2 G_{\varkappa}} f(x) \right|^q dx \right)^{1/q} < \infty, 
\end{align} 
where $r_B$ denotes the radius of the ball $B$, and the supremum is taken over all balls $B$. 

Combining various results from \cite{Duong-Yan-Hardy-BMO-Interpolation-CPAM-2005, Preisner-Sikora-Yan-Hardy-BMO-harmonic-weights-2019}, we can establish the following interpolation result. 

\begin{lemma} \label{lem:Hardy-L2-space-interpolation}
Let $T$ be a bounded sublinear operator from $H^1_{at, 2} \left( \mathbb{R}^{n_1 + n_2} \right)$ to $L^1 \left( \mathbb{R}^{n_1 + n_2} \right)$. If $T$ is also bounded on $L^2 \left( \mathbb{R}^{n_1 + n_2} \right)$, then $T$ is bounded on $L^p \left( \mathbb{R}^{n_1 + n_2} \right)$ for all $1 < p < 2$. 
\end{lemma}
\begin{proof}
By duality, it follows that $T^* : L^\infty \left( \mathbb{R}^{n_1 + n_2} \right) \to \left( H^1_{at, 2} \left( \mathbb{R}^{n_1 + n_2} \right) \right)^*$ is bounded. 

We begin with invoking Theorem B of \cite{Preisner-Sikora-Yan-Hardy-BMO-harmonic-weights-2019} to get that the dual space of $ H^1_{at, 2} \left( \mathbb{R}^{n_1 + n_2} \right)$ is $BMO_{2} \left( \mathbb{R}^{n_1 + n_2} \right)$. Here we have used the fact that the heat semigroup $\{e^{-tG_{\varkappa}}\}_{t>0}$ is \textit{conservative}, that is, $e^{-tG} (1) = \int_{\mathbb{R}^{n_1 + n_2}}e^{-tG_{\varkappa}} (x, y) \, dy=1 $, for all $t>0$, and therefore one can do the analysis of \cite{Preisner-Sikora-Yan-Hardy-BMO-harmonic-weights-2019} with $h = 1$ itself. 

Next, we have by Proposition $2.5$ of \cite{Duong-Yan-Hardy-BMO-Interpolation-CPAM-2005} (which, in fact, is a consequence of Proposition $3.1$ of \cite{Martell-Sharp-Maximal-Studia-Math-2004}), the continuity of the inclusion map $BMO_1 \xhookrightarrow{i} BMO_{G_{\varkappa}, 1}$. Also, the continuity of the inclusion map $BMO_{p_2} \left( \mathbb{R}^{n_1 + n_2} \right) \xhookrightarrow{i} BMO_{p_1} \left( \mathbb{R}^{n_1 + n_2} \right)$, for any $1 \leq p_1 \leq p_2 < \infty$, follows trivially from H\"older's inequality. 

Putting all the above results together, we get the following chain of bounded operators: 
$$ L^\infty \left( \mathbb{R}^{n_1 + n_2} \right) \xhookrightarrow{T^*} BMO_{2} \left( \mathbb{R}^{n_1 + n_2} \right) \xhookrightarrow{i} BMO_{1} \xhookrightarrow{i} BMO_{G_{\varkappa}, 1} \left( \mathbb{R}^{n_1 + n_2} \right).$$

As a consequence, we have that $T^* : L^\infty \left( \mathbb{R}^{n_1 + n_2} \right) \to BMO_{G_{\varkappa}, 1} \left( \mathbb{R}^{n_1 + n_2} \right)$ is bounded. With that we can invoke Theorem $5.2$ of \cite{Duong-Yan-Hardy-BMO-Interpolation-CPAM-2005} to conclude that $T^*$ is bounded on $L^p \left( \mathbb{R}^{n_1 + n_2} \right)$ for all $2 < p < \infty$. By duality, we get that $T$ is bounded on $L^p \left( \mathbb{R}^{n_1 + n_2} \right)$ for all $1 < p < 2$. 

This completes the proof of the Lemma \ref{lem:Hardy-L2-space-interpolation}. 
\end{proof}


\subsection{Hardy--Littlewood--Sobolev inequality} \label{subsec:HLS-inequality}

As earlier, let $\mathcal{M}_{G_\varkappa}$ denote the maximal operator corresponding to the heat semigroup $\{e^{-t G_\varkappa}\}_{t>0}$. It is well-known that $\mathcal{M}_{G_\varkappa}$ is bounded on $L^p(\mathbb{R}^{n_1 + n_2})$ for all $1 < p \leq \infty$. We essentially repeat the steps of the proof of Theorem $3$ of \cite{Applebaum-Banuelos-Probabilistic-hardy-ineq}, with minor modifications, in establishing the following version of the Hardy--Littlewood--Sobolev inequality.

\begin{theorem}[Hardy--Littlewood--Sobolev inequality] \label{theorem:Sobolev-embedding} 
For $1 < p < q < \infty$, $0 < b < Q/p$ and $ \frac{1}{p} - \frac{1}{q} = \frac{b}{Q}$, we have 
\begin{equation} \label{eq:Sobolev-ineq}
\| \left| B(\cdot, 1) \right|^{b/Q} (I + G_\varkappa)^{-b/2} f \|_q \lesssim_{b,p,q} \| f \|_p.
\end{equation}
\end{theorem} 
\begin{proof} Write $T_t=e^{-t G_\varkappa}$ for $t>0$, and denote by $p_t$ the kernel of the operator $T_t$. Fix $p$ and $q$ satisfying the conditions of the statement of the theorem. Using H\"older's inequality we get 
\begin{align*}
\left| T_t f (x) \right| & = \left| \int_{\mathbb{R}^{n_1 + n_2}}  f(y) p_t (x,y) \, dy \right| \\ 
& \leq \left( \int_{\mathbb{R}^{n_1 + n_2}} \left| f(y) \right|^p \left| p_t (x,y) \right| \, dy \right)^{1/p} \left( \int_{\mathbb{R}^{n_1 + n_2}} \left| p_t (x,y) \right| \, dy \right)^{1/{p^{\prime}}} \\ 
& \lesssim \left| B(x, \sqrt{t}) \right|^{-1/p} \|f\|_p. 
\end{align*}

Next, recall that 
$$ (I+G_\varkappa)^{-b/2} f(x) = \frac{1}{\Gamma(b/2)} \int_0^\infty e^{-t} t^{\frac{b}{2} - 1} T_t f(x) \, dt, $$
with the integral converging absolutely, a fact that is well known, and the same can also be verified from the analysis that we are now going to perform. 

We break the integration in $t$-variable in two parts, on $(0 , {\delta_0}]$ and $({\delta_0}, \infty)$ where the exact value of ${\delta_0} > 0$ would be prescribed later. Let us write 
\begin{align*}
J_b f(x) = \frac{1}{\Gamma(b/2)} \int_0^{\delta_0} e^{-t} t^{\frac{b}{2} - 1} T_t f(x) \, dt, \quad \text{and} \quad K_b f(x) = \frac{1}{\Gamma(b/2)} \int_{\delta_0}^\infty e^{-t} t^{\frac{b}{2} - 1} T_t f(x) \, dt, 
\end{align*}
so that $ (I+G_\varkappa)^{-b/2} f(x) = J_b f(x) + K_b f(x) $.

Note first that 
$$ \left| J_b f(x) \right| \leq \frac{1}{\Gamma(b/2)} \mathcal{M}_{G_\varkappa}f(x) \int_0^{\delta_0} t^{\frac{b}{2} - 1} \, dt \, = C_b \mathcal{M}_{G_\varkappa}f(x) {\delta_0}^{\frac{b}{2}}. $$ 

On the other hand, making use of the following trivial estimate 
\begin{equation*} 
e^{-t} \frac{\left| B(x, 1) \right|^{1/p}}{\left| B(x, \sqrt{t}) \right|^{1/p}} \lesssim e^{-t} \left( 1 + \frac{1}{\sqrt{t}} \right)^{Q/p} \lesssim_p \left( \frac{1}{\sqrt{t}} \right)^{Q/p}, 
\end{equation*} 
where the first inequality follows from the doubling measure property, we get  
\begin{align*}
\left| K_b f(x) \right| & \leq \frac{1}{\Gamma(a/2)} \int_{\delta_0}^\infty e^{-t} t^{\frac{b}{2} - 1} \left| T_t f(x) \right| \, dt \\ 
& \lesssim_b \int_{\delta_0}^\infty e^{-t} t^{\frac{b}{2} - 1} \left| B(x, \sqrt{t}) \right|^{-1/p} \|f\|_p \, dt \\ 
& = \left| B(x, 1) \right|^{-1/p} \|f\|_p \int_{\delta_0}^\infty t^{\frac{b}{2} - 1} \left( e^{-t} \frac{\left| B(x, 1) \right|^{1/p}}{\left| B(x, \sqrt{t}) \right|^{1/p}} \right) \, dt \\ 
& \lesssim_p \left| B(x, 1) \right|^{-1/p} \|f\|_p \int_{\delta_0}^\infty t^{\frac{b}{2} - \frac{Q}{2p} - 1} \, dt \\ 
& \sim_{b,p,q} \left| B(x, 1) \right|^{-1/p} \|f\|_p \, {\delta_0}^{\frac{b}{2} - \frac{Q}{2p}}. 
\end{align*}

Combining the above estimates of $ \left| J_b f(x) \right| $ and $ \left| K_b f(x) \right| $, we get 
$$ \left| (I+G_\varkappa)^{-b/2} f(x) \right| \lesssim_{b,p,q} \mathcal{M}_{G_\varkappa}f(x) {\delta_0}^{\frac{b}{2}} + \left| B(x, 1) \right|^{-1/p} \|f\|_p \, {\delta_0}^{\frac{b}{2} - \frac{Q}{2p}}. $$

Now, in order to optimize in ${\delta_0} > 0$, we choose ${\delta_0} = \left( \frac{\left| B(x, 1) \right|^{-1/p} \|f\|_p}{\mathcal{M}_{G_\varkappa}f(x)} \right)^{2p/Q},$ and then 
\begin{align*}
\left| (I+G_\varkappa)^{-b/2} f(x) \right| & \lesssim_{b,p,q} \left| B(x, 1) \right|^{-b/Q} \|f\|_p^{bp/Q} \left( \mathcal{M}_{G_\varkappa}f(x) \right)^{1- bp/Q} \\ 
& = \left| B(x, 1) \right|^{-b/Q} \|f\|_p^{bp/Q} \left(\mathcal{M}_{G_\varkappa}f(x) \right)^{p/q}, 
\end{align*}

implying that  
\begin{align*}
\int_{\mathbb{R}^{n_1 + n_2}} \left( \left| B(x, 1) \right|^{b/Q} \left| (I+G_\varkappa)^{-b/2} f(x) \right| \right)^q dx 
& \lesssim_{b,p,q} \|f\|_p^{bpq/Q} \int_{\mathbb{R}^{n_1 + n_2}} \left( \mathcal{M}_{G_\varkappa}f(x) \right)^{p} dx \\ 
& \lesssim_{p, \varkappa} \|f\|_p^{bpq/Q} \|f\|_p^p \\ 
& = \|f\|_p^q. 
\end{align*}

This completes the proof of Theorem \ref{theorem:Sobolev-embedding}. 
\end{proof}


\subsection{Weighted Calder\'on-Vaillancourt results} \label{subsec:weighted-Calderon-Villaincourt}
In this section, we shall prove a result which can be thought of as a weighted analogue of Calder\'on-Vaillancourt type theorem for Grushin pseudo-multipliers.

In order to do so, we first show that the ideas of \cite{Bagchi-Garg-1} allow one to have a Calder\'on-Vaillancourt type theorem even in the presence of Hermite-type shifts. More precisely, given $m \in \mathscr{S}^0_{\rho, \delta}(G_{\varkappa})$ and $\vec{c} \in \mathbb{R}^{n_1}$, let us consider symbol functions $\widetilde{m} : \mathbb{R}^{n_1 + n_2} \times (\mathbb{R}_+)^{n_1} \to \mathbb{C}$ and $M_{\vec{c}} : \mathbb{R}^{n_1 + n_2} \times (\mathbb{R}_+)^{n_1} \times (\mathbb{R}^{n_2} \setminus \{0\}) \to \mathbb{C}$ defined by 
\begin{align} \label{define:shifts-CV-multi-symb}
\widetilde{m} (x, \tau) = m (x, |\tau|_1) \quad \text{and} \quad M_{\vec{c}} (x, \tau, \kappa) = \widetilde{m} \left( x, \tau + |\kappa| \vec{c} \right). 
\end{align}

We have the following $L^2$-boundedness result. 

\begin{lemma} \label{lem:shifted-CV} 
Given a compact set $\mathcal{K} \subset \mathbb{R}^{n_1}$, there exists a constant $C = C_{\mathcal{K}}$ such that the following holds true. For any $m \in \mathscr{S}^0_{\rho, \delta}(G_{\varkappa})$, define $\widetilde{m} (x, \tau)$ and $M_{\vec{c}}$ as in \eqref{define:shifts-CV-multi-symb}. Then, 
\begin{align} \label{ineq:shifts-and-CV}
\| M_{\vec{c}} (x, \boldsymbol{L}, \boldsymbol{U}) \|_{op} \leq C_{\mathcal{K}} \| m \|_{\mathscr{S}^0_{\rho, \delta}}, 
\end{align}
for all $\vec{c} \in \mathcal{K}$. 
\end{lemma} 
\begin{proof}
Choose and fix $\phi \in C_c^\infty (\mathbb{R}^{n_2})$ such that $\phi (\kappa) = 1$ for all $|\kappa| \leq 1/2$ and $\phi (\kappa) = 0$ for all $|\kappa| \geq 1$, and decompose $M_{\vec{c}} = M_{1, \vec{c}} + M_{2, \vec{c}}$, where 
$$ M_{1, \vec{c}} (x, \tau, \kappa) = M_{\vec{c}} (x, \tau, \kappa) \, \phi (\kappa) \quad \text{and} \quad M_{2, \vec{c}} (x, \tau, \kappa) = M_{\vec{c}} (x, \tau, \kappa) \, ( 1 - \phi (\kappa) ). $$

Let us first analyse the operator $M_{2, \vec{c}} (x, \boldsymbol{L}, \boldsymbol{U})$ corresponding to the symbol function $M_{2, \vec{c}} (x, \tau, \kappa)$. It follows directly from the definition of $M_{2, \vec{c}}$ and the support condition on $\phi$ that $M_{2, \vec{c}}$
satisfies the symbol condition 
\begin{align} \label{ineq:joint-symb-lesser-assumption-2}
\left| X^\Gamma \partial_{\tau}^{\theta} \partial_{\kappa}^{\beta} M_{2, \vec{c}} (x, \tau, \kappa) \right| \leq_{\Gamma, \theta, \beta, \mathcal{K}} (1 + |\tau|)^{- (1 + \rho) \frac{|\theta|}{2} + \delta \frac{|\Gamma|}{2}}, 
\end{align}
and the cancellation condition \eqref{def:grushin-symb-vanishing-0-condition} for all $\Gamma \in \mathbb{N}^{n_1 + n_1 n_2}$, $\theta \in \mathbb{N}^{n_1}$ and $\beta \in \mathbb{N}^{n_2}$. 

It then follows from Remark $1.8$ of \cite{Bagchi-Garg-1} that $M_{\vec{c}}$ satisfies \eqref{ineq:shifts-and-CV}, that is, for all $\vec{c} \in \mathcal{K}$, 
\begin{align} \label{ineq:shifts-and-CV-part1}
\| M_{2, \vec{c}} (x, \boldsymbol{L}, \boldsymbol{U}) \|_{op} \leq C_{\mathcal{K}} \| m \|_{\mathscr{S}^0_{\rho, \delta}}. 
\end{align}

Next, we analyse the operator $M_{1, \vec{c}} (x, \boldsymbol{L}, \boldsymbol{U})$. Using the Taylor series expansion (see also the discussion around $(9.7)$ of \cite{Bagchi-Garg-1}), for every $N \in \mathbb{N}$ we can write 
\begin{align*}
M_{1, \vec{c}} (x, \tau, \kappa) & = \phi (\kappa) \, \widetilde{m} \left( x, \tau + |\kappa| \vec{c} \right) \\ 
& = \phi (\kappa) \sum_{|\alpha| \leq N} \frac{\vec{c}^{\alpha}}{ \, \alpha!} \, |\kappa|^{|\alpha|} \partial^{\alpha}_{\tau} \widetilde{m} \left( x, \tau \right) \\ 
& \quad + \phi (\kappa) \sum_{|\alpha| = N+1} \frac{N+1}{\alpha!} \int_0^1 (1-t)^{N} |\kappa|^{|\alpha|} \, \partial^{\alpha}_{\tau}  \widetilde{m} \left( x, \tau + t |\kappa| \vec{c} \right) \, dt \\ 
& = M_{1, \vec{c}}^{1, N} (x, \tau, \kappa) + M_{1, \vec{c}}^{2, N} (x, \tau, \kappa). 
\end{align*}

We shall fix $N$ in just a while. But, let us first show that $M_{1, \vec{c}}^{1, N} (x, \tau, \kappa)$ corresponds to an $L^2$-bounded operator with operator norm satisfying \eqref{ineq:shifts-and-CV}. For this, 
note that since $(\tau, \kappa) \mapsto  |\kappa|^{|\alpha|} \phi (\kappa)$ is bounded, it follows from Plancherel's theorem that it corresponds to an $L^2$-bounded operator, with operator bound depending only on $N$ and the function $\phi$. On the other hand, we have that each of $\partial^{\alpha}_{\tau} \widetilde{m} \left( x, \tau \right) \in \mathscr{S}^0_{\rho, \delta}(\boldsymbol{L}).$ It then follows from the work of \cite{Bagchi-Garg-1} that $\partial^{\alpha}_{\tau} \widetilde{m} \left( x, \boldsymbol{L} \right)$, and hence $M_{1, \vec{c}}^{1, N} (x, \boldsymbol{L}, \boldsymbol{U})$, satisfies \eqref{ineq:shifts-and-CV} with the implicit constant may also be depending on $N$ and the function $\phi$. 

Finally, in order to argue for $M_{1, \vec{c}}^{2, N} (x, \tau, \kappa)$, let us take one piece of the finite sum in its definition, say, $\phi (\kappa) |\kappa|^{|\alpha|} \partial^{\alpha}_{\tau} \widetilde{m} \left( x, \tau + |\kappa| \vec{c} \right)$ with $|\alpha| = N+1$. Taking $N$ to be sufficiently large, one can ensure that not only $(x, \tau, \kappa) \mapsto \phi (\kappa) |\kappa|^{|\alpha|} \partial^{\alpha}_{\tau} \widetilde{m} \left( x, \tau + t |\kappa| \vec{c} \right)$ satisfies the symbol condition \eqref{ineq:joint-symb-lesser-assumption-2}, but it also satisfies the cancellation condition \eqref{def:grushin-symb-vanishing-0-condition} for all $|\beta| \leq 4 \left( \lfloor\frac{Q}{4}\rfloor + 1 \right)$. With these conditions at hand, the result follows from Remark $1.8$ of \cite{Bagchi-Garg-1}. 
\end{proof}

We have the following weighted analogue of Calder\'on-Vaillancourt type theorem. 

\begin{proposition}
\label{prop:weighted-CV} 
Let $m \in \mathscr{S}^{0}_{\rho, \delta}(G)$ with $\delta < \rho$. Then, for every $\mathfrak{b} \geq 0$, we have 
\begin{align*} 
\int_{\mathbb{R}^{n_1 + n_2}} |m(x, G)f(x)|^{2} |B(x, 1)|^{\mathfrak{b}} \; dx \lesssim_{\mathfrak{b}} \| m \|^2_{\mathscr{S}^0_{\rho,\delta}} \int_{\mathbb{R}^{n_1 + n_2}}  |f(x)|^{2} |B(x, 1)|^{\mathfrak{b}} \; dx.
\end{align*}
\end{proposition}

\begin{proof}
Note that 
\begin{align*}
& \int_{\mathbb{R}^{n_1 + n_2}} |m(x, G) f(x)|^{2} |B(x, 1)|^{\mathfrak{b}} \, dx \\
& \lesssim \int_{|x'| \leq 1} |m(x, G) f(x)|^{2}  \, dx + \int_{|x'|> 1} |m(x, G)f(x)|^{2} |x'|^{\mathfrak{b} \varkappa n_2} \, dx \\
& := I + II.
\end{align*}

For term $I$, since $m \in \mathscr{S}^{0}_{\rho, \delta}(G)$, we can invoke Theorem \ref{thm:L2-joint-calc} to conclude that 
\begin{align*}
I \leq \int_{\mathbb{R}^{n_1 + n_2}} |m(x, G)f(x)|^{2} \, dx & \lesssim  \| m \|^2_{\mathscr{S}^0_{\rho, \delta}} \int_{\mathbb{R}^{n_1 + n_2}} |f(x)|^{2} \, dx \\ 
& \lesssim_{\mathfrak{b}} \| m \|^2_{\mathscr{S}^0_{\rho, \delta}} \int_{\mathbb{R}^{n_1 + n_2}} |f(x)|^{2} |B(x,1)|^{\mathfrak{b}} \, dx.
\end{align*}

Now, for term $II$, if we could prove that 
\begin{align}\label{Global-weighted-L2}
\int_{|x'|> 1} |m(x, G)f(x)|^{2} |x'|^{4k} \, dx \lesssim_{k} \| m \|^2_{\mathscr{S}^0_{\rho, \delta}} \int_{\mathbb{R}^{n_1 + n_2}} |f(x)|^{2} (1 + |x'|)^{4k} \, dx \end{align}
for all $k \in \mathbb{N}$, then we will have by interpolation that \eqref{Global-weighted-L2} holds true with $k$ replaced by arbitrary $b>0$. In particular, that would imply  
\begin{align*}
\int_{|x'|> 1} |m(x, G) f(x)|^{2} |x'|^{\mathfrak{b} \varkappa n_2} \, dx \lesssim_{\mathfrak{b}} \| m \|^2_{\mathscr{S}^0_{\rho, \delta}} \int_{\mathbb{R}^{n_1 + n_2}}  |f(x)|^{2} |B(x, 1)|^{\mathfrak{b}} \, dx .
\end{align*}

Now, to establish \eqref{Global-weighted-L2}, it suffices to prove that 
\begin{align} \label{Global-L2-revised}
\int_{|x'|> 1} |x'^{\alpha} m(x, G) f(x)|^{2} \, dx \lesssim_{k} \| m \|^2_{\mathscr{S}^0_{\rho, \delta}} \int_{\mathbb{R}^{n_1 + n_2}} |f(x)|^{2} (1 + |x'|)^{4k} \, dx \, ,
\end{align}
where  $|\alpha|=2k$.

With $\psi_j$'s as in \eqref{decompose-dyadic-spectral-symbol-frequency}, let us decompose $m(x, G) = \sum_{j\geq 0} m_j(x, G)$, where $m_j(x, \eta) = m(x, \eta) \psi_j (\eta).$ Also, as earlier, let us write $\widetilde{m}_j (x, \tau) = m_j (x, |\tau|_1)$. Then, the kernel of the operator $m_j(x, G)$ is given by 
\begin{align*}
m_j(x, G)(x,y) = \widetilde{m}_j(x, \boldsymbol{L})(x,y) = \int_{\mathbb{R}^{n_2}} \sum_{\mu} \widetilde{m}_j(x, (2 \mu + \tilde{1})|\lambda|) \Phi_{\mu}^{\lambda}(x') \Phi_{\mu}^{\lambda}(y') e^{-i\lambda \cdot (x''-y'')} \, d\lambda, 
\end{align*}
and writing $x'=(x'-y')+y'$, it suffices to prove that 
\begin{align} \label{Global-L2-revised-2}
& \sup_{N \in \mathbb{N}} \int_{|x'|> 1}  \left| \int_{\mathbb{R}^{n_1 + n_2}} y'^{(\alpha - \gamma)} (x'-y')^{\gamma} \sum_{j = 0}^N \widetilde{m}_j(x, \boldsymbol{L}) (x,y) f(y) \, dy \right|^{2} \, dx \\ 
\nonumber & \quad \lesssim_{k} \| m \|^2_{\mathscr{S}^0_{\rho, \delta}} \int_{\mathbb{R}^{n_1 + n_2}} |f(x)|^{2} (1 + |x'|)^{4k} \, dx.
\end{align}

So, in the rest of the proof, we shall establish the claimed estimate \eqref{Global-L2-revised-2}. Making use of Lemma $4.5$ of \cite{Bagchi-Garg-1}, we can express the integral kernel $y'^{(\alpha - \gamma)}(x'-y')^{\gamma} \widetilde{m}_j(x, \boldsymbol{L})(x, y)$ as a finite linear combination of terms of the form
\begin{align*} 
& y'^{(\alpha - \gamma)} \int_{\mathbb{R}^{n_2}}\int_{[0,1]^{|\gamma|}} \sum_{\mu} C_{\mu, \vec{c}} \left(\tau^{\frac{1}{2} \gamma_2} \partial_{\tau}^{\gamma_1} \sum_{j = 0}^N \widetilde{m}_j \right) (x,(2\mu  + \tilde{1} + \vec{c}(s))|\lambda|) \Phi_{\mu}^{\lambda}(x^{\prime}) \Phi_{\mu + \tilde{\mu}}^{\lambda}(y^{\prime}) \\ 
& \nonumber \quad \quad \quad \quad \quad \quad e^{-i \lambda \cdot (x''-y'')} \, ds \, d\lambda ,  
\end{align*} 
where $|\gamma_2| \leq |\gamma_1| \leq |\gamma|$, $|\gamma_1| - \frac{1}{2} |\gamma_2| = \frac{|\gamma|}{2}$, $|\tilde{\mu}| \leq |\gamma|$,  $C_{\mu, \vec{c}}$ is a bounded function of $\mu$ and $\vec{c}$, and $|\vec{c}(s)| \leq 4 |\gamma|$. 

Therefore, it boils down to estimating the norm of the operator having integral kernel 
$$ \int_{ \mathbb{R}^{n_2}} \sum_{\mu} C_{\mu, \vec{c}} \left( \tau^{ \frac{1}{2} \gamma_2} \partial_{\tau}^{ \gamma_1 } \sum_{j = 0}^N \widetilde{m}_j \right) (x, ( 2\mu  + \tilde{1} + \vec{c}(s)) |\lambda|) \Phi_{\mu}^{\lambda} (x^{\prime}) \Phi_{\mu + \tilde{\mu}}^{\lambda}(y^{\prime}) e^{-i \lambda \cdot (x''-y'')} \, d\lambda. $$ 

But, if we define the operator $\mathcal{T}$ on $L^2(\mathbb{R}^{n_1 + n_2})$ by $ \left( (\mathcal{T} g)^{\lambda}, \Phi^{\lambda}_{\mu} \right) = C_{\mu, \vec{c}} \, \left( g^{\lambda}, \Phi^{\lambda}_{\mu + \tilde{\mu}} \right)$, then it follows from the Plancherel theorem that the operator $\mathcal{T}$ is bounded on $L^2(\mathbb{R}^{n_1 + n_2})$ with the operator norm depending on $\gamma$ (and hence on $k$). Therefore, we are led to establishing $L^2$-boundedness of operators with kernels 
$$\int_{ \mathbb{R}^{n_2} } \sum_{\mu} \left( \tau^{\frac{1}{2} \gamma_2} \partial_{\tau}^{ \gamma_1 } \sum_{j = 0}^N \widetilde{m}_j \right) (x, ( 2\mu  + \tilde{1} + \vec{c}(s)) |\lambda|) \Phi_{\mu}^{\lambda} (x^{\prime}) \Phi_{\mu}^{\lambda}(y^{\prime}) e^{-i \lambda \cdot (x''-y'')} \, d\lambda. $$ 
But, the same holds true by Lemma \ref{lem:shifted-CV}, duly keeping in mind that 
$$\left\| \tau^{\frac{1}{2} \gamma_2} \partial_{\tau}^{\gamma_1} \sum_{j = 0}^N \widetilde{m}_j \right\|_{\mathscr{S}^0_{\rho, \delta}} \lesssim_{\gamma} \| m \|_{\mathscr{S}^0_{\rho, \delta}}.$$

This completes the proof of Proposition \ref{prop:weighted-CV}. 
\end{proof}

\begin{remark} \label{rem:shifted-and-weighted-CV}
Analogues of Lemma \ref{lem:shifted-CV} and Proposition \ref{prop:weighted-CV} also hold true for the joint functional calculus, for $m \in \mathscr{S}^{0}_{\rho, \delta}(\boldsymbol{L}, \boldsymbol{U})$, in any of the following situations: 
\begin{itemize}
\item if $\delta < \rho$, or 

\item if $\delta = \rho$ and $m$ satisfies the cancellation condition \eqref{def:grushin-symb-vanishing-0-condition}
for all $\beta \in \mathbb{N}^{n_2}$ with $|\beta| \leq 4 \left( \lfloor\frac{Q}{4}\rfloor + 1 \right)$.
\end{itemize}
\end{remark}


\section{Kernel estimates}
\label{sec:Kernel-estimates}
In this section we shall prove kernel estimates for Grushin pseudo-multiplier operators which are among the essential tools to prove our main results in the subsequent sections. These kernel estimates follow from the weighted Plancherel estimates proved in \cite{BBGG-1} and \cite{Bagchi-Garg-1}. We recall them here for our ready reference. As usual, let $K_{m(x, G_{\varkappa})}$ and $K_{m(x,\boldsymbol{L}, \boldsymbol{U} )}$ denote the kernels of pseudo-multipliers $m(x, G_\varkappa)$ and $m(x,\boldsymbol{L}, \boldsymbol{U} )$ respectively. 


\subsection{For pseudo-multipliers associated with \texorpdfstring{$G_{\varkappa}$}{}} \label{subsec:kernel-est-Grushin}

We begin with recalling the following estimates from \cite{BBGG-1}. 
\begin{lemma}[Corollaries 4.4, 4.5, \cite{BBGG-1}] \label{lem:grad-weighted-L-p-estimate}
For $2 \leq p \leq \infty$, $\mathfrak{r} \geq 0$ and $\epsilon > 0$, we have 
\begin{align*} 
& |B(x,R^{-1})|^{1/2} \left\| |B(\cdot, R^{-1})|^{1/2 - 1/p} \, (1+Rd(x,\cdot))^{\mathfrak{r}} X_{x}^{\Gamma} K_{m(x, G_{\varkappa})} (x,\cdot) \right\|_p \\ 
& \quad \lesssim_{\Gamma, p, \mathfrak{r} , \epsilon} \sup_{x_{0}} \sum_{\Gamma_1 + \Gamma_2 = \Gamma} R^{|\Gamma_1|} \|X^{\Gamma_2}_{x} m(x_0,R^2 \cdot) \|_{W_{\mathfrak{r} + \epsilon}^{\infty}}, \\ 
& |B(x,R^{-1})|^{1/2} \left\| |B(\cdot, R^{-1})|^{1/2 - 1/p} \, (1+Rd(x,\cdot))^{\mathfrak{r}} X_{y}^{\Gamma} K_{m(x, G_{\varkappa})} (x,\cdot) \right\|_p \\ 
& \quad \lesssim_{\Gamma, p, \mathfrak{r} , \epsilon} \sup_{x_{0}} R^{|\Gamma|} \| m(x_0, R^2\cdot) \|_{W_{\mathfrak{r} + \epsilon}^{\infty}}, 
\end{align*} 
for all $\Gamma \in \mathbb{N}^{n_0}$ and for every bounded Borel function $m : \mathbb{R}^{n_1 + n_2} \times \mathbb{R} \rightarrow \mathbb{C}$ whose support in the last variable is in $[0, R^2]$ for any $R>0$. 
\end{lemma} 

In order to apply Lemma \ref{lem:grad-weighted-L-p-estimate}, we decompose the operator $m(x,G_{\varkappa})$ with the help of $\psi_j$'s (given by \eqref{decompose-dyadic-spectral-symbol-frequency}) as follows: 
$$m(x,G_{\varkappa})=\sum_{j\geq 0} m_j(x,G_{\varkappa}), ~ where~m_j(x,G_{\varkappa})=m(x,G_{\varkappa})\psi_{j}(G_{\varkappa}).$$
If $K_j(x,y)$ denotes the integral kernel of $m_j(x, G_{\varkappa})$, then we have the following estimates.

\begin{corollary} \label{cor:wighted-plancherel-estimates-for-symbol-kernel}
Let $m\in \mathscr{S}^{\sigma}_{1-a, \delta}(G_{\varkappa})$. Then, for all $j, \mathfrak{r} \geq 0$ and $\epsilon > 0$, we have 
\begin{align} 
\sup_{x \in \mathbb{R}^{n_1 + n_2}} |B(x, 2^{-j/2})| \int_{\mathbb{R}^{n_1 + n_2}} d(x, y)^{2 \mathfrak{r}} |K_{j}(x, y)|^2 \, dy & \lesssim_{ \mathfrak{r}, \epsilon} 2^{j \sigma} 2^{-j\mathfrak{r}(1-a)} 2^{ja \epsilon},  \label{cond:General-hypo} \\ 
\sup_{x \in \mathbb{R}^{n_1 + n_2}} |B(x, 2^{-j/2})| \int_{\mathbb{R}^{n_1 + n_2}} d(x, y)^{2 \mathfrak{r}} |X_{x} K_{j}(x, y)|^2 \, dy & \lesssim_{ \mathfrak{r}, \epsilon} 2^{j \sigma} 2^{-j \mathfrak{r}(1-a)} 2^{ja \epsilon} 2^{j} \label{cond:General-hypo-grad} \\
\sup_{x,y \in \mathbb{R}^{n_1 + n_2}} |B(x, 2^{-j/2})|^{1/2} |B(y, 2^{-j/2})|^{1/2} d(x, y)^{\mathfrak{r}} |K_{j} (x, y)| & \lesssim_{ \mathfrak{r}, \epsilon} 2^{j \sigma/2} 2^{-j \mathfrak{r}(1-a)/2} 2^{ja \epsilon/2}, \label{cond:General-hypo-sup} \\ 
\sup_{x,y \in \mathbb{R}^{n_1 + n_2}} |B(x, 2^{-j/2})|^{1/2} |B(y, 2^{-j/2})|^{1/2} d(x, y)^{\mathfrak{r}} |X_{x} K_{j}(x, y)| & \lesssim_{ \mathfrak{r}, \epsilon} 2^{j \sigma/2} 2^{-j \mathfrak{r}(1-a)/2} 2^{j/2} 2^{ja \epsilon/2}, \label{cond:General-hypo-grad-sup} \\ 
\sup_{x,y \in \mathbb{R}^{n_1 + n_2}} |B(x, 2^{-j/2})|^{1/2} |B(y, 2^{-j/2})|^{1/2} d(x, y)^{\mathfrak{r}} |X_{y} K_{j}(x, y)| & \lesssim_{ \mathfrak{r}, \epsilon} 2^{j/2} 2^{-j \mathfrak{r}(1-a)/2} 2^{j/2} 2^{ja \epsilon/2}. \label{cond:General-hypo-y-grad-sup} 
\end{align}
\end{corollary}

\begin{proof}
Let us show how to estimate \eqref{cond:General-hypo} by Lemma \ref{lem:grad-weighted-L-p-estimate}. To see this, let us apply Lemma \ref{lem:grad-weighted-L-p-estimate} with $p=2$, $R=2^{j/2}$ and $\Gamma = 0$, to get that 
\begin{align*}
\int_{\mathbb{R}^{n_1 + n_2}} d(x,y)^{2\mathfrak{r}} |K_{j}(x,y)|^2 \, dy & \lesssim_{r, \epsilon} |B(x, 2^{-j/2})|^{-1} 2^{-j \mathfrak{r}} \sup_{x_{0}} \left\| m_{j} (x_0, 2^{j} \cdot ) \right\|^2_{W_{\mathfrak{r}  + \epsilon}^{\infty}}\\
& \lesssim_{\mathfrak{r}, \epsilon} 2^{j \sigma} 2^{-j\mathfrak{r}(1-a)} 2^{ja \epsilon}, 
\end{align*}
where the last inequality follows from the fact that $m \in \mathscr{S}^{\sigma}_{1-a, \delta}(G_\varkappa)$. This completes the proof of estimate \eqref{cond:General-hypo}. 

Similarly, estimate \eqref{cond:General-hypo-grad} follows from the first estimate of Lemma \ref{lem:grad-weighted-L-p-estimate} with $p = 2$, $R = 2^{j/2}$ and $|\Gamma| = 1$, estimate \eqref{cond:General-hypo-sup} follows from Lemma \ref{lem:grad-weighted-L-p-estimate} with 
$p = \infty$, $R = 2^{j/2}$ and $|\Gamma| = 0$, estimate \eqref{cond:General-hypo-grad-sup} follows from the first inequality of Lemma \ref{lem:grad-weighted-L-p-estimate} with 
$p = \infty$, $R = 2^{j/2}$ and $|\Gamma| = 1$, and estimate \eqref{cond:General-hypo-y-grad-sup} follows from the second inequality of Lemma \ref{lem:grad-weighted-L-p-estimate} with $p = \infty$, $R=2^{j/2}$ and $|\Gamma| = 1$.
\end{proof}


\subsection{For pseudo-multipliers associated with joint functional calculus}  \label{subsec:kernel-est-joint-functional}

In this subsection we shall prove weighted Plancherel estimates for the joint functional calculus of $\boldsymbol{L}$ and $\boldsymbol{U}$. In fact, assuming the extra condition \eqref{def:grushin-symb-vanishing-0-condition}, we get conditions of types $\eqref{cond:General-hypo}$ to $\eqref{cond:General-hypo-y-grad-sup}$ without the extra growth of $2^{j a \epsilon}$. More precisely, 
\begin{lemma}[Theorem 1.12, \cite{Bagchi-Garg-1}] \label{lem:weighted-Plancherel-Lp} 
Let $2 \leq p \leq \infty$ and $r_0 \in \mathbb{N}$. For all $R > 0$ and $0 \leq \mathfrak{r} \leq 4 \lfloor{ \frac{r_0}{4} \rfloor}$ we have 
\begin{align*} 
\left| B(x, R^{-1}) \right|^{1/2} \left\| \left| B(\cdot, R^{-1}) \right|^{\frac{1}{2} - \frac{1}{p}} \left( 1 + R d(x, \cdot) \right)^{\mathfrak{r}} K_{m( \boldsymbol{L}, \boldsymbol{U} )} (x, \cdot) \right\|_p \lesssim_{p,\mathfrak{r}} \left\| m \left(R^2 \cdot \right) \right\|_{ W^{\mathfrak{r} }_{\infty}}.
\end{align*} 
for any bounded Borel function $m : \mathbb{R}^{n_1} \times \mathbb{R}^{n_2} \to \mathbb{C}$ such that $\supp m \subseteq [-R^2,  R^2]^{n_1 + n_2},$ and $\lim_{\kappa \to 0} \partial_\kappa^{\beta} m(\tau, \kappa) = 0$ for all $|\beta| \leq r_0.$ 
\end{lemma}

One can use the methodology of the proof of Corollary 4.2 of \cite{BBGG-1} to prove a result analogous to Lemma \ref{lem:weighted-Plancherel-Lp} for pseudo-multipliers $m(x, \boldsymbol{L}, \boldsymbol{U})$. 
\begin{corollary} \label{cor:pseudo-weighted-Plancherel-Lp}
Let $2 \leq p \leq \infty$ and $r_0 \in \mathbb{N}$. For all $R > 0$ and $0 \leq \mathfrak{r} \leq 4 \lfloor{ \frac{r_0}{4} \rfloor}$ we have 
\begin{align*} 
\left| B(x, R^{-1}) \right|^{1/2} \left\| \left| B(\cdot, R^{-1}) \right|^{\frac{1}{2} - \frac{1}{p}} \left( 1 + R d(x, \cdot) \right)^{\mathfrak{r}} K_{m(x,\boldsymbol{L}, \boldsymbol{U} )} (x, \cdot) \right\|_p \lesssim_{p,\mathfrak{r}}  \sup_{x_0} \left\| m \left(x_0, R^2 \cdot \right) \right\|_{W^{\mathfrak{r}}_{\infty}}.
\end{align*} 
for any bounded Borel function $m : \mathbb{R}^{n_1 + n_2} \times \mathbb{R}^{n_1} \times \mathbb{R}^{n_2} \to \mathbb{C}$ whose support in the last two variables is in $[-R^2, R^2]^{n_1 + n_2}$  and $\lim_{\kappa \to 0} \partial_\kappa^{\beta} m( x, \tau, \kappa) = 0$ for all $|\beta| \leq r_0.$ 
\end{corollary}

One can then modify the proof of Lemma \ref{lem:weighted-Plancherel-Lp} and prove the following weighted Plancherel estimate for the gradient of the integral kernel. 
\begin{lemma} \label{lem:weighted-Plancherel-Lp-gradient} 
Let $2 \leq p \leq \infty$ and $r_0 \in \mathbb{N}$. For all $R > 0$ and $0 \leq \mathfrak{r} \leq 4 \lfloor{ \frac{r_0}{4} \rfloor}$ we have 
\begin{align*} 
\left| B(x, R^{-1}) \right|^{1/2} \left\| \left| B(\cdot, R^{-1}) \right|^{\frac{1}{2} - \frac{1}{p}} \left( 1 + R d(x, \cdot) \right)^{\mathfrak{r}} X_x K_{m( \boldsymbol{L}, \boldsymbol{U} )} (x, \cdot) \right\|_p \lesssim_{p,\mathfrak{r}} R \left\| m \left(R^2 \cdot \right) \right\|_{W^{\mathfrak{r}}_{\infty}}. 
\end{align*} 
for any bounded Borel function $m : \mathbb{R}^{n_1} \times \mathbb{R}^{n_2} \to \mathbb{C}$ such that $\supp m \subseteq [-R^2,  R^2]^{n_1 + n_2},$ and $\lim_{\kappa \to 0} \partial_\kappa^{\beta} m(\tau, \kappa) = 0$ for all $|\beta| \leq r_0.$ 
\end{lemma}

\begin{corollary}
\label{cor:pseudo-weighted-Plancherel-Lp-gradient} 
Let $2 \leq p \leq \infty$ and $r_0 \in \mathbb{N}$. For all $R > 0$ and $0 \leq \mathfrak{r} \leq 4 \lfloor{ \frac{r_0}{4} \rfloor}$ we have 
\begin{align*} 
& \left| B(x, R^{-1}) \right|^{1/2} \left\| \left| B(\cdot, R^{-1}) \right|^{\frac{1}{2} - \frac{1}{p}} \left( 1 + R d(x, \cdot) \right)^{\mathfrak{r}} X_x K_{m( x,\boldsymbol{L}, \boldsymbol{U} )} (x, \cdot) \right\|_p \\ 
\nonumber & \quad \lesssim_{p,\mathfrak{r}} \sup_{x_0} \left( R  \left\| m \left( x_0, R^2 \cdot \right) \right\|_{W^{\mathfrak{r}}_{\infty}} + \|X_{x} m(x_0,R^2 \cdot) \|_{W_{\mathfrak{r}}^{\infty}} \right). 
\end{align*} 
for any bounded  Borel function $m : \mathbb{R}^{n_1 + n_2} \times \mathbb{R}^{n_1} \times \mathbb{R}^{n_2} \to \mathbb{C}$ whose support in the last two variables is in $[-R^2, R^2]^{n_1 + n_2}$ and $\lim_{\kappa \to 0} \partial_\kappa^{\beta} m( x, \tau, \kappa) = 0$ for all $|\beta| \leq r_0.$ 
\end{corollary}

Choose $\psi_j$ as in \eqref{decompose-dyadic-spectral-symbol-frequency} and for $j\geq 0$, we define $m_{j}(x, \tau, \kappa) := m(x, \tau, \kappa) \psi_j(|(\tau, \kappa)|_1)$. Then we decompose the pseudo-multiplier operator $T = m(x, \boldsymbol{L}, \boldsymbol{U})= \sum_{j=0}^{\infty} T_j $ where $T_j= m(x, \boldsymbol{L}, \boldsymbol{U}) \psi_j(|(\tau, \kappa)|_1)=  m_j(x, \boldsymbol{L}, \boldsymbol{U} )$. For convenience, let us denote the kernel of the operator $T_j$ by $K_j(x,y)$ itself. 

Using Corollaries \ref{cor:pseudo-weighted-Plancherel-Lp} and \ref{cor:pseudo-weighted-Plancherel-Lp-gradient}, one can essentially repeat the proof of Corollary \ref{cor:wighted-plancherel-estimates-for-symbol-kernel} to obtain the following analogous result for integral kernels $K_j$. 

\begin{corollary}\label{cor:joint-funct-wighted-plancherel-esti}
Let $m \in \mathscr{S}^{-Qa/2}_{1-a, \delta}(\boldsymbol{L}, \boldsymbol{U}) $ satisfy condition \eqref{def:grushin-symb-vanishing-0-condition} for all $\beta \in \mathbb{N}^{n_2}$ with $|\beta| \leq r_0$. Then for all $j \geq 0$ and $0 \leq \mathfrak{r} \leq r_0$, we have 
\begin{align} 
\sup_{x\in \mathbb{R}^{n_1 + n_2}} |B(x, 2^{-j/2})| \int_{\mathbb{R}^{n_1 + n_2}} d(x,y)^{2\mathfrak{r}} |K_{j}(x,y)|^2 \, dy & \lesssim_{r_0} 2^{-jQa/2} 2^{-j\mathfrak{r}(1-a)} ,  \label{cond:joint-functional-kernel-hypo} \\ 
\sup_{x\in \mathbb{R}^{n_1 + n_2}} |B(x, 2^{-j/2})| \int_{\mathbb{R}^{n_1 + n_2}} d(x,y)^{2\mathfrak{r}} |X_{x} K_{j}(x,y)|^2 \, dy & \lesssim_{r_0} 2^{-jQa/2} 2^{-j \mathfrak{r}(1-a)} \label{cond:joint-functional-kernel-hypo-grad} 
\end{align}
\end{corollary}


\section{Unweighted boundedness for \texorpdfstring{$\mathscr{S}^{ - Qa/q  }_{1-a, \delta}(G_{\varkappa})$}{} with \texorpdfstring{$1 \leq q < 2$}{}} \label{sec:unweighted-boundedness}

This section is dedicated to prove Theorem \ref{thm:unweightedLp:pseudo}, using which we shall obtain unweighted boundedness for pseudo-multipliers $m(x, G_\varkappa)$ with symbols coming from $\mathscr{S}^{ - Qa/q  }_{1-a, \delta}(G_{\varkappa})$ with $1< q < 2$ as an application of Fefferman--Stein interpolation theorem between symbol classes $\mathscr{S}^{ - Qa }_{1-a, \delta}(G_{\varkappa})$ and $\mathscr{S}^{ - Qa/2 }_{1-a, \delta}(G_{\varkappa})$.


\subsection{The case of \texorpdfstring{$\mathscr{S}^{- Qa}_{1-a, \delta}(G_{\varkappa})$}{}} \label{subsec:H1-L1 boundedness}

Recall that $C_0$ is the constant appearing in the triangle inequality \eqref{def:quasi-constant} of the quasi-distance and $C_{1,\varkappa}$ is the constant appearing in Lemma \ref{lem:general-grushin-Mean-value}. Let us start with the following lemma which will be useful in our purpose.

\begin{lemma} \label{lem:unweightedLp-kernel} 
Let $m \in \mathscr{S}^{-Qa}_{1-a, \delta}(G_{\varkappa})$. Fix any ball $B = B(\mathfrak{z}, r)$ with $r < 1 $. Then for any arbitrarily small $\epsilon > 0$ there exists a constant $C := C_{\epsilon, n_1, n_2, \varkappa, m} > 0$ such that 
\begin{align} 
\sup\limits_{y \in B} \int_{\mathbb{R}^{n_1 + n_2} \setminus B'} | K_j(x, y) - K_j(x, \mathfrak{z}) | \, dx \leq C r^{- (1 - a_{0})/2}2^{- j (1 - a_{0})/4}, ~ \text{if} ~j > j_0, \label{Staubachlemma1} \\ 
\sup\limits_{y \in B} \int_{\mathbb{R}^{n_1 + n_2} \setminus B'} |K_j(x, y) - K_j(x, \mathfrak{z})| \, dx \leq C r^{(1 + a_{0})/2}2^{j (1 + a_{0})/4}, ~ \text{if} ~j \leq j_0 \label{Staubachlemma2}, 
\end{align} 
where $a_{0} = a (1 + 2 \epsilon)$, $B' = B(\mathfrak{z}, 2C_0 C_{1,\varkappa}r^{1 - a_{0}})$, and $j_0$ is the integer such that $2^{j_0} \leq r^{-2} < 2^{j_0 + 1}$. As a consequence, we have
\begin{align*}
\sum_{j} \sup\limits_{y \in B} \int_{\mathbb{R}^{n_1 + n_2} \setminus B'} | K_j(x, y) - K_j(x, \mathfrak{z}) | \, dx \leq C.    
\end{align*}
\end{lemma}	
\begin{proof}
For each $l \in \mathbb{N}$, let us write 
$$\mathcal{A}_{l} = \left\{x: 2C_0 C_{1,\varkappa} 2^l r^{1 - a_{0}} \leq d(x, \mathfrak{z}) \leq 2C_0 C_{1,\varkappa} 2^{l + 1} r^{1 - a_{0}} \right\}.$$

We shall first prove estimate \eqref{Staubachlemma1} and then estimate \eqref{Staubachlemma2}. 

\medskip \noindent \textbf{\underline{Proof of estimate \eqref{Staubachlemma1}}:} Fix $y \in B$ and let $j$ be an integer such that $j > j_{0}$. Now, 
$$\int_{\mathbb{R}^{n_1 + n_2} \setminus B'} | K_j(x, y) - K_j(x, \mathfrak{z}) | \, dx \leq \int_{\mathbb{R}^{n_1 + n_2} \setminus B'} | K_j(x, y) | \, dx + \int_{\mathbb{R}^{n_1 + n_2} \setminus B'} |K_j(x, \mathfrak{z}) \, dx ,$$
and since the estimation of both of the the above terms are similar, we shall only pursue the first one. 

Let us decompose the integral $\int_{\mathbb{R}^{n_1 + n_2} \setminus B'} |K_{j} (x, y)| \, dx $ as $\sum_{l = 0}^{\infty} \int_{\mathcal{A}_{l}} | K_{j} (x, y)| \, dx$. Now, there are two possibilities to consider, for two different ranges of $l$. 

First, if $l$ is such that $2C_0 C_{1,\varkappa} 2^{l + 1} r^{1 - a_{0}} \geq \frac{1}{2} |\mathfrak{z}'|$, then using \eqref{cond:General-hypo-sup} we have 
\begin{align*}
\int_{\mathcal{A}_{l}} | K_{j} (x, y) | \, dx & = \int_{\mathcal {A}_{l}} d(x, {y})^{- ( Q + \frac{1}{2} )} d(x, {y})^{ Q + \frac{1}{2} } | K_{j}(x, {y}) | \, dx \\
& \lesssim \int_{\mathcal{A}_{l}} d(x, \mathfrak{z})^{- ( Q + \frac{1}{2} )}  \frac{2^{-j Q a/2} 2^{-j (\frac{Q}{2} + \frac{1}{4})(1-a)} 2^{ja \epsilon /2}}{|B(x, 2^{-j/2})|^{1/2} |B(y, 2^{-j/2})|^{1/2} }  \, dx \\
& \lesssim 2^{jQ/2} 2^{-jQa/2} 2^{-j(\frac{Q}{2} + \frac{1}{4})(1-a)} 2^{ja\epsilon /2}\frac{ |B(\mathfrak{z}, 2C_0 C_{1,\varkappa} 2^{l + 1} r^{1-a_{0}})|}{(2^{l}r^{1-a_0})^{Q + \frac{1}{2}}} \\
& \lesssim \frac{1}{2^{l/2}}r^{-(1 - a_0)/2} 2^{-j (1 - a_0)/4},
\end{align*}
where in the second inequality we have used the fact that $d(x, y) \sim d(x, \mathfrak{z})$ for all $x \in \mathcal{A}_{l}, y \in B(\mathfrak{z}, r) $ and in the fourth inequality we used the fact $|B(\mathfrak{z}, 2C_0 C_{1,\varkappa} 2^{l + 1} r^{1-a_{0}})| \lesssim (2C_0 C_{1,\varkappa} 2^{l + 1} r^{1-a_{0}})^{Q}$ because $2C_0 C_{1,\varkappa} 2^{l + 1} r^{1-a_{0}} \geq \frac{1}{2} |\mathfrak{
z}'|$.

On the other hand, when $l$ is such that $2C_0 C_{1,\varkappa} 2^{l + 1} r^{1-a_{0}} < \frac{1}{2} |\mathfrak{
z}'|$, then $d(x,\mathfrak{z})< \frac{1}{2} |\mathfrak{
z}'|$, implying that $|\mathfrak{
z}'| \leq 2 |x'|$ for all $x \in \mathcal{A}_{l}$. Similarly, we can show that  in this case we have $|\mathfrak{z}'| \leq 2 |y'|$. Now, using condition  \eqref{cond:General-hypo-sup} we obtain
\begin{align*}
& \int_{\mathcal{A}_{l}} |K_{j}(x, y)| \, dx \\ 
& = \int_{\mathcal {A}_{l}} d(x,  {y})^{-(n_1 + n_2 + \frac{1}{2})} d(x, {y})^{n_1 + n_2 + \frac{1}{2}} |K_{j}(x, {y})| \, dx \\
& \lesssim \int_{\mathcal {A}_{l}}  \frac{1}{(2^l r^{1 - a_0})^{n_1 + n_2 + 1/2}} \frac{2^{-jQa/2} 2^{-j(\frac{n_1 + n_2}{2} + \frac{1}{4})(1-a)} 2^{ja \epsilon /2}}{|B(x, 2^{-j/2})|^{1/2}|B(y, 2^{-j/2})|^{1/2}} \, dx\\
& \lesssim 2^{-j(n_1 + n_2)a/2} 2^{-j(\frac{n_1 + n_2}{2} + \frac{1}{4})(1-a)} 2^{ja\epsilon /2} \frac{1}{(2^l r^{1-a_0})^{n_1 + n_2+1/2}} \frac{|B(\mathfrak{z}, 2C_0 C_{1,\varkappa} 2^{l+1}r^{1-a_{0}})|}{|B(\mathfrak{z}, 2^{-j/2})|}\\
& \lesssim 2^{-j(n_1 + n_2)a/2} 2^{-j(\frac{n_1 + n_2}{2} + \frac{1}{4})(1-a)} 2^{ja\epsilon /2} \frac{1}{(2^l r^{1-a_0})^{n_1 + n_2 + 1/2}} \frac{(2^{l + 1} r^{1-a_{0}})^{n_1 + n_2} |\mathfrak{z}'|^{\varkappa n_{2}}}{2^{-j(n_1 + n_2)/2}|\mathfrak{z}'|^{\varkappa n_{2}}}\\
&\lesssim \frac{1}{2^{l/2}} r^{-(1-a_0)/2} 2^{-j(1-a_0)/4},
\end{align*}
where in the third inequality we have used the condition $2C_0 C_{1,\varkappa} 2^{l + 1} r^{1-a_{0}} < \frac{1}{2} |\mathfrak{
z}'|$ implying that $|B(\mathfrak{z}, 2C_0 C_{1,\varkappa} 2^{l+1}r^{1-a_{0}})| \lesssim$ $( 2^{l + 1} r^{1-a_{0}})^{n_1 + n_2} |\mathfrak{z}'|^{\varkappa n_{2}}$.

Considering both cases in $l$, and summing over $l \geq 0$, one gets the claimed estimate \eqref{Staubachlemma1}. 

\medskip \noindent \textbf{\underline{Proof of estimate \eqref{Staubachlemma2}}:} Fix $y \in B$ and let $j$ be such that $j \leq j_{0}$ and fix $y\in B$. This time, we make use of the mean-value estimate from Lemma \ref{lem:general-grushin-Mean-value} to get 
\begin{align*}
|K_{j}(x, y) - K_j(x, \mathfrak{z})| \lesssim \, d(\mathfrak{z}, y) \int_{0}^{1} \left|X_y K_j(x, \gamma_0(t)) \right| \, dt \leq r \int_{0}^{1} \left|X_y K_j(x, \gamma_0(t)) \right| \, dt,
\end{align*}
where $\gamma_0(t) \in B(\mathfrak{z}, C_{1, \varkappa} \, r)$, and therefore it suffices to estimate, the integral 
$$ r \int_{ \mathbb{R}^{n_1 + n_2} \setminus B'} \left|X_y K_j(x, \gamma_0(t)) \right| \, dx, $$
uniformly in $t \in [0,1]$. 

In order to do this, let us fix $ \gamma_0(t) \in B(\mathfrak{z}, C_{1,\varkappa}r)$, and as earlier, analyse in two different ranges of $l$. 

First, if $l$ is such that $2C_0 C_{1,\varkappa} 2^{l + 1} r^{1 - a_{0}} \geq \frac{1}{2} |\mathfrak{z}'|$, then using estimate 
\eqref{cond:General-hypo-y-grad-sup}, we get 
\begin{align*}
r \int_{\mathcal{A}_l} \left|X_y K_j(x, \gamma_0(t)) \right| \, dx & = r \int_{\mathcal{A}_{l}} d(x, \gamma_0(t))^{-(Q + \frac{1}{2})} d(x, \gamma_0(t))^{Q + \frac{1}{2}} |X_{y} K_j(x, \gamma_0(t))| \, dx \\
& \lesssim r  \int_{\mathcal{A}_{l}} d(x, \mathfrak{z})^{-(Q + \frac{1}{2})}  \frac{2^{-jQa/2} 2^{-j (\frac{Q}{2}  + \frac{1}{4})(1 - a)} 2^{ja\epsilon /2} 2^{j/2}}{ |B(x, 2^{-j/2})|^{1/2} |B(\gamma_0(t), 2^{-j/2})|^{1/2}} \, dx \\
& \lesssim r 2^{jQ/2} 2^{-jQa/2} 2^{-j (\frac{Q}{2} + \frac{1}{4})( 1 - a)} 2^{ja\epsilon /2} 2^{j/2}  \frac{B(\mathfrak{z}, 2C_0 C_{1,\varkappa} 2^{l + 1} r^{1 - a_{0} })}{ (2^{l} r^{1-a_0})^{Q + \frac{1}{2}}} \\
& \lesssim  \frac{1}{2^{l/2}} r^{(1 + a_0)/2} 2^{j (1 + a_0)/4},
\end{align*}
where in the first inequality we have used the fact $ d(x, \gamma_0(t)) \sim d(x, \mathfrak{z})$ whenever $x \in \mathcal{A}_{l}$ and  $\gamma_0(t) \in B(\mathfrak{z}, C_{1, \varkappa})$.

Finally, if $l$ is such that $2C_0 C_{1,\varkappa} 2^{l + 1} r^{1 - a_{0}} < \frac{1}{2} |\mathfrak{z}'|$ then we have $|\mathfrak{z}'|< 2 |x'|$  and $|\mathfrak{z}'| < 2 |\gamma_0(t)'|$ for all $x \in \mathcal{A}_{l}$. Then, using estimate \eqref{cond:General-hypo-y-grad-sup}, we obtain
\begin{align*}
& r \int_{\mathcal{A}_l} \left|X_y K_j(x, \gamma_0(t)) \right| \, dx \\
& = r \int_{\mathcal {A}_{l}} d(x,  \gamma_0(t))^{-(n_1 + n_2 + \frac{1}{2})} d(x, \gamma_0(t) )^{n_1+  n_2 + \frac{1}{2}} |X_{y} K_j(x, \gamma_0(t))| \, dx \\
& \lesssim r \int_{\mathcal {A}_{l}} 
 \frac{1}{(2^k r^{1-a_0})^{n_1 + n_2+1/2}} \frac{2^{-jQa/2} 2^{-j(\frac{n_1 + n_2}{2} + \frac{1}{4})(1-a)} 2^{ja\epsilon /2} 2^{j/2}}{|B(x, 2^{-j/2})|^{1/2}|B(\gamma_0(t), 2^{-j/2})|^{1/2}} \, dx \\
& \lesssim r 2^{-j(n_1 + n_2)a/2} 2^{-j(\frac{n_1 + n_2}{2}+\frac{1}{4})(1-a)} 2^{ja\epsilon /2} 2^{j/2} \frac{1}{(2^l r^{1 - a_0})^{n_1 + n_2 + 1/2}} \frac{|B( \mathfrak{z}, 2C_0 C_{1,\varkappa} 2^{l + 1} r^{1 - a_{0}})|}{|B( \mathfrak{z}, 2^{-j/2})|}\\
& \lesssim 2^{-j(n_1 + n_2)a/2} 2^{-j(\frac{n_1 + n_2}{2} + \frac{1}{4})(1-a)} 2^{ja\epsilon /2} 2^{j/2} \frac{1}{(2^l r^{1-a_0})^{n_1 + n_2 + 1/2}} \frac{(2^{l+1} r^{1-a_{0}})^{n_1 + n_2} |\mathfrak{z}'|^{\varkappa n_{2}}}{2^{-j(n_1 + n_2)/2}|\mathfrak{z}'|^{\varkappa n_{2}}}\\
& \lesssim \frac{1}{2^{l/2}} r^{(1+a_0)/2} 2^{j(1+a_0)/4}. 
\end{align*}

As earlier, the claimed estimate \eqref{Staubachlemma2} follows upon summing over $l$, duly using the estimates in the two different regions in $l$ as above. This completes the proof of Lemma \ref{lem:unweightedLp-kernel}. 
\end{proof}


We are now in a position to prove Theorem \ref{thm:unweightedLp:pseudo}.

\begin{proof}[Proof of Theorem \ref{thm:unweightedLp:pseudo}]
We shall show that under the assumptions of Theorem \ref{thm:unweightedLp:pseudo}, the operator $T=m(x,G_{\varkappa})$ extends as a bounded linear operator, say $\tilde{T}$, from $H^1 \left( \mathbb{R}^{n_1 + n_2} \right)$ to $L^1 \left( \mathbb{R}^{n_1 + n_2} \right)$. But then in view of Proposition $4.2$ of \cite{Meda-Sjogen-Vallarino-H^1-L^1-bound-PAMS-2008}, we will also have that $\tilde{T}$ coincides with $T$ on $H^1_{G_\varkappa} \left( \mathbb{R}^{n_1 + n_2} \right) \cap L^2 \left( \mathbb{R}^{n_1 + n_2} \right)$. Therefore, we can use the notation $T$ itself in place of $\tilde{T}$, and the boundedness of $T$ on $L^p \left( \mathbb{R}^{n_1 + n_2} \right)$, for $1 < p < 2$, would follow from Lemma \ref{lem:Hardy-L2-space-interpolation}.  

Thanks to Theorem $4.1$ of \cite{Meda-Sjogen-Vallarino-H^1-L^1-bound-PAMS-2008}, in order to show that $m(x,G_{\varkappa})$ is $(H^1, L^1)$, it suffices to show that there exists some $C>0$ such that $\|m(\cdot,G_{\varkappa})h\|_{L^{1}}\leq C$ for all $(1, 2)$-atoms $h$. 

So, let us take an arbitrary $(1, 2)$-atom $h$ such that $\supp h \subseteq B = B(\mathfrak{z}, r)$, $\int_{B} h = 0$, and $\|a\|_{L^2} \leq |B|^{-1/{2}}$. We shall analyse the following two cases. 

\medskip \noindent \textbf{\underline{Case 1 $(r \geq 1)$:}} Let $\widetilde{B}$ denote the ball $B(\mathfrak{z}, 2C_0 C_{1,\varkappa} r)$. Then, 
\begin{align*}
\int_{\mathbb{R}^{n_1 + n_2}}| m(x, G_{\varkappa}) h(x) | \, dx & = \int_{\widetilde{B}} | m(x, G_{\varkappa}) h(x) | \, dx + \int_{\mathbb{R}^{n_1 + n_2} \setminus \widetilde{B}} |m(x, G_{\varkappa}) h(x)| \, dx =: J_1 + J_{2},
\end{align*}
and 
\begin{align*}
J_{1} & \leq |\widetilde{B}|^{\frac{1}{2}} \left(\int_{\widetilde{B}} |m(x,G_{\varkappa})h(x)|^{2} \, dx \right)^{\frac{1}{2}} \lesssim_{m} |\widetilde{B}|^{\frac{1}{2}} \left( \int_{B} |h(x)|^{2} \, dx \right)^{\frac{1}{2}} \leq |\widetilde{B}|^{\frac{1}{2}}  |B|^{-\frac{1}{2}} \leq C.
\end{align*}

Recall that $m(x, G_{\varkappa}) = \sum_{j} m_j(x, G_{\varkappa})$. Using $\int_B h = 0$, and H\"older's inequality, we get 
\begin{align*}
	J_{2} & \leq \sum_{j\geq 0} \int_{\mathbb{R}^{n_1 + n_2} \setminus \widetilde{B}} \int_{B} |K_{j}(x, y) - K_{j}(x, \mathfrak{z})| |h(y)| \, dy \, dx \\
	& \leq \sum_{j\geq 0} \left(\sup\limits_{ y \in B} \int_{\mathbb{R}^{n_1 + n_2}\setminus \widetilde{B}} |K_j(x, y) - K_j (x, \mathfrak{z})| \, dx \right) \int_{B} |h(y)| \, dy\\
	& \leq \sum_{j \geq 0} \left( \sup\limits_{ y \in B} \int_{\mathbb{R}^{n_1 + n_2} \setminus \widetilde{B}} |K_j(x, y)-K_j(x, \mathfrak{z})| \, dx \right) |B|^{\frac{1}{2}} \left( \int_{B} |h(y)|^2 \, dy \right)^{\frac{1}{2}}\\
	& \lesssim \sum_{j \geq 0} \left( \sup\limits_{ y \in B} \int_{ \mathbb{R}^{n_1 + n_2} \setminus \widetilde{B}} |K_j(x, y) - K_j(x, \mathfrak{z})| \, dx \right).
\end{align*}

Now the proof will follow if we can establish 
\begin{align} \label{Global-kerenel-difference}
\sup\limits_{y \in B} \int_{\mathbb{R}^{n_1 + n_2} \setminus \widetilde{B}} |K_{j}(x, y) - K_{j}(x, \mathfrak{z})| \, dx \leq 2^{-j(1-a_0)/4},
\end{align}
where $a_0 = a (1 + 2 \epsilon)$ with $\epsilon>0$ being very small such that $a_0 < 1$.

The rest of the proof is devoted to prove the claimed estimate \eqref{Global-kerenel-difference}. For the same, let us fix $y \in B$, then 
$$ \int_{\mathbb{R}^{n_1 + n_2} \setminus \widetilde{B}} |K_j(x, y) - K_j(x, \mathfrak{z})| \, dx \leq \int_{\mathbb{R}^{n_1 + n_2}\setminus \widetilde{B}} |K_j(x, y)|  dx + \int_{\mathbb{R}^{n_1 + n_2} \setminus \widetilde{B}} |K_j(x, \mathfrak{z})| dx .$$

The estimation of both of the above terms is similar and therefore it is sufficient to estimate the first one. For each $l \in \mathbb{N}$, let us write 
\begin{align*}
\mathcal{A}_{l} & = \left\{x: 2 C_0 C_{1,\varkappa} 2^l r \leq d(x, \mathfrak{z}) \leq 2 C_0 C_{1,\varkappa} 2^{l + 1} r \right\},
\end{align*}
and decompose the integral $ \int_{\mathbb{R}^{n_1 + n_2} \setminus \widetilde{B}} |K_j(x, y)|\ dx $ as $\sum_{l = 0}^{ \infty} \int_{\mathcal{A}_{l}} |K_j(x, y)| \, dx.$ 

As earlier, there are two possibilities to consider. If $l$ is such that $ 2 C_0 C_{1,\varkappa} 2^{l + 1} r > \frac{|\mathfrak{z}'|}{2}$, then using \eqref{cond:General-hypo-sup}, with $\epsilon$ very small so that $a_0 < 1$, we get
\begin{align*}
\int_{\mathcal{A}_{l}} |K_j(x, y)| \, dx & = \int_{\mathcal {A}_{l}} d(x, {y})^{-(Q + \frac{1}{2})} d(x, {y})^{Q + \frac{1}{2}} |K_j(x, {y})| \, dx\\
& \lesssim \int_{\mathcal{A}_{l}} \frac{1}{(2^{l} r)^{Q + \frac{1}{2}}} \frac{2^{-jQa/2} 2^{-j(\frac{Q}{2} + \frac{1}{4})(1 - a)} 2^{j a \epsilon /2}}{|B(x, 2^{-j/2})|^{1/2} |B(y, 2^{-j/2})|^{1/2}}  \, dx\\
& \lesssim 2^{jQ/2} 2^{-jQa/2} 2^{-j (\frac{Q}{2} +\frac{1}{4})(1 - a)} 2^{j a \epsilon /2} \frac{|B(\mathfrak{z}, 2 C_0 C_{1,\varkappa} 2^{l + 1}r)|}{(2^{l} r)^{Q + \frac{1}{2}}} \lesssim \frac{1}{2^{l/2}} 2^{-j(1-a_0)/4},
\end{align*}
where we have used the fact that $ r \geq 1$. 

Next, for $l$ such that $ 2 C_0 C_{1,\varkappa} 2^{l + 1}r \leq \frac{|\mathfrak{z}'|}{2}$, we have $|\mathfrak{z}'|< 2 |x'|$ for all $x \in \mathcal{A}_{l}$. Therefore, using again condition \eqref{cond:General-hypo-sup}, with $\epsilon$ very small so that $a_0 < 1$, we obtain
\begin{align*}
 \int_{\mathcal{A}_{l}} |K_j(x, y)| \, dx & = \int_{\mathcal {A}_{l}} d(x, y)^{-(n_1 + n_2 + \frac{1}{2})} d(x, y)^{n_1 + n_2 + \frac{1}{2}} |K_j(x, {y})| \, dx \\
& \lesssim \int_{\mathcal {A}_{l}}  \frac{1}{(2^l r)^{n_1 + n_2 + 1/2}} \frac{2^{-j (n_1 + n_2) a/2} 2^{-j (\frac{n_1 + n_2}{2} + \frac{1}{4}) (1 - a)} 2^{j a \epsilon /2}}{|B(x, 2^{-j/2})|^{1/2} |B(y, 2^{-j/2})|^{1/2}} \, dx \\
& \lesssim 2^{-j (n_1 + n_2) a/2} 2^{-j(\frac{n_1 + n_2}{2} + \frac{1}{4})(1 - a)} 2^{j a \epsilon /2} \frac{1}{(2^l r)^{n_1 + n_2 + 1/2}} \frac{|B(\mathfrak{z}, 2C_0 C_{1,\varkappa} 2^{l+1} r)|}{|B(\mathfrak{z}, 2^{-j/2})|} \\
& \lesssim 2^{-j (n_1 + n_2) a/2} 2^{-j (\frac{n_1 + n_2}{2} + \frac{1}{4}) (1 - a)} 2^{j a \epsilon /2} \frac{1}{(2^l r)^{n_1 + n_2 + 1/2}} \frac{(2^{l + 1} r)^{n_1 + n_2} |\mathfrak{z}'|^{\varkappa n_{2}}}{2^{-j (n_1 + n_2)/2} |\mathfrak{z}'|^{\varkappa n_{2}}} \\
&\lesssim \frac{1}{2^{l/2}}  2^{-j (1 - a_0)/4}, 
\end{align*}
where again we have used the fact that $r \geq 1$. 

Using the above two estimates in different regimes of $l$, we get the claimed estimate \eqref{Global-kerenel-difference}. 

\medskip \noindent \textbf{\underline{Case 2 $(r < 1)$:}} Let us write $B' = B( \mathfrak{z}, 2 C_0 C_{1,\varkappa}  r^{1 - a_{0}} ) $ with $ a_{0} = a ( 1 + 2 \epsilon) $ be such that $ a_{0} < min \{ 1, 2a \}$. We decompose
\begin{align*}
\int_{\mathbb{R}^{n_1 + n_2}} | m(x, G_{\varkappa}) h(x) | \, dx & = \int_{B'} | m(x, G_{\varkappa}) h(x)| \, dx + \int_{\mathbb{R}^{n_1 + n_2} \setminus B'} | m(x, G_{\varkappa}) h(x) |  \, dx =: I + II.
\end{align*}

Let us first estimate $ I $. For $b \in \mathbb{R}$, denote by $ J_{b} $ the potential operator $(I + G)^{b/2}$. Since $a_0 \leq 2 a $ we have $ m(x, \eta) (1 + \eta)^{Q a_0/4} \in \mathscr{S}^{0}_{\rho, \delta}( G_{\varkappa} )$ and hence assumption \eqref{assumption-weighted-L2} implies that $ m(x, G_{\varkappa}) \circ J_{Q a_0/2} \in \mathcal{B} \left( L^{2}(\mathbb{R}^{n_1 + n_2}, \omega_{\mathfrak{b}}) \right)$, where $\omega_{\mathfrak{b}} = |B(x, 1)|^{\mathfrak{b}}$ for $0 \leq \mathfrak{b} < 1$. Therefore,  
\begin{align*}
I & = \int_{B'} | m(x, G_{\varkappa}) h(x) | \, dx \\ 
& \leq |B'|^{1/2} \left( \int_{B'} |m(x, G_{\varkappa}) h(x) |^2 \, dx \right)^{1/2} \\
& = |B'|^{1/2} |B(\mathfrak{z}, 1)|^{-a_{0}/2} \left( \int_{B'} |B(x, 1)|^{a_{0}} |m(x,G_{\varkappa}) J_{Q a_{0}/2} J_{-Q a_{0}/2} h(x) |^2 \, dx \right)^{1/2}\\
& \lesssim |B'|^{1/2} |B(\mathfrak{z}, 1)|^{-a_{0}/2} \left( \int_{B'} |B(x, 1)|^{a_{0}} |J_{-Q a_{0}/2} h(x) |^2 \, dx \right)^{1/2},
\end{align*}	
and then using the Hardy--Littlewood--Sobolev inequality \eqref{eq:Sobolev-ineq} with $ p = \frac{2}{1 + a_{0}}, q = 2$ and $b = \frac{Q a_0}{2}$, the above estimate implies 
\begin{align*}
I \lesssim |B'|^{1/2} |B(\mathfrak{z}, 1)|^{-a_{0}/2} \|h\|_{ L^{\frac{2}{1 + a_{0} }}} \lesssim |B'|^{1/2} |B(\mathfrak{z}, 1)|^{-a_{0}/2}  |B|^{-(1 - a_{0})/2}.
\end{align*}

Now, we estimate the quantity $|B'|^{1/2} |B(\mathfrak{z}, 1)|^{-a_{0}/2}  |B|^{-(1 - a_{0})/2}$ as follows. Note that if $r^{1 - a_{0}} > |\mathfrak{z}'|$, then 
$$ |B'|^{1/2} |B(\mathfrak{z}, 1)|^{-a_{0}/2} |B|^{-(1 - a_0)/2} \lesssim |B'|^{1/2} |B|^{-(1 - a_0)/2} \lesssim \left(r^{1 - a_{0}} \right)^{Q/2} \left(r^{Q} \right)^{-(1 - a_0)/2} = 1,$$
whereas, if $ r^{1 - a_{0}} \leq |\mathfrak{z}'|$, then
\begin{align*}
 |B'|^{1/2} |B(\mathfrak{z}, 1)|^{-a_{0}/2} |B|^{-(1 - a_0)/2} & \lesssim \left\{ \left(r^{1 - a_{0}} \right)^{n_1 + n_2} |\mathfrak{z}'|^{\varkappa n_{2}} \right\}^{1/2}  |\mathfrak{z}'|^{- a_0 \varkappa n_{2}/2} \left\{ r^{n_1 + n_2} |\mathfrak{z}'|^{\varkappa n_{2}} \right\}^{-(1 - a_0)/2}\\
& = 1.   
\end{align*}
Combining the above estimates, we get that $I \leq C.$

To estimate $II$, we also make use of the cancellation of $h$ to have 
\begin{align*}
II & \leq \sum_{j} \int_{B} \int_{\mathbb{R}^{n_1 + n_2} \setminus B'} |K_j(x, y) - K_j(x,  \mathfrak{z})| |h(y)| \, dy  \, dx \\ 
& \leq \sum_{j} \left( \sup \limits_{ y \in B} \int_{\mathbb{R}^{n_1 + n_2} \setminus B'} | K_j (x, y) - K_j (x, \mathfrak{z}) | \, dx \right) \int_{B} |h(y)| \, dy \\ 
& \lesssim \sum_{j} \sup\limits_{y \in B} \int_{\mathbb{R}^{n_1 + n_2} \setminus B'} |K_j(x, y) - K_j(x, \mathfrak{z})| \, dx \leq C, 
\end{align*} 
where the last inequality is true in view of Lemma \ref{lem:unweightedLp-kernel}. 

This completes the proof of Theorem \ref{thm:unweightedLp:pseudo}.
\end{proof}


\subsection{The case of \texorpdfstring{$\mathscr{S}^{-Qa/q}_{1-a, \delta}(G_{\varkappa})$}{} with \texorpdfstring{$1 < q < 2$}{}} \label{subsec:Lp-boundedness-intermediatary-classes}

\begin{theorem} \label{thm:unweighted-Lq-intermediate-class}
Let $m \in \mathscr{S}^{-Qa/q}_{1-a, \delta}(G_{\varkappa})$ with $\sigma = - Qa/q$ and $1 < q < 2$. Then, the operator $m(x, G_{\varkappa})$ is bounded on $ L^q \left( \mathbb{R}^{n_1 + n_2} \right)$. 
\end{theorem}
\begin{proof}
We shall prove the theorem using Fefferman--Stein complex interpolation method. For $T = m(x, G_{\varkappa})$, we shall also make use of the fact that $T \circ \left( I + G_{\varkappa} \right)^{\frac{Qa}{2q}} \in \mathcal{B} (L^2 (\mathbb{R}^{n_1 + n_2}))$. 

Let us write $S = \{ z \in \mathbb{C} : 0 < Re z < 1 \}$, and for each $z \in \bar{S}$, consider the symbol function 
\begin{align*}
m_{z}(x, \eta) = e^{z^2} m(x, \eta)(1 + \eta)^{\frac{Qa}{2q} - \frac{Qa}{4}(1 + z)}, 
\end{align*}
and denote by $T_z$ the associated operator $T_z = m_z(x, G_{\varkappa})$. 

Observe that 
\begin{align*}
\|T_z\|_{L^2 \rightarrow L^2} & = \|e^{z^2} T \circ \left( I + G_{\varkappa} \right)^{\frac{Qa}{2q} -\frac{Qa}{4} (1 + z)} \|_{L^2 \rightarrow L^2} \\
& = \left|e^{z^2} \right| \| T \circ \left( I + G_{\varkappa} \right)^{\frac{Qa}{2q}} \circ \left( I + G_{\varkappa} \right)^{ -\frac{Qa}{4} (1 + z)} \|_{L^2 \rightarrow L^2} \\ 
& \leq \left|e^{z^2} \right| \| T \circ \left( I + G_{\varkappa} \right)^{\frac{Qa}{2q}} \|_{L^2 \rightarrow L^2} \| \left( I + G_{\varkappa} \right)^{ -\frac{Qa}{4} (1 + z)} \|_{L^2 \rightarrow L^2} \\ 
& \lesssim_{T} \left|e^{z^2} \right| \| \left( I + G_{\varkappa} \right)^{ -\frac{Qa}{4} (1 + z)} \|_{L^2 \rightarrow L^2}. 
\end{align*}

But, Plancherel's theorem implies that 
$$\| \left( I + G_{\varkappa} \right)^{ -\frac{Qa}{4} (1 + z)} \|_{L^2 \rightarrow L^2} = \sup_{\eta > 0} \left| \left( 1 + \eta \right)^{ -\frac{Qa}{4} (1 + z)} \right| = \sup_{\eta > 0} \left| \left( 1 + \eta \right)^{ -\frac{Qa}{4} (1 + Re (z))} \right|,$$ 
and therefore we get that $\sup_{z \in \bar{S}} \|T_z\|_{L^2 \rightarrow L^2} < \infty. $

Next, define the following set which is which is dense in $L^2(\mathbb{R}^{n_1 + n_2})$: 
\begin{align*}
\mathcal{V}_{\varkappa} = \left\{ f \in L^2(\mathbb{R}^{n_1 + n_2}) : f(x) = \int_{\mathbb{R}^{n_2}} \right. & e^{-i \lambda \cdot x^{\prime \prime}} \sum_{k \text{-finite}} C(\lambda, \nu_{\varkappa, k}) h^{\lambda}_{\varkappa, k}(x') \, d\lambda, \text{ where } C(\lambda, k) \text{ is} \\ 
& \text{compactly supported in } \lambda \text{-variable for each } k \left. \right\}. 
\end{align*}

It can be easily verified that for each $f, g \in \mathcal{V}_{\varkappa}$, the map $z \mapsto \int_{\mathbb{R}^{n_1 + n_2}} (T_z \, f) \, g$, is holomorphic in the open strip $S$ and continuous in $\bar{S}$, and we leave the details. 

Let us claim that 
\begin{align} \label{ineq:H1-L1-interpolation-step}
\sup_{\{z ~ : ~ Re(z) = 1\}} \|T_z\|_{H^1 \rightarrow L^1} < \infty. 
\end{align}

If the claimed estimate \eqref{ineq:H1-L1-interpolation-step} holds true, then it shall follow from the Fefferman--Stein interpolation theorem that for every for $t \in (0, 1)$, the operator $T_{t}$ is bounded on $L^{p_t}(\mathbb{R}^{n_1 + n_2})$ where $p_t = \frac{2}{1+t}$. In particular, choosing $t = \frac{2}{q} - 1$, one would get that $T = T_{\frac{2}{q} - 1}$ is bounded on $L^q (\mathbb{R}^{n_1 + n_2})$. So, let us prove estimate \eqref{ineq:H1-L1-interpolation-step}. Note that for any $z = 1 + i u_2$, we have 
\begin{align*}
\left|X_x^\Gamma \partial_{\eta}^l m_{z}(x, \eta)\right| \lesssim e^{-u_2^2} \, P(|z|) \, (1+\eta)^{-\frac{Qa}{2} - \frac{l(2-a)}{2} + \delta \frac{|\Gamma|}{2}} \lesssim_l (1+\eta)^{-\frac{Qa}{2} - \frac{l(2-a)}{2} + \delta \frac{|\Gamma|}{2}}, 
\end{align*}
where $P$ is a polynomial in one complex variable, of degree at most $l$, and in the last inequality we have used the simple fact that $\sup_{u_2 \in \mathbb{R}} e^{-u_2^2}P(|z|) = C_l < \infty.$ 

We have thus shown that $m_{1 + i u_2} (x,\eta) \in \mathscr{S}^{-Qa}_{1-a,\delta}(G_{\varkappa})$ with seminorm of symbols $m_{1 + i u_2}$ being uniform in $u_2$. One can then invoke Theorem \ref{thm:unweightedLp:pseudo} to conclude the bound as claimed in \eqref{ineq:H1-L1-interpolation-step}. This completes the proof of the Theorem \ref{thm:unweighted-Lq-intermediate-class}.
\end{proof}


\section{Sparse domination results for \texorpdfstring{$\mathscr{S}^{ - Qa/q  }_{1-a, \delta}(G_{\varkappa})$}{} with \texorpdfstring{$1\leq q< 2$}{}} \label{sec:sparse-domination-results}

In this section we shall prove our sparse domination result that is Theorem \ref{thm:Main-Sparse-Result}. 
The proof of Theorem \ref{thm:Main-Sparse-Result} requires the sparse domination principle Theorem~1.1 of \cite{Lorist-pointwaise-sparse2021} (we also refer Theorem 2.11 of \cite{BBGG-1}, where we have stated this result particularly for the homogeneous space associated to the Grushin metric). The mentioned principle depends on two intermediate conditions. First is an appropriate unweighted boundedness for the pseudo-multiplier operators associated with the symbol classes $\mathscr{S}^{-Qa/q}_{1-a, \delta}(G_{\varkappa})$ with $1 \leq q < 2$. For $q=1$, the unweighted boundedness was addressed in Theorem \ref{thm:unweightedLp:pseudo}, and in the case of $1<q<2$ the same was addressed in Theorem \ref{thm:unweighted-Lq-intermediate-class}. Secondly, we need end-point bounds for the grand maximal truncated operator, and this will be established in Lemma \ref{lem:staubach-Intermediate}. 

In the following lemma we establish some kernel estimates which are essential for the proof of the sparse domination. In the Euclidean setting similar estimates were obtained in \cite{Michalowski-Rule-Staubach-Canad2012}. In \cite{Michalowski-Rule-Staubach-Canad2012} the authors have heavily relied on the Hausdorff--Young theorem. However, in our context of the Grushin operator, in the absence an exact Hausdorff--Young theorem, we perform very delicate modifications of ideas of \cite{Michalowski-Rule-Staubach-Canad2012}. 
 
\begin{lemma} \label{lem: kernel-difference-p}
Let $m \in \mathscr{S}^{-Qa/q}_{1-a, \delta}(G_{\varkappa})$ with $\sigma = - Qa/q$ and $1 \leq q < 2$. Let $B=B(\mathfrak{z}, r)$ be any ball with $0< r \leq 1$. Then for $ 1\leq p \leq  2$, $l \geq 1$, $s = 3C_0^2 C_{1,\varkappa}$, $0\leq \theta \leq 1$ and for sufficiently small $\epsilon>0$, the following estimates hold:
\begin{align} \label{ms1-HD} 
\sup_{y \in B} \sum_{j=0}^{\infty} \left(\int_{2^{l} sr^{\theta} \leq d(v,\mathfrak{z}) \leq  2^{l+1} sr^{\theta} }| K_j (y,v)-K_j(\mathfrak{z},v)|^{p'} dv \right)^{1/p'} \lesssim_{L, \theta, p, q} (2^{l})^{- L} r^{L (1 - a - \theta) + \frac{Qa}{q} - \frac{Q}{p}- a \epsilon}, 
\end{align}
whenever $-\frac{Qa}{q} + \frac{Q}{p} < L \, (1-a) < -\frac{Qa}{q} + \frac{Q}{p} + (1 + a\epsilon)$, and 
\begin{align} \label{ms2-HD} 
\sup_{y \in B} & \sum_{j = 0}^{\infty} \left(\int_{ 2^{l} sr^{\theta} \leq d(v,\mathfrak{z}) \leq  2^{l+1} sr^{\theta}}| K_j (y,v)- K_j(\mathfrak{z},v)|^{p'} dv \right)^{1/p'} \quad\quad\quad\quad\quad\quad\quad\quad\quad\quad\quad\quad \\
\nonumber & \lesssim_{L, \theta, p, q} (2^{l})^{-L} r^{L(1-a-\theta) + \frac{(n_1 + n_2)a}{q} - \frac{n_1 + n_2}{p} - a \epsilon} 
 |\mathfrak{z}'|^{-\varkappa n_2/p}, 
\end{align} 
whenever $-\frac{(n_1 + n_2)a}{q} + \frac{n_1 + n_2}{p} < L \, (1-a) < -\frac{(n_1 + n_2)a}{q} + \frac{n_{1}+ n_{2}}{p} + (1+ a\epsilon)$ and $\frac{1}{2}|\mathfrak{z}'| \geq  2^{l+1}sr^{\theta}$.
\end{lemma}

\begin{proof} 
We shall first prove estimate \eqref{ms1-HD} and then estimate \eqref{ms2-HD}. 

\medskip \noindent \textbf{\underline{Proof of estimate \eqref{ms1-HD}}:} Fix $y\in B$, and let $j_0$ be the integer such that $2^{j_0}r^{2}\simeq 1$. Write $\mathcal{A}_{l} := \{v\in \mathbb{R}^{n_1 + n_2} : 2^{l}sr^{\theta} \leq d(v,\mathfrak{z}) \leq  2^{l+1}sr^{\theta}\}$, and denote
\begin{align*}
I_{1} & :=\sum_{j\leq j_{0}} \left(\int_{\mathcal{A}_{l}} | K_j(y,v) - K_j(\mathfrak{z},v)|^{p'} \, dv \right)^{1 / p'},\\
I_{2} & := \sum_{j> j_{0}} \left(\int_{\mathcal{A}_{l}}| K_j(y,v) |^{p'} \, dv \right)^{\frac{1}{p'}}, \quad \text{and} \quad I_{3} := \sum_{j> j_{0}} \left(\int_{\mathcal{A}_{l}}| K_j(\mathfrak{z}, v) |^{p'} \, dv \right)^{\frac{1}{p'}}.
\end{align*}

\medskip \textbf{\underline{Estimate of $I_1$}:} Using mean-value estimate from Lemma \ref{lem:general-grushin-Mean-value} we obtain 
\begin{align*} 
|K_{j}(y,v) - K_{j}(\mathfrak{z},v)| & \lesssim r \int_0^1 \left| \left( X_x K_j \right) \left( \gamma_0(t), v \right) \right| dt, 
\end{align*}
with $\gamma_0(t) \in B(\mathfrak{z}, C_{1, \varkappa} \, r)$. Therefore, using \eqref{cond:General-hypo-grad} and \eqref{cond:General-hypo-grad-sup} and subsequently the estimate $|B(y, 2^{-j/2})| \gtrsim 2^{-jQ/2}$ for all $y\in \mathbb{R}^{n_1+n_2}$, we obtain
\begin{align*}
I_1 & \lesssim r \sum_{j\leq j_{0}} \left( \int_{\mathcal{A}_{l}} \left(\int_{0}^{1}|X_x K_j(\gamma_{0}(t), v)| \, dt \right)^{p'} \, dv \right)^{\frac{1}{p'}} \\
& \lesssim r \sum_{j \leq j_{0}}  \int_{0}^{1} \left(\int_{\mathcal{A}_{l}} |X_x K_j(\gamma_{0}(t), v)|^{p'} \, dv \right)^{\frac{1}{p'}} \, dt \\
& \lesssim_{L} r \sum_{j \leq  j_{0}} \int_{0}^{1} \left( \int_{\mathcal{A}_{l}}  \left\{ (2^lsr^{\theta})^{-L} d(\gamma_{0}(t), v)^{L} |X_x K_j(\gamma_{0}(t), v)| \right\}^{p'} \, dv \right)^{\frac{1}{p'}} \, dt \\ 
& \lesssim_L r (2^l sr^{\theta})^{-L} \sum_{j \leq j_{0}} \int_{0}^{1} \left(\sup_{v \in \mathcal{A}_{l}}\left\{ d(\gamma_{0}(t), v)^{L} |X_x K_j(\gamma_{0}(t), v)| \right\}^{p'-2} \right.\\
& \quad \quad \quad \quad \quad \quad \quad \left.\int_{\mathcal{A}_{l}}  \left\{ d(\gamma_{0}(t), v)^{2L} |X_x K_j(\gamma_{0}(t), v)|^{2} \right\} \, dv \right)^{\frac{1}{p'}}\, dt \\ 
& \lesssim_{L,\epsilon} r (2^l sr^{\theta})^{-L} \sum_{j \leq j_{0}} \left( \left\{ 2^{-jQa/2q} 2^{-jL(1-a)/2}  2^{jQ/2} 2^{ja\epsilon/2}  \right\}^{\frac{p'-2}{p'}} \right.\\
& \left. \quad \quad \quad \quad \quad \left\{ 2^{-jQa/q} 2^{-jL(1-a)} 2^{jQ/2}2^{ja\epsilon} \right\} \right)^{\frac{1}{p'}} 2^{\frac{j}{2}(\frac{p'-2}{p'} + \frac{2}{p'})} \\
& \lesssim_{L, \epsilon} r (2^l sr^{\theta})^{-L} \sum_{j\leq  j_{0}} 2^{-\frac{jQa}{2q}}  2^{-\frac{jL(1-a)}{2}} 2^{\frac{jQ}{2}(\frac{p'-2}{p'} + \frac{1}{p'})} 2^{\frac{j}{2}} 2^{\frac{ja\epsilon}{2}} \\
& \lesssim_{L, \epsilon} r (2^l sr^{\theta})^{-L} \sum_{j \leq j_{0}}   2^{-\frac{jQa}{2q} } 2^{-\frac{jL(1-a)}{2}} 2^{\frac{jQ}{2p}} 2^{\frac{j}{2}} 2^{\frac{ja\epsilon}{2}}. 
\end{align*}
Now, if $\frac{L(1-a)}{2} - \left(\frac{-Qa}{2q} + \frac{Q}{2p} \right) - \frac{1 + a \epsilon}{2} < 0$, which is same as $L < \frac{\frac{-Qa}{q} + \frac{Q}{p}}{1-a} + \frac{1 + a\epsilon}{1-a}$, then the above sum is convergent and 
$$ I_1 \lesssim_{L, \epsilon} r (2^l sr^{\theta})^{-L} 2^{j_0 \left\{ \left(\frac{-Qa}{2q} + \frac{Q}{2p} \right) -\frac{L(1-a)}{2} + \frac{1 + a\epsilon}{2} \right\}} \lesssim_{L, s, \theta} (2^{l})^{-L} r^{ L
(1-a-\theta)+\frac{Qa}{q}-\frac{Q}{p}-a\epsilon}.$$

\medskip \textbf{\underline{Estimate of $I_2$}:} Note that 
\begin{align*} 
I_2 & = \sum_{j> j_{0}} \left(\int_{\mathcal{A}_{l}}| K_j(y, v) |^{p'} \, dv \right)^{\frac{1}{p'}} \\ 
\nonumber & \lesssim_{L} \sum_{j > j_{0}} \left(\int_{\mathcal{A}_{l}} \left\{ (2^l sr^{\theta})^{-L} d(y, v)^{L} |K_j(y, v)| \right\}^{p'} \, dv  \right)^{\frac{1}{p'}} \\ 
\nonumber & \lesssim_{L} (2^l sr^{\theta})^{-L} \sum_{j> j_{0}} \left( \int_{\mathcal{A}_{l}} \left\{ d(y, v)^{L} |K_j(y, v)| \right\}^{p'-2} \left\{ d(y, v)^{2L}|K_j(y, v)|^{2} \right\} \, dv \right)^{\frac{1}{p'}}\\
\nonumber & \lesssim_{L} (2^l sr^{\theta})^{-L} \sum_{j > j_{0}} \left( \sup_{v \in \mathbb{R}^{n_1 + n_2}} \left\{ d(y, v)^{L} |K_j(y, v)| \right\}^{p'-2} \int_{\mathcal{A}_{l}}  \left\{ d(y, v)^{2L}|K_j(y, v)|^{2} \right\} \, dv \right)^{\frac{1}{p'}}, 
\end{align*}
and then making use of estimates \eqref{cond:General-hypo} and \eqref{cond:General-hypo-sup} together with $|B(y, 2^{-j/2})| \gtrsim 2^{-jQ/2}$ and $|B(v, 2^{-j/2})| \gtrsim 2^{-jQ/2}$, we get 
\begin{align*}
I_{2} & \lesssim_{L, \epsilon} \sum_{j \geq j_{0}} (2^l sr^{\theta})^{-L} \left( \left\{ 2^{-jQa/2q} 2^{-j L (1-a)/2} 2^{jQ/2} 2^{j a \epsilon/2} \right\}^{p'-2} \left\{ 2^{-jQa/q} 2^{-jL(1-a)} 2^{jQ/2}  2^{ja\epsilon} \right\} \right)^{\frac{1}{p'}} \\
& \lesssim_{L, \epsilon} (2^l sr^{\theta})^{-L} \sum_{j\geq j_{0}} 2^{-jQa/2q} 2^{-jL(1-a)/2} 2^{\frac{jQ}{2} \left( \frac{p'-2}{p'} + \frac{1}{p'} \right)}  2^{ja\epsilon/2}\\ 
& \lesssim_{L, \epsilon} (2^l sr^{\theta})^{-L} \sum_{j\geq j_{0}} 2^{-jQa/2q} 2^{-jL(1-a)/2}2^{jQ/2p} 2^{ja\epsilon/2}. 
\end{align*}

Clearly, the above some converges when $L>\frac{\frac{-Qa}{q} + \frac{Q}{p}}{1-a}$ and $\epsilon>0$ is sufficiently small, and in that case we get
\begin{align*}
I_{2} & \lesssim_{L, \epsilon} (2^l sr^{\theta})^{-L} 2^{-j_{0}Qa/2q} 2^{-j_{0} L(1-a)/2} 2^{\frac{j_{0}Q}{2p}} 2^{j_0a\epsilon/2}\\
& \lesssim_{L, \epsilon} (2^l sr^{\theta})^{-L} r^{\frac{Qa}{q} + L(1-a) - \frac{Q}{p} - a\epsilon} \lesssim_{L, \theta, \epsilon} (2^{l})^{-L} r^{(1-a-\theta)L+\frac{Qa}{q} - \frac{Q}{p}-a\epsilon}.  
\end{align*}

Performing similar calculations, one can show that 
$$I_{3} \lesssim_{L, \theta, \epsilon} (2^{l})^{-L} r^{(1-a-\theta)L+\frac{Qa}{q} - \frac{Q}{p}-a\epsilon},$$
provided that $L>\frac{\frac{-Qa}{q}+\frac{Q}{p}}{1-a}$. 

The above estimates of $I_{1}, I_{2}$, and $I_{3}$ together imply estimate \eqref{ms1-HD}. 

\medskip \noindent \textbf{\underline{Proof of estimate \eqref{ms2-HD}}:} Let us decompose the left hand side of \eqref{ms2-HD} into three parts $I_1, I_2,$ and $I_3$ exactly as we did in the proof of \eqref{ms1-HD}. We shall only show the changes in the estimate of $I_{1}$. Similar arguments can be carried out for $I_{2}$ and $I_{3}$. We here use the fact that if $d(\mathfrak{z}, v)\thicksim 2^{l} sr^{\theta}$ and  $\frac{1}{2}|\mathfrak{z}'| \geq  2^{l+1} sr^{\theta}$ then $|\mathfrak{z}'| \leq 2 |v'|$. Similarly, $|\mathfrak{z}'| \leq 2 |\gamma_0(t)^{'}|$. Then, 
\begin{align*}
I_1 
& \lesssim_{L} r (2^l sr^{\theta})^{-L} \sum_{j\leq j_{0}} \int_{0}^{1}  \sup_{v \in \mathcal{A}_{l}} \left\{ d(\gamma_{0}(t), v)^{L} |X_x K_j(\gamma_{0}(t), v)| \right\}^{\frac{p'-2}{p'}}\\ 
&\quad\quad\quad\quad\quad\quad \left(\int_{\mathcal{A}_{l}}  \left\{ d(\gamma_{0}(t), v)^{2L} |X_x K_j(\gamma_{0}(t), v)|^{2} \right\} \, dv \right)^{\frac{1}{p'}}\, dt \\ 
& \lesssim_{L, \epsilon} r (2^l sr^{\theta})^{-L} \sum_{j \leq j_{0}} \int_{0}^{1} \sup_{v \in \mathcal{A}_{l}}  \left\{  \frac{2^{-jQa/2q} 2^{-jL(1-a)/2} 2^{ja\epsilon/2} 2^{j/2}}{|B(\gamma_{0}(t),2^{-j/2})|^{-1/2} |B(v,2^{-j/2})|^{-1/2}} \right\}^{\frac{p' - 2}{p'}} \\
&\quad \quad\quad\quad\quad\quad\quad\quad \left( 2^{-jQa/q} 2^{-jL(1-a)} |B(\gamma_{0}(t),2^{-j/2})|^{-1} 2^{ja\epsilon}2^{j} \right)^{\frac{1}{p'}}\, dt\\
 & \lesssim_{L, \epsilon} r (2^l sr^{\theta})^{-L} \sum_{j\leq j_{0}} \int_{0}^{1}  \sup_{v \in \mathcal{A}_{l}} \left\{ 2^{-jQa/2q} 2^{-jL(1-a)/2} 2^{j(n_1 + n_2)/2} |\gamma_0(t)^{'}|^{-\frac{\varkappa n_2}{2}} |v'|^{-\frac{\varkappa n_2}{2}} 2^{\frac{j(1+a\epsilon)}{2}}  \right\}^{\frac{p'-2}{p'}}\\
 & \quad\quad\quad\quad\quad\quad\quad\quad \left( 2^{-jQa/q} 2^{-jL(1-a)} 2^{j(n_1 + n_2)/2}  |\gamma_0(t)^{'}|^{-\varkappa n_2} 2^{ja\epsilon} 2^{j} \right)^{\frac{1}{p'}}\, dt\\
& \lesssim_{L, \epsilon} r (2^l sr^{\theta})^{-L} \sum_{j\leq j_{0}} \left\{ 2^{-jQa/2q} 2^{-jL(1-a)/2}  2^{j(n_1 + n_2)/2} 2^{ja\epsilon/2} |\mathfrak{z}'|^{-\varkappa n_2} \right\}^{\frac{p'-2}{p'}} \\
&  \quad\quad\quad\quad\quad\quad\quad \left\{ 2^{-jQa/q} 2^{-jL(1-a)} 2^{j(n_1 + n_2)/2} 2^{ja\epsilon} |\mathfrak{z}'|^{-\varkappa n_2} \right\} ^{\frac{1}{p'}} 2^{\frac{j}{2}(\frac{p'-2}{p'}+\frac{2}{p'})} \\
& \lesssim_{L, \epsilon} r (2^l sr^{\theta})^{-L} \sum_{j\leq  j_{0}} 2^{-jQa/2q}  2^{-\frac{jL(1-a)}{2}} 2^{\frac{j(n_1 + n_2)}{2}(\frac{p'-2}{p'}+\frac{1}{p'})} 2^{\frac{j}{2}}2^{\frac{ja\epsilon}{2}} |\mathfrak{z}'|^{-\varkappa  n_2(\frac{p'-2}{p'}+\frac{1}{p'})} \\
& \lesssim_{L, \epsilon} r (2^l sr^{\theta})^{-L} \sum_{j\leq  j_{0}}   2^{-j(n_1 + n_2)a/2q} 2^{-\frac{jL(1-a)}{2}} 2^{\frac{j(n_1 + n_2)}{2p}} 2^{\frac{j}{2}} 2^{\frac{ja\epsilon}{2}} |\mathfrak{z}'|^{-\frac{\varkappa n_2}{p}}. 
\end{align*}

Now, if we take $L$ such that $L \, (1-a) < -\frac{(n_1 + n_2)a}{q} + \frac{n_1 + n_2}{p} + 1 + a\epsilon$, then the above sum is convergent, and 
\begin{align*}
I_1 & \lesssim r (2^l sr^{\theta})^{-L} 2^{j_0 \left\{ \left( -\frac{(n_1 + n_2)a}{2q}+\frac{n_1 + n_2}{2p} \right) -\frac{L(1-a)}{2} + \frac{1 + a\epsilon}{2} \right\}} |\mathfrak{z}'|^{-\frac{\varkappa n_2}{p}}\\
& \lesssim_{L, \epsilon, \theta} (2^{l})^{-L} r^{\frac{(1-a-\theta)L+(n_1 + n_2)a}{q}-\frac{n_1 + n_2}{p}-a\epsilon}|\mathfrak{z}'|^{-\frac{\varkappa n_2}{p}}.  
\end{align*}

Similarly, in estimating terms $I_2$ and $I_3$, we would require the condition $L \, (1-a) > -\frac{(n_1 + n_2)a}{q}+ \frac{n_1 + n_2}{p}$, and with that one can show that 
$$ I_2, I_3 \lesssim_{L, \epsilon, \theta} (2^{l})^{-L} r^{\frac{(1-a-\theta)L+(n_1 + n_2)a}{q}-\frac{n_1 + n_2}{p}-a\epsilon}|\mathfrak{z}'|^{-\frac{\varkappa n_2}{p}},$$
completing the proof of \eqref{ms2-HD}. 

This completes the proof of Lemma \ref{lem: kernel-difference-p}. 
\end{proof}


\subsection{Estimates for the grand maximal truncated operator} \label{subsec: estimates-grand-maximal}
Now we focus on the following pointwise estimates for the grand maximal truncated operator which is among the key ingredients for the sparse domination. 

\begin{lemma} \label{lem:staubach-Intermediate} 
Let $m\in \mathscr{S}^{-Qa/q}_{1-a, \delta}(G_{\varkappa})$ with $\sigma = - Qa/q$ and $1 \leq q < 2$ and take $s=3C_0^2 C_{1,\varkappa}$. Then, for the operator $T=m(x, G_{\varkappa})$, and for any $q<p<\infty$, we have 
$$\mathcal{M}_{T, s}^{\sharp}f(x) \lesssim_{T,s,p} \, \mathcal M_{p}f(x),$$
for every $f\in C_{c}^{\infty}(\mathbb{R}^{n_1 + n_2})$. 
\end{lemma}
\begin{proof}
Let us fix some $x\in \mathbb{R}^{n_1 + n_2}$ and a ball $B=B(\mathfrak{z}, r)$ containing $x$. Let $y, z\in B$. Denote $r_{0}=sr$. For each $l \in \mathbb{N}$, we consider $\mathcal{A}_{l} = \left\{v: 2^l r_0 \leq d(\mathfrak{z}, v) \leq  2^{l+1} r_0 \right\}.$ 

Now, we have 
\begin{align} \label{est:main-summation1-lem:staubach-Intermediate}
&|T(f\chi_{\mathbb{R}^{n_1 + n_2} \setminus B(\mathfrak{z}, r_{0})})(y) - T(f\chi_{\mathbb{R}^{n_1 + n_2} \setminus B(\mathfrak{z}, r_{0})})(z)|\\
\nonumber & \quad \leq\sum_{j=0}^{\infty} \int_{\mathbb{R}^{n_1 + n_2}\setminus B(\mathfrak{z}, r_{0})}|K_j(y,v)-K_j(z,v)||f(v)| \, dv \\ 
\nonumber & \quad \lesssim \sum_{j=0}^{\infty} \sum_{l = 0}^{\infty}\int_{\mathcal{A}_{l}}|K_j(y, v) - K_j(z, v)||f(v) |\, dv
\end{align}

We need to consider the following two cases for $r_0$. 

\medskip \noindent \textbf{\underline{Case 1 $(r_{0} \geq 1)$:}} First of all, we have 
\begin{align*}
\int_{\mathcal{A}_{l}}|K_j(y,v )-K_j(z,v)||f (v)|\, dv \leq \int_{\mathcal{A}_{l}}|K_j(y, v)||f(v)| \, dv + \int_{\mathcal{A}_{l}}|K_j(z, v)| |f(v)|\, dv. 
\end{align*}
Estimation of the above two terms is similar, so we only pursue the first one. 

As also seen in the previous section, we perform separate analysis in two regions for $l$ as follows. 

First, let $l$ be such that $ 2^{l+1}r_0 \geq \frac{1}{2} |\mathfrak{z}'| $, then using \eqref{cond:General-hypo-sup} we proceed as follows. Choose  $\epsilon_1 = \frac{aQ}{1-a}(1-\frac{1}{q})$. Then, 
\begin{align*}
& \int_{\mathcal{A}_{l}}|K_j(y, v)| |f(v)| \, dv \\
& = \int_{\mathcal {A}_{l}} d({y}, v)^{-(Q + \frac{1}{2} + \epsilon_1)}d({y}, v)^{Q + \frac{1}{2} + \epsilon_1 }|K_j({y}, v)| |f(v)|\, dv \\
& \lesssim_{\epsilon} \int_{\mathcal{A}_{l}} (2^{l} r_0)^{-(Q+ \frac{1}{2} + \epsilon_1)} 2^{-jQa/2q} 2^{-j(\frac{Q}{2}+\frac{1}{4} + \frac{\epsilon_1}{2})(1-a)} 2^{ja\epsilon /2} \frac{1}{|B(x, 2^{-j/2})|^{1/2}}\frac{|f(v)|}{|B(y, 2^{-j/2})|^{1/2}}\ dv\\
& \lesssim_{\epsilon} 2^{-j(1-a)/4} (2^{l} r_0)^{-(Q + \frac{1}{2} + \epsilon_1 )} 2^{aj\epsilon/2}  |B(\mathfrak{z},2^{l} r_0)| \mathcal{M}f(x)\\
& \lesssim_{s,\epsilon} 2^{-j(1-a)/4} 2^{ja\epsilon /2} \frac{(2^{l}  r_0)^{Q}}{(2^{l} r_0)^{(Q + \frac{1}{2} + \epsilon_1)}} \\ 
& \lesssim_{s,\epsilon} 2^{-j(1-a)/8} \frac{1}{2^{l/2}} \mathcal{M}f(x),
\end{align*}
where in the second last inequality we have chosen $\epsilon>0$ such that $\epsilon < \frac{1-a}{4a}$ and the last inequality follows from the fact that $r_{0}\geq 1$. 

Next, if $l$ is such that $ 2^{l+1}r_0 < \frac{1}{2}|\mathfrak{z}'|$, then we proceed as follows. Using \eqref{cond:General-hypo-sup} with $\mathfrak{r}=n_1 + n_2+\frac{1}{2}+\epsilon_{2}$ where $\epsilon_2 = \frac{a(n_1 + n_2)}{1-a}(1-\frac{1}{q}) $, we obtain 
\begin{align*}
 \int_{\mathcal{A}_{l}}|K_j(y, v)| |f(v)| & \, dv \lesssim_{s, \epsilon} 2^{-j(1-a)/4} 2^{ja\epsilon /2} \frac{(2^{l} r_0)^{n_1 + n_2}}{(2^{l} r_0)^{(n_1 + n_2 + \frac{1}{2} + \epsilon_2)}}|\mathfrak{z}'|^{\varkappa n_{2}} |\mathfrak{z}'|^{-\varkappa n_2} \mathcal{M}f(x) \\
 & \lesssim_{s,\epsilon} 2^{-j(1-a)/8} \frac{1}{2^{l/2}} \mathcal{M}f(x),
\end{align*}
provided that $0 < \epsilon < \frac{1-a}{4a}$. 

Putting the above two estimates in \eqref{est:main-summation1-lem:staubach-Intermediate}, we obtain 
\begin{align*}
|T(f \chi_{\mathbb{R}^{n_1 + n_2} \setminus B(\mathfrak{z}, r_{0})}(y) - T(f \chi_{\mathbb{R}^{n_1 + n_2} \setminus B(\mathfrak{z}, r_{0})}(z)| & \lesssim_{s,\epsilon} \sum_{j = 0}^\infty \sum_{l = 0}^\infty 2^{-j(1-a)/8} \frac{1}{2^{l/2}} \mathcal{M}f(x) \\ 
& \lesssim_{s,\epsilon} \mathcal{M}f(x),
\end{align*}
completing the proof of Lemma \ref{lem:staubach-Intermediate} in the case of $r_{0} \geq 1$. 

\medskip \noindent \textbf{\underline{Case 2 $(r_{0} < 1)$:}} Again, we consider two regions of $l$. 

To start with, let $l_{0} \geq 0$ be such that $\frac{1}{2}|\mathfrak{z}'|\leq 2^{l+1} r_{0}$ for all $l > l_{0}$. Then, 
\begin{align*}
|T(f\chi_{\mathbb{R}^{n_1 + n_2} \setminus B(\mathfrak{z}, r_{0})})(y) - T(f\chi_{\mathbb{R}^{n_1 + n_2}  \setminus B(\mathfrak{z}, r_{0})})(z)| \leq \mathcal{J}_{1} + \mathcal{J}_{2},
\end{align*}
where,
\begin{align*}
& \mathcal{J}_{1}:= \sum_{l > l_0} \sum_{j=0}^{\infty} \int_{\mathcal{A}_{l}}|K_j(y, v)-K_j(z, v)||f(v)|\, dv,\\
& \mathcal{J}_{2}:= \sum_{l = 0}^{l_{0}}\sum_{j=0}^{\infty} \int_{\mathcal{A}_{l}}|K_j(y, v)-K_j(z, v)||f(v)|\, dv.
\end{align*}

\medskip \underline{Estimate for $\mathcal{J}_{1}$:}
For $l> l_{0}$, we proceed as follows. 

Using \eqref{ms1-HD} with $L>0$ such that $-\frac{Qa}{q} + \frac{Q}{p} < L \, (1-a) < -\frac{Qa}{q} + \frac{Q}{p} + 1 + a\epsilon$, we obtain 
\begin{align*}
& \sum_{j=0}^{\infty}\int_{\mathcal{A}_{l}}|K_j(y, v)-K_j(z, v)||f(v)|\, dv\\
& \lesssim \sum_{j=0}^{\infty}\left(\int_{\mathcal{A}_{l}}|K_j (y, v)-K_j(z, v)|^{p'}\, dv \right)^{1/p'} \left(\int_{\mathcal{A}_{l}} |f(v)|^{p} \, dv \right)^{1/p}\\
& \lesssim_{\epsilon, p} (2^{l})^{- L} r_{0}^{-La+\frac{Qa}{q}-\frac{Q}{p}-a\epsilon}|B(\mathfrak{z}, 2^{l} r_{0})|^{1/p} \mathcal{M}_{p}f(x)\\
& \lesssim_{\epsilon, p} r_{0}^{-a(L-\frac{Q}{q}+\epsilon)} 2^{-l(L-\frac{Q}{p})}\mathcal{M}_{p}f(x),
\end{align*}
and therefore 
\begin{align} \label{est1:lem:staubach-Intermediate-r0<1}
\mathcal{J}_{1}\lesssim_{\epsilon, p} \mathcal{M}_{p}f(x) \sum_{l\geq l_{0}} r_{0}^{-a(L - \frac{Q}{q} + \epsilon)} 2^{-l(L-\frac{Q}{p})}.
\end{align}

It is straightforward to see that the infinite sum over $l > l_0$ in \eqref{est1:lem:staubach-Intermediate-r0<1} converges and will be bounded by constant independent of $r_{0}$ provided $\frac{Q}{p} < L < \frac{Q}{q}$. Actually, we need $L<\frac{Q}{q}-\epsilon$, but that can also be ensured by assuming $ L < \frac{Q}{q}$ and choosing $\epsilon$ suitably small. Altogether, we have the following conditions on $L$: 
\begin{align*} 
\max \left\{ \frac{Q}{p}, \frac{-\frac{Qa}{q} + \frac{Q}{p}}{1-a} \right\} < L < \min \left\{ \frac{Q}{q}, \frac{-\frac{Qa}{q} + \frac{Q}{p}}{1-a} + \frac{1}{1-a} \right\}. 
\end{align*}

But, since $p > q$, we already have $\frac{Q}{p} < \frac{Q}{q}$ and $\frac{\frac{-Qa}{q} + \frac{Q}{p}}{1-a} \leq \frac{Q}{p}$, and therefore the only condition to be verified is the following:
\begin{align*}
\frac{Q}{p} < \frac{\frac{-Qa}{q} + \frac{Q}{p}}{1-a} + \frac{1}{1-a} \iff 1 + a Q \left( \frac{1}{p} - \frac{1}{q} \right) > 0,
\end{align*} 
which is possible by choosing $p > q$ sufficiently close to $q$.

\medskip {\underline{Estimate for $\mathcal{J}_{2}$}:}
By definition of $l_0$, we have that $\frac{1}{2}|\mathfrak{z}'|\geq 2^{l + 1} r_{0}$ for all $l \leq l_{0}$. This time, we use \eqref{ms2-HD} with $L>0$ such that $-\frac{(n_{1}+n_{2})a}{q} + \frac{n_1 + n_2}{p} < L \, (1-a) < -\frac{(n_{1}+n_{2})a}{q} + \frac{n_{1}+ n_{2}}{p} + 1 + a\epsilon$, to get 
\begin{align*} 
& \sum_{j=0}^{\infty}\int_{\mathcal{A}_{l}}|K_j(y, v)-K_j(z, v)||f(v)|\, dv\\
&\lesssim \sum_{j=0}^{\infty}\left(\int_{\mathcal{A}_{l}}|K_j (y, v)-K_j(z, v)|^{p'}\, dv \right)^{1/p'} \left(\int_{\mathcal{A}_{l}} |f(v)|^{p} \, dv \right)^{1/p}\\
& \lesssim_{\epsilon, p} (2^{l})^{-L} r^{-La + \frac{(n_{1}+n_{2})a}{q} - \frac{n+1}{p}-a\epsilon} |\mathfrak{z}'|^{-\varkappa n_2/p}|B(\mathfrak{z}, 2^{l} r_{0})|^{1/p} \mathcal{M}_{p}f(x)\\
& \lesssim_{\epsilon, p} (2^{l})^{-L} r^{-La + \frac{(n_{1}+n_{2})a}{q} - \frac{n_1 + n_2}{p}-a\epsilon} |\mathfrak{z}'|^{-\varkappa n_2/p} 
(2^l r_0)^{\frac{n_1 + n_2}{p}} |\mathfrak{z}'|^{\varkappa n_2/p} |B(\mathfrak{z}, 2^{l} r_{0})|^{1/p} \mathcal{M}_{p}f(x)\\
& \lesssim_{\epsilon, p} r_{0}^{-a(L-\frac{(n_{1}+n_{2})}{q}+\epsilon)} 2^{-l(L-\frac{n_1 + n_2}{p})}\mathcal{M}_{p}f(x), 
\end{align*}
and therefore 
\begin{align} \label{est2:lem:staubach-Intermediate-r0<1}
\mathcal{J}_{2} \lesssim_{\epsilon, p} \mathcal{M}_{p}f(x) \sum_{l \leq l_{0}} r_{0}^{-a(L-\frac{(n_{1}+n_{2})}{q}+\epsilon)} 2^{-l(L-\frac{n_1 + n_2}{p})}.
\end{align}

The sum over $l \leq l_0$ in \eqref{est2:lem:staubach-Intermediate-r0<1} will be bounded by constant independent of $r_{0}$ provided $\frac{n_1 + n_2}{p} < L < \frac{n_{1}+n_{2}}{q}.$ Combining the conditions on $L$ we need to ensure the following
 \begin{align*}
\max \left\{ \frac{n_1 + n_2}{p}, \frac{-\frac{(n_{1}+n_{2})a}{q} + \frac{n_1 + n_2}{p}}{1 - a} \right\} < L < \min \left\{ \frac{n_1 + n_2}{q}, \frac{-\frac{(n_{1}+n_{2})a}{q} + \frac{n_1 + n_2}{p}}{1 - a}+\frac{1}{1 - a} \right\}. 
\end{align*}
As earlier, one notices that the two conditions are simultaneously satisfied if  
$$1 + a(n_1 +n_2) \left( \frac{1}{p} - \frac{1}{q}\right) > 0,$$
which is possible by choosing $p > q$ sufficiently close to $q$.

This completes the proof of the Lemma \ref{lem:staubach-Intermediate}.
\end{proof}


\subsection{Proof of Theorem \ref{thm:Main-Sparse-Result}} \label{subsec:Proof-thm:Main-Sparse-Result}

We are now in a position to prove Theorem \ref{thm:Main-Sparse-Result}. 

\begin{proof}[Proof of Theorem \ref{thm:Main-Sparse-Result}] 
We divide the proof into two parts. 

Let us first prove it for the operator $T = m(x, G_{\varkappa})$ with $m \in \mathscr{S}^{-Qa}_{1 - a, \delta}(G_{\varkappa})$. It follows from Theorem \ref{thm:unweightedLp:pseudo} that $T$ is bounded from $L^p(\mathbb{R}^{n_1 + n_2})$ to itself for $1<p<2$. Also, we obtain from Lemma \ref{lem:staubach-Intermediate} that 
$$ \mathcal{M}_{T, s}^{\sharp}f(x) \lesssim_{T,p} \, \mathcal M_{p}f(x),$$
for every for every $1 < p < \infty$, $f \in C_{c}^{\infty}(\mathbb{R}^{n_1 + n_2})$, and $s=3C_{0}^2 C_{1, \varkappa}$. 

The above inequality implies that $\mathcal{M}_{T, s}^{\sharp}$ is weak-type $(p, p)$ for all $1<p<\infty$. Therefore, the sparse domination principle (Theorem 1.1 in \cite{Lorist-pointwaise-sparse2021}) is applicable, ensuring that for each $1<r<\infty$ and for each $f\in C_{c}^{\infty}(\mathbb{R}^{n_1 + n_2})$, there exist a sparse family $S$ such that 
\begin{align*}
|Tf(x)| \lesssim \mathcal{A}_{r, S}f(x)
\end{align*}
for almost every $x\in \mathbb{R}^{n_1 + n_2}$.

The proof of the sparse domination result for the operator $T = m(x, G_{\varkappa})$ with $m \in \mathscr{S}^{-Qa/q}_{1 - a, \delta}(G_{\varkappa})$ and $1 < q < 2$, is similar to that of the previous case. Observe that Theorem \ref{thm:unweighted-Lq-intermediate-class} implies that $T$ is bounded from $L^q(\mathbb{R}^{n_1 + n_2})$ to itself and an application of Lemma~\ref{lem:staubach-Intermediate} implies that  $\mathcal{M}_{T, s}^{\sharp}$ is weak-type $(p, p)$ for all $q<p<\infty$. Hence, the proof follows from Theorem 1.1 of \cite{Lorist-pointwaise-sparse2021}.
\end{proof}

\begin{remark} \label{rem:about-chanillo-type-class}
As mentioned in Section \ref{subsec:about-chanillo-type-class}, our techniques of establishing sparse domination do not extend to classes $ \mathscr{S}^{-Qa/2}_{1-a, \delta}(G_{\varkappa})$. This is because even if one proves refined weighted Plancherel estimates (of Section \ref{sec:Kernel-estimates}) with the removal of the extra $\epsilon > 0$ from the order of differentiability in the Sobolev norm of symbol functions, we note that with kernel estimates of Lemma \ref{lem: kernel-difference-p} we can not work in the case of $ \mathscr{S}^{-Qa/q}_{1-a, \delta}(G_{\varkappa})$ when $q = 2$. For example, to get the convergence of the infinite sum over $l > l_0$ in \eqref{est1:lem:staubach-Intermediate-r0<1} (or the sum over $l \leq l_0$ in \eqref{est2:lem:staubach-Intermediate-r0<1}), we need to have $p > q$. But, estimates \eqref{ms1-HD} and \eqref{ms2-HD} of Lemma \ref{lem: kernel-difference-p} are valid only for $p \leq 2$. So, with these estimates at hand, we are forced to restrict ourselves to the analysis of classes $ \mathscr{S}^{-Qa/q}_{1-a, \delta}(G_{\varkappa})$ with $q < 2$.
\end{remark}


\section{Weighted boundedness result for \texorpdfstring{$\mathscr{S}^{-Qa/2}_{1-a, \delta}(\boldsymbol{L}, \boldsymbol{U})$}{} } \label{sec:joint-funct-calculus-weighted-boundedness}

In this section, we shall study weighted boundedness for symbol classes $\mathscr{S}^{-Qa/2}_{1-a, \delta}(\boldsymbol{L}, \boldsymbol{U})$ with the help of the Fefferman--Stein sharp maximal function, which is defined as follows. For a locally integrable function $f$, 
$$\mathcal{M}^{\sharp}f(x):=\sup_{x\in B}\frac{1}{|B|}\int_{B}|f-f_{B}|,$$
where $f_{B}:=\frac{1}{|B|}\int_{B}f$ denotes the average of $f$ over the ball $B$.  

We begin with stating the following kernel estimates which will be crucial for the proof of Theorem \ref{thm:joint-fun-Channilo-type}. The proof follows arguing exactly as in the proof of Lemma \ref{lem: kernel-difference-p}, and we leave those details. As earlier, take $s=3C_0^2 C_{1,\varkappa}$. 

\begin{lemma} \label{lem:kernel-estimate-channilo-type-joint-funct}
Let $m \in \mathscr{S}^{-Qa/2}_{1-a, \delta}(\boldsymbol{L}, \boldsymbol{U}) $ be such that it satisfies condition \eqref{def:grushin-symb-vanishing-0-condition} for all $|\beta| \leq \lfloor \frac{Q}{2} + \frac{1}{1-a} \rfloor + 1$. Let $B=B(\mathfrak{z}, r)$ be any ball with $0< r \leq 1$. Then for any ${y\in B}$,
\begin{align} \label{lem:kernel-estimate-difference-1} 
\sum_{j = 0}^{\infty} \left(\int_{2^{l} s r^{1-a} \leq d(v, \mathfrak{z}) \leq  2^{l+1} s r^{1-a} }| K_j (y,v)-K_j(\mathfrak{z},v)|^{2} \, dv \right)^{1/2} \lesssim_{L} ( 2^l r^{1-a})^{-L} d(y, \mathfrak{z})^{(1-a) (L - \frac{Q}{2})}, 
\end{align} 
whenever \, $\frac{Q}{2} < L < \frac{Q}{2} + \frac{1}{(1-a)}$, and  
\begin{align} \label{lem:kernel-estimate-difference-2} 
& \sum_{j = 0}^{\infty} \left(\int_{ 2^{l} s r^{1-a} \leq d(v,\mathfrak{z}) \leq  2^{l+1} s r^{1-a} } | K_j (y,v)-K_j(\mathfrak{z}, v)|^{p'} \, dv  \right)^{1/p'}  \\
\nonumber & \quad \lesssim_{L} (2^{l}r^{1-a})^{-L} d(y, \mathfrak{z})^{(1-a) (L-\frac{(n_1 + n_2)}{2})} |\mathfrak{z}'|^{-\varkappa n_{2}/2},
\end{align}
whenever \, $\frac{n_1 + n_{2}}{2} < L < \frac{n_1 + n_{2}}{2} + \frac{1}{(1-a)} $ and $\frac{1}{2}|\mathfrak{z}'| \geq 2^{l}sr^{1-a}$.
\end{lemma}

We shall establish pointwise domination of the Fefferman--Stein maximal function in Subsection \ref{subsec:Proof-thm:Chanilloy-type}. For the same, we need the following $L^2$-boundedness result, which we state and prove here. 
\begin{theorem} \label{thm:Compo-operator-L2-boundedness}
Let $m \in \mathscr{S}^{-Qa/2}_{1-a, 1-a}(\boldsymbol{L}, \boldsymbol{U})$ satisfies condition \eqref{def:grushin-symb-vanishing-0-condition} for all $|\beta| \leq 4 \left( \lfloor\frac{Q}{4}\rfloor + 1 \right)$. Then, the operator $(I + G)^{Qa/4} m(x,\boldsymbol{L}, \boldsymbol{U})$ is $L^2$-bounded. 
\end{theorem}
\begin{proof}
Note first that in the case of the Euclidean pseudo-differential operators, if we have $m \in \mathscr{S}^{-Qa/2}_{1-a, 1-a}(\Delta)$ then $m(x, \Delta)^*$ is also a pseudo-differential operator with symbol from the same class $\mathscr{S}^{-Qa/2}_{1-a, 1-a}(\Delta)$. Therefore, $m(x, \Delta)^* (I + \Delta)^{Qa/4}$ is a pseudo-differential operator with symbol from the class $\mathscr{S}^{0}_{1-a, 1-a}(\Delta)$, and thus it is $L^2$-bounded. Finally, since $\left\| (I + \Delta)^{Qa/4} m(x, \Delta) \right\|_{op} = \left\| \left( (I + \Delta)^{Qa/4} m(x, \Delta) \right)^* \right\|_{op} = \left\| m(x, \Delta)^* (I + \Delta)^{Qa/4} \right\|_{op}$, it follows that $(I + \Delta)^{Qa/4} m(x, \Delta)$ is $L^2$-bounded. 

In the context of the Grushin operator, we need not have the adjoint operator $m(x,\boldsymbol{L}, \boldsymbol{U})^*$ to be a pseudo-multiplier operator, therefore we do not have a direct argument to conclude the theorem. We therefore write a more detailed analysis to argue our claim. 

Given $m \in \mathscr{S}^{-Qa/2}_{1-a, 1-a}(\boldsymbol{L}, \boldsymbol{U})$, let us write $T =  m(x,\boldsymbol{L}, \boldsymbol{U})$ and $T_1 = \widetilde{m}(x,\boldsymbol{L}, \boldsymbol{U})$ where $\widetilde{m} (x, \tau, \kappa) = m(x, \tau, \kappa) (1 + |\tau|^2 + |\kappa|^2)^{Qa/8}$. It can be easily verified that $\widetilde{m} \in \mathscr{S}^{0}_{1-a, 1-a}(\boldsymbol{L}, \boldsymbol{U})$, with the symbol seminorm of $\widetilde{m}$ controlled by that of $m$. Furthermore, let us denote by $\mathcal{P}_a$ the spectral multiplier operator in the joint functional calculus of $(\boldsymbol{L}, \boldsymbol{U})$, with symbol $(\tau, \kappa) \mapsto (1 + |\tau|_1)^{Qa/4} (1 + |\tau|^2 + |\kappa|^2)^{-Qa/8}$. Now, 
\begin{align*}
(I + G)^{Qa/4} T = (I + G)^{Qa/4} T_1 (I + G)^{-Qa/4} \mathcal{P}_a, 
\end{align*}
and since the operator $\mathcal{P}_a$ is $L^2$-bounded (by Plancherel theorem), in order to show that $(I + G)^{Qa/4} T$ is $L^2$-bounded, it suffices to prove that $(I + G)^{Qa/4} T_1 (I + G)^{-Qa/4}$ is $L^2$-bounded. 

For any $b>0$, let us consider the Sobolev space $H^{b, 2}(\mathbb{R}^{n_1 + n_2}) = (I + G)^{b/2} L^2 (\mathbb{R}^{n_1 + n_2})$. Using the spectral resolution of $G$, we have 
\begin{align*}
H^{b, 2}(\mathbb{R}^{n_1 + n_2}) = \left\{ f \in L^2(\mathbb{R}^{n_1 + n_2}) : \left\| (I + G)^{b/2}f \right\|_{L^2} < \infty \right\}, 
\end{align*}
where 
$$ \left\| (I + G)^{b/2}f \right\|^2_{L^2} = \int_{R^{n_2}} \sum_{\mu} (1 + (2|\mu| + n_1)|\lambda|)^{b} \left| (f^{\lambda}, \Phi^{\lambda}_{\mu} ) \right|^2 \, d\lambda. $$

It is easy to observe that $(I + G)^{Qa/4} T_1 (I + G)^{-Qa/4} \in \mathcal{B} \left( L^2(\mathbb{R}^{n_1 + n_2}) \right)$ if and only if $T_1 \in \mathcal{B} \left( H^{\frac{Qa}{2}, 2}(\mathbb{R}^{n_1 + n_2}) \right)$. Now, if we can prove that $T_1 \in \mathcal{B} \left( H^{2N, 2}(\mathbb{R}^{n_1 + n_2}) \right)$ for each $N \in \mathbb{N}_{+}$, then we can invoke the interpolation theorem for operators with change of measures (see, for example, Section 5 in \cite{Stein-Weiss-Fourier-analysis-book}) to conclude that $T_1 \in \mathcal{B} \left( H^{b, 2}(\mathbb{R}^{n_1 + n_2}) \right)$ for each $b>0$, and our claim would then follow by taking $b = \frac{Qa}{2}$. So, we are left with showing that $T_1 \in \mathcal{B} \left( H^{2N, 2}(\mathbb{R}^{n_1 + n_2}) \right)$, which is equivalent to proving that $(I + G)^{N} T_1 (I + G)^{-N} \in \mathcal{B} \left( L^2 (\mathbb{R}^{n_1 + n_2}) \right)$, and again it is equivalent to proving that $(I + G)^{N} T_1 \circ m_{N, 0} (\boldsymbol{L}, \boldsymbol{U}) \in \mathcal{B} \left( L^2 (\mathbb{R}^{n_1 + n_2}) \right)$, where $m_{N, 0} (\tau, \kappa) = (1 + |\tau|_1^2 + |\kappa|^2)^{-N/2}$. 

It is enough to show that $G^{N_1} T_1 \circ m_{N, 0} (\boldsymbol{L}, \boldsymbol{U}) \in \mathcal{B} \left( L^2 (\mathbb{R}^{n_1 + n_2}) \right)$ for all $0 \leq N_1 \leq N$. Recall that $G = - \sum_{j=1}^{n_1} X_j^2 - \sum_{j=1}^{n_1} \sum_{k=1}^{n_2} X_{j,k}^2$, where $X_j= \frac{\partial}{\partial x'_j}$, $X_{j,k}= x'_j \frac{\partial}{\partial x_k''}$. So, we shall be done if we could prove that $X^{\Gamma} T_1 \circ m_{N, 0} (\boldsymbol{L}, \boldsymbol{U}) \in \mathcal{B} \left( L^2 (\mathbb{R}^{n_1 + n_2}) \right)$, where $|\Gamma| \leq 2N$ and $X = (X_j, X_{j, k})_{j, k}$. 

With $f \in S(\mathbb{R}^{n_1 + n_2}) $, using Leibniz formula we can write 
\begin{align*} 
& X^{\Gamma} T_1 \circ m_{N, 0} (\boldsymbol{L}, \boldsymbol{U}) f (x) \\ 
\nonumber & \quad = \sum_{\Gamma_1 + \Gamma_2 = \Gamma} C_{\Gamma_1, \Gamma_2} \int_{R^{n_2}} \sum_{\mu} \frac{X^{\Gamma_1} \widetilde{m} \left( x, (2 \mu + \tilde{1}) |\lambda|, \lambda \right)}{\left\{ 1 + (2|\mu|+n_1)^2 |\lambda|^2 + |\lambda|^2 \right\}^{N/2}} (f^{\lambda}, \Phi^{\lambda}_{\mu} ) \, X^{\Gamma_2} \left\{ \Phi^{\lambda}_{\mu}(x') e^{-i\lambda \cdot x''} \right\} d\lambda. 
\end{align*} 

Let us carefully look at the action of $X^{\Gamma_2}$ on $\Phi^{\lambda}_{\mu}(x') e^{-i\lambda \cdot x''}$. In doing so, we shall make use of the known properties of the annihilation operators $A_j (\lambda) = \frac{\partial}{\partial x'_j} + |\lambda| x'_j$, and the creation operators $A_j (\lambda)^* = - \frac{\partial}{\partial x'_j} + |\lambda| x'_j$, namely, 
\begin{align*} 
A_j(\lambda) \Phi_{\mu}^{\lambda} = \left( (2 \mu_j) |\lambda| \right)^{1/2} \Phi_{\mu - e_j}^{\lambda} \quad \textup{and} \quad A_j(\lambda)^* \Phi_{\mu}^{\lambda} = \left( (2 \mu_j + 2) |\lambda|\right)^{1/2} \Phi_{\mu + e_j}^{\lambda}. 
\end{align*} 
Since $X_j \left\{ \Phi^{\lambda}_{\mu}(x') e^{-i\lambda \cdot x''} \right\} = \frac{\partial}{\partial x'_j} \left\{ \Phi^{\lambda}_{\mu}(x') e^{-i\lambda \cdot x''} \right\} = \frac{1}{2} \left( A_j(\lambda) - A_j(\lambda)^* \right) \left\{ \Phi^{\lambda}_{\mu}(x') e^{-i\lambda \cdot x''} \right\}$, and $X_{j,k} \left\{ \Phi^{\lambda}_{\mu}(x') e^{-i\lambda \cdot x''} \right\} = \lambda_k x'_j \left\{ \Phi^{\lambda}_{\mu}(x') e^{-i\lambda \cdot x''} \right\} = \frac{\lambda_k}{2 |\lambda|} \left( A_j(\lambda) + A_j(\lambda)^* \right) \left\{ \Phi^{\lambda}_{\mu}(x') e^{-i\lambda \cdot x''} \right\}$, a successive application of gradient fields $X_{j}$'s and $X_{j,k}$'s implies that $X^{\Gamma_2} \left\{ \Phi^{\lambda}_{\mu}(x') e^{-i\lambda \cdot x''} \right\}$ can be expressed as a finite linear combination of terms of the form 
$$ \mathfrak{H}_{\Gamma_2} (\mu) \frac{\lambda^{\Gamma_3}}{|\lambda|^{|\Gamma_3|}} \left\{ (2|\mu| + n_1)|\lambda| \right\}^{\frac{|\Gamma_2|}{2}} \Phi^{\lambda}_{\mu +\tilde{\mu}}(x'), $$ 
where $|\tilde{\mu}| \leq |\Gamma_2|$, $|\Gamma_3| \leq |\Gamma_2|$, and $\mathfrak{H}_{\Gamma_2} (\mu)$ is a bounded function of $\mu$.

Summarising, we get that $X^{\Gamma} T_1 \circ m_{N, 0} (\boldsymbol{L}, \boldsymbol{U}) f (x)$ can be expressed as a finite linear combination of 
\begin{align*}
& \int_{R^{n_2}} \sum_{\mu} \frac{\left\{ X^{\Gamma_1} \widetilde{m} \left( x, (2 (\mu - \tilde{\mu}) + \tilde{1}) |\lambda|, \lambda \right) \right\} \left\{ (2|\mu - \tilde{\mu}| + n_1)|\lambda| \right\}^{\frac{|\Gamma_2|}{2}}}{\left\{ 1 + (2|\mu - \tilde{\mu}|+n_1)^2 |\lambda|^2 + |\lambda|^2 \right\}^{N/2}} \\ 
& \quad \quad \quad \quad \left( \left( \widetilde{\mathcal{T}}_{\Gamma_2, \Gamma_3} f \right)^{\lambda}, \Phi^{\lambda}_{\mu} \right) \Phi^{\lambda}_{\mu}(x') e^{-i\lambda \cdot x''} \, d\lambda, 
\end{align*}
where $\widetilde{\mathcal{T}}_{\Gamma_2, \Gamma_3}$ is the operator on 
defined by $ \left( \left( \widetilde{\mathcal{T}}_{\Gamma_2, \Gamma_3} f \right)^{\lambda}, \Phi^{\lambda}_{\mu} \right) = \mathfrak{H}_{\Gamma_2} (\mu - \tilde{\mu}) \frac{\lambda^{\Gamma_3}}{|\lambda|^{|\Gamma_3|}} \left( f^{\lambda}, \Phi^{\lambda}_{\mu - \tilde{\mu}} \right)$. 

While the $L^2$-boundedness of $\widetilde{\mathcal{T}}_{\Gamma_2, \Gamma_3}$ follows from Plancherel's theorem, it is straightforward to verify that the symbol function $M(x, \tau, \kappa) \in \mathscr{S}^0_{1-a, 1-a}(\boldsymbol{L}, \boldsymbol{U})$, where
$$ M \left( x, \tau, \kappa \right) = \left\{ X^{\Gamma_1} \widetilde{m} \left( x, \tau, \kappa \right) \right\} |\tau|_1^{\frac{|\Gamma_2|}{2}} \left\{ 1 + |\tau|_1^2 + |\kappa|^2 \right\}^{-N/2},$$
and it therefore boils down to analysing operators $T_{\tilde{\mu}}$, where 
\begin{align} \label{est:pseudo-mult-operator-with-shift}
T_{\tilde{\mu}} f(x) & = \int_{R^{n_2}} \sum_{\mu} M \left( x, (2 (\mu - \tilde{\mu}) + \tilde{1}) |\lambda|, \lambda \right) \left(f^{\lambda}, \Phi^{\lambda}_{\mu} \right) \Phi^{\lambda}_{\mu}(x') e^{-i\lambda \cdot x''} \, d\lambda. 
\end{align}
But, the same holds true by Lemma \ref{lem:shifted-CV}, implying that for all $|\tilde{\mu}| \leq 2N$, 
$$\|T_{\tilde{\mu}}\|_{op} \lesssim_N \|M\|_{\mathscr{S}^0_{1-a, 1-a}} \lesssim_N \|m\|_{\mathscr{S}^{-Qa/2}_{1-a, 1-a}}.$$

This completes the proof of Proposition \ref{prop:weighted-CV}. 
\end{proof}


\subsection{Pointwise domination of the Fefferman--Stein maximal function} \label{subsec:Proof-thm:Chanilloy-type}

The following estimate is, in fact, the heart of the proof of Theorem \ref{thm:joint-fun-Channilo-type}.
\begin{theorem} \label{thm:Chanilloy-type}
Let $m \in \mathscr{S}^{-Qa/2}_{1-a, \delta}(\boldsymbol{L}, \boldsymbol{U}) $ be such that it satisfies condition   \eqref{def:grushin-symb-vanishing-0-condition} for all $|\beta| \leq \lfloor \frac{Q}{2} + \frac{1}{1-a} \rfloor + 1$. Then, for the operator $T=m(x, \boldsymbol{L}, \boldsymbol{U})$, we have 
$$ \mathcal{M}^{\sharp} (T f)(x) \lesssim_{T,s} \mathcal M_{2}f(x),$$ 
for every $f\in C_{c}^{\infty}(\mathbb{R}^{n_1 + n_2})$ and almost every $x \in \mathbb{R}^{n_1 + n_2}$. 
\end{theorem}
\begin{proof} 
Fix $f\in C_{c}^{\infty}(\mathbb{R}^{n_1 + n_2})$, a point $x\in \mathbb{R}^{n_1 + n_2}$ and a ball $B=B(\mathfrak{z}, r)$ containing $x$. As earlier, we take $s=3C_0^2 C_{1,\varkappa}$ and consider the two cases in $r$. 

\medskip \noindent \textbf{\underline{Case 1 $(r \geq 1)$}:} Writing $sB=B(\mathfrak{z}, sr)$, we decompose $f=f_{1}+f_{2}$ where $f_{1} = f\chi_{sB}$. 

Now, 
\begin{align*}
\frac{1}{|B|} \int_{B} |Tf(y) - Tf_{2}(\mathfrak{z})| \, dy \leq \tilde{I}_1 + \tilde{I}_2,
\end{align*} 
where $$\tilde{I}_1 = \frac{1}{|B|} \int_{B} |Tf_{1}(y)| \, dy \quad \text{and} \quad \tilde{I}_2 = \frac{1}{|B|}\int_{B}|Tf_{2}(y)-Tf_{2}(\mathfrak{z})| \, dy.$$ 

First, we consider the term $\tilde{I}_1$. Using H\"older's inequality and the $L^2$-boundedness of the operator $T$ we get
\begin{align*}
\tilde{I}_1 \leq \frac{1}{|B|^{\frac{1}{2}}} \left(\int_{B} |Tf_{1}(y)|^2 \, dx \right)^{\frac{1}{2}} & \lesssim_{T} \frac{1}{|B|^{\frac{1}{2}}} \left(\int_{sB}|f_{1}(y)|^2 \, dy \right)^{\frac{1}{2}} \lesssim  \frac{|sB|^{\frac{1}{2}}}{|B|^{\frac{1}{2}}} \mathcal{M}_{2}f(x)  \lesssim_{s} \mathcal{M}_2f(x).
\end{align*}

Next, we estimate $\tilde{I}_2$. For this part, we follow the same line of arguments as in Case 1 of the proof of Lemma \ref{lem:staubach-Intermediate}. For $l\in \mathbb{N}$, we write $\mathcal{A}_{l}= \left\{v: 2^l sr \leq d(\mathfrak{z}, v) \leq  2^{l+1}sr \right\}.$ 
Then
\begin{align*}
| Tf_{2}(y) - Tf_{2} (\mathfrak{z})| \leq  \sum_{j=0}^{\infty} \sum_{l=0}^{\infty} \int_{\mathcal{A}_{l}} |K_j(y,v) - K_j (\mathfrak{z}, v)| |f_{2}(v)| \, dv. 
\end{align*}

Now, for those $l$ for which $ 2^{l+1} sr \geq \frac{1}{2}| \mathfrak{z}'| $ holds, we make use of condition \eqref{cond:joint-functional-kernel-hypo} with $ r_0 = \lfloor Q/2 \rfloor + 1 $, to obtain
\begin{align*}
& \int_{\mathcal{A}_{l}} |K_j(y,v) - K_j (\mathfrak{z}, v)| |f_{2}(v)| \, dv \\
& \quad \lesssim \left( \int_{\mathcal{A}_{l}} d(y,v)^{Q+1} |K_j (y,v)|^{2} \, dv \right)^{\frac{1}{2}} \times \left( \int_{\mathcal{A}_{l}} \frac{1}{ d(y,v)^{Q + 1}} |f_{2}(v)|^{2} \, dv \right)^{\frac{1}{2}} \\
& \quad \quad + \left( \int_{\mathcal{A}_{l}} d(\mathfrak{z}, v)^{Q + 1} |K_j (\mathfrak{z}, v)|^{2} \, dv \right)^{\frac{1}{2}} \times \left( \int_{\mathcal{A}_{l}} \frac{1}{ d(\mathfrak{z}, v)^{Q + 1}} |f_{2}(v)|^{2} \, dv \right)^{\frac{1}{2}} \\
& \quad \lesssim_{T, s} 2^{- \frac{j (1 - a) }{4}} (2^{l}r)^{-(\frac{Q + 1}{2})} |B(\mathfrak{z}, 2^{l}sr)|^{\frac{1}{2}} \mathcal {M}_{2} f(x) \\ 
& \quad \lesssim_{T,s} 2^{- \frac{j ( 1 - a )}{4}} 2^{- l/2} \mathcal{M}_{2} f(x), 
\end{align*}
where the second last inequality follows using the fact $d(y,v) \sim d(\mathfrak{z}, v)$, and in the last inequality we have used the fact that $r \geq 1$ and $ |B(\mathfrak{z}, 2^{l}sr)| \lesssim_{s} (2^{l}r)^{Q} $ for $ 2^{l+1} sr \geq \frac{1}{2}| \mathfrak{z}'| $.

Next, for those $l$ for which $ 2^{l+1} sr <  \frac{1}{2} |\mathfrak{z}'| $ holds, we have 
\begin{align*}
& \int_{\mathcal{A}_{l}} |K_j(y,v) - K_j (\mathfrak{z}, v)| |f_{2}(v)| \, dv \\
& \quad \lesssim \left( \int_{\mathcal{A}_{l}} d(y, v)^{n_1 + n_2 + 1} |K_j (y,v)|^{2} \, dv \right)^{\frac{1}{2}} \times \left( \int_{\mathcal{A}_{l}} \frac{1}{d(y,v)^{n_1 + n_2 + 1}} |f_{2}(v)|^{2} \, dv \right)^{\frac{1}{2}} \\
& \quad \quad + \left( \int_{\mathcal{A}_{l}} d(\mathfrak{z}, v)^{n_1 + n_2 + 1} |K_j (\mathfrak{z}, v)|^{2} \, dv \right)^{\frac{1}{2}} \times \left( \int_{\mathcal{A}_{l}} \frac{1}{d(\mathfrak{z}, v)^{n_1 + n_2 + 1}} |f_{2}(v)|^{2} \, dv \right)^{\frac{1}{2}} \\
& \quad \lesssim_{T,s}  2^{-\frac{j(1-a)}{4}} |\mathfrak{z}'|^{-\frac{\varkappa n_2}{2}} \frac{(2^{l}r)^{\frac{n_1 + n_2}{2}}}{(2^{l}r)^{(\frac{n_1 + n_2 + 1}{2})}} |\mathfrak{z}'|^{\frac{\varkappa n_2}{2}}  \mathcal{M}_{2} f(x) \\
& \quad \lesssim_{T,s} 2^{-\frac{j(1-a)}{4}} 2^{-l/2} \mathcal{M}_{2} f(x), 
\end{align*}
where in the second last inequality we have used the fact $d(y,v) \sim d(\mathfrak{z}, v)$ and $ |B(\mathfrak{z}, 2^{l}sr)| \lesssim_{s} (2^{l}r)^{n_1 + n_2} |\mathfrak{z}'|^{\frac{\varkappa n_2}{2}} $ for $ 2^{l+1} sr < \frac{1}{2}| \mathfrak{z}'| $.

In view of the above two estimates, we have 
\begin{align*}
|Tf_{2}(y) - Tf_{2}(\mathfrak{z})| \leq \mathcal{M}_{2}f(x), 
\end{align*}
completing the proof in Case 1.

\medskip \noindent \textbf{\underline{Case 2 $(r < 1)$}:} With $B' = B(\mathfrak{z}, s r^{1-a})$, this time we decompose $ f = f_{1} + f_{2}$ where $f_{1} = f \chi_{B'} $. 

Now, 
\begin{align*}
\frac{1}{|B|}\int_{B} |Tf(x) - Tf_{2} (\mathfrak{z})| \, dx \leq \widetilde{II}_1 + \widetilde{II}_2,
\end{align*}
where $$ \widetilde{II}_1 = \frac{1}{|B|} \int_{B} |Tf_{1}(x)| \, dx \quad \text{and} \quad \widetilde{II}_2 = \frac{1}{|B|} \int_{B} |Tf_{2}(x) - Tf_{2} (\mathfrak{z}) | \, dx.$$ 

Let us first estimate $\widetilde{II}_1$. For $b \in \mathbb{R}$, denote by $ J_{b} $ the potential operator $(I + G)^{b/2}$. Since $ x \in B(\mathfrak{z}, r)$ and $r < 1$, we have $|B(\mathfrak{z}, 1)| \sim |B(x, 1)|$. Now, using Theorem \ref{thm:Compo-operator-L2-boundedness}, and the Hardy--Littlewood--Sobolev inequality \eqref{eq:Sobolev-ineq} for $J_{-Qa/2}$ with $q = \frac{2}{1-a}$ and $p = 2$, we get
\begin{align*}
\widetilde{II}_1 = \frac{1}{|B|}\int_{B}|Tf_{1}(x)| \, dx & \leq \frac{1}{|B|} |B|^{\frac{1}{q'}}\left(\int_{B} |T f_{1}|^q \, dx \right)^{\frac{1}{q}} \\
& = \frac{1}{|B|} |B|^{\frac{1}{q'} }\left(\int_{B} |J_{-Qa/2} J_{Qa/2} Tf_{1}|^q \, dx \right)^{\frac{1}{q}} \\
& \lesssim \frac{1}{|B|} |B|^{\frac{1}{q'}}|B(\mathfrak{z}, 1)|^{-\frac{a}{2}}\left(\int_{B}(|B(x, 1)|^{\frac{a}{2}} |J_{-Qa/2} J_{Qa/2} Tf_{1}|)^q \, dx \right)^{\frac{1}{q}} \\
&  \lesssim_{a} \frac{1}{|B|} |B|^{\frac{1}{q'}} |B(\mathfrak{z}, 1)|^{-\frac{a}{2}} \left( \int_{\mathbb{R}^{n_1 + n_2}} |J_{Qa/2} T f_{1}|^2 \, dx \right)^{\frac{1}{2}} \\
&  \lesssim_{T} \frac{1}{|B|} |B|^{\frac{1}{q'}} |B(\mathfrak{z}, 1)|^{-\frac{a}{2}} |B'|^{\frac{1}{2}} \left(\frac{1}{|B'|} \int_{B'} |f_{1}|^2 \, dx \right)^{\frac{1}{2}}\\
& \lesssim \frac{1}{|B|} |B|^{\frac{1}{q'}} |B(\mathfrak{z}, 1)|^{-\frac{a}{2}} |B'|^{\frac{1}{2}} \mathcal{M}_{2}f(x).
\end{align*}	

Now, observe that if $ sr^{1-a} > |\mathfrak{z}'| $, then  
\begin{align*}
\frac{1}{|B|} |B|^{\frac{1}{q'}} |B(\mathfrak{z}, 1)|^{-\frac{a}{2}} |B'|^{\frac{1}{2}} \lesssim \frac{1}{r^Q} r^{\frac{(1 + a) Q}{2}} r^{\frac{(1 - a) Q}{2}} s^{\frac{Q}{2}} \leq C_{s},
\end{align*}
whereas, if $ s r^{1-a} \leq |\mathfrak{z}'|$, then
\begin{align*}
\frac{1}{|B|} |B|^{\frac{1}{q'}} |B(\mathfrak{z}, 1)|^{-\frac{a}{2}} |B'|^{\frac{1}{2}} \lesssim r^{\frac{(n_1 + n_2)(a-1)}{2}} |\mathfrak{z}'|^{\frac{\varkappa n_{2}(a-1)}{2}} |\mathfrak{z}'|^{-\frac{\varkappa n_{2} a}{2}}  r^{\frac{(1-a)(n_1 + n_2)}{2}} s^{\frac{n_1 + n_2}{2}} |\mathfrak{z}'|^{\frac{n_2}{2}} \leq C_{s}.
\end{align*}
Put together, we get that $\widetilde{II}_1 \lesssim_{T, s} \mathcal{M}_{2}f(x)$. 

Next, we estimate the term $\widetilde{II}_2$. Let us write $\mathcal{B}_{l}: = \{v:  s 2^{l} r^{1-a} \leq d(v, \mathfrak{z}) \leq s 2^{l + 1} r^{1-a}\}$, and let $l_{0} \geq 0$ be such that $\frac{1}{2}|\mathfrak{z}'| \geq s 2^{l} r^{1-a}$ for all $l \leq l_{0}$. 

Now, 
\begin{align*}
|Tf_{2}(x) - Tf_{2}(\mathfrak{z})|
& \lesssim \sum_{j = 0}^{\infty} \sum_{l = 0}^{l_0} \left( \int_{\mathcal{B}_{l}} |K_j (x,y) - K_j (\mathfrak{z},y)|^{2} \, dy \right)^{\frac{1}{2}} \times \left(\int_{\mathcal{B}_{l}} |f_{2}(y)|^{2} \, dy \right)^{\frac{1}{2}}\\
& \quad + \sum_{j = 0}^{\infty} \sum_{l > l_0}^{\infty} \left( \int_{\mathcal{B}_{l}} |K_j (x,y) - K_j (\mathfrak{z},y)|^{2} \, dy \right)^{\frac{1}{2}} \times \left(\int_{\mathcal{B}_{l}} |f_{2}(y)|^{2} \, dy \right)^{\frac{1}{2}} \\
& =: \widetilde{II}_{2,1} + \widetilde{II}_{2,2}.
\end{align*}

On one hand, using \eqref{lem:kernel-estimate-difference-2} from Lemma \ref{lem:kernel-estimate-channilo-type-joint-funct}, we get 
\begin{align*}
\widetilde{II}_{2,1} & \lesssim_{T, L} \sum_{l=0}^{l_0}  (2^{l} r^{1-a})^{-L} r^{(1-a) (L- \frac{(n_1 + n_2)}{2})} |\mathfrak{z}'|^{-\frac{\varkappa n_2}{2}} |B(\mathfrak{z}, 2^{l} sr^{1-a})|^{\frac{1}{2}} \mathcal{M}_{2} f(x)\\
& \lesssim_{T, L} \sum_{l=0}^{l_0}  (2^{l} r^{1-a})^{-L} r^{(1-a) (L- \frac{(n_1 + n_2)}{2})} |\mathfrak{z}'|^{-\frac{\varkappa n_2}{2}} (2^{l} s r^{1-a})^{\frac{n_1 + n_2}{2}} |\mathfrak{z}'|^{\frac{\varkappa n_2}{2}} \mathcal{M}_{2} f(x)\\
& \lesssim_{s} \sum_{l=0}^{l_0} 2^{-l (L-\frac{n_1 + n_2}{2})} \mathcal{M}_{2} f(x) \\ 
& \lesssim_{L} \mathcal{M}_{2} f(x), 
\end{align*}
provided that $L > \frac{n_1+n_2}{2}$. 

On the other hand, using \eqref{lem:kernel-estimate-difference-1} from Lemma \ref{lem:kernel-estimate-channilo-type-joint-funct}, we get
\begin{align*}
\widetilde{II}_{2,2} \lesssim_{T, L} \sum_{l = l_{0}}^{\infty} (2^l r^{1-a})^{-L} r^{(1-a)(L-\frac{Q}{2})} |B(\mathfrak{z}, 2^{l} sr^{1-a})|^{\frac{1}{2}} \mathcal{M}_{2} f(x) & \lesssim_{s} \sum_{l=0}^{l_0} 2^{-l (L - \frac{Q}{2})} \mathcal{M}_{2} f(x) \\ 
& \lesssim_{L} \mathcal{M}_{2} f(x),
\end{align*}
provided that $ L > \frac{Q}{2}$. 

This completes the proof of Theorem \ref{thm:Chanilloy-type}. 
\end{proof}


\subsection{Proof of Theorem \ref{thm:joint-fun-Channilo-type}} \label{subsec:Proof-thm:joint-fun-Channilo-type}
Proof of Theorem~\ref{thm:joint-fun-Channilo-type} is a consequence of Theorem~\ref{thm:Chanilloy-type} and good-$\lambda$-inequality. The proof follows from standard arguments. But, for self-containment, we write below a brief sketch. 

\begin{proof}[Proof of Theorem \ref{thm:joint-fun-Channilo-type}]
Recall that it was proved in part (i) of Lemma 4.11 in \cite{Grafakos-Homogeneous-Multilinear} that the following inequality holds: 
\begin{align}
\label{F-S}
\int_{\mathbb{R}^{n_1 + n_2}} \mathcal{M} F(x)^p w(x) \, dx \lesssim \int_{\mathbb{R}^{n_1 + n_2}}\mathcal{M}^{\sharp} F(x)^p w(x)~dx,    
\end{align}
for any $w \in A_{\infty}(\mathbb{R}^{n_1 + n_2})$, $0 < p_{0} < p < \infty$, and for all $F \in L^1_{loc}$ such that $\mathcal{M} F \in L^{p_{0}, \infty}(w)$. 

In our case we are concerned with $2 < p < \infty$ and $w \in A_{p/2}(\mathbb{R}^{n_1 + n_2})$. With $p$ and $w$ fixed, the reverse H\"older's inequality for Muckenhoupt weights implies there is $p_0> 2$ such that $w \in A_{p_0/2}(\mathbb{R}^{n_1 + n_2})$ with $p>p_0 > 2$. 

Now, if we can we prove that $\|m(\cdot, \boldsymbol{L}, \boldsymbol{U}) f\|_{L^{p_0}(w)} \leq C_{m(\cdot, \boldsymbol{L}, \boldsymbol{U}), f, p_0, w} < \infty$, then we can apply the above inequality \eqref{F-S} to $F = |m(\cdot, \boldsymbol{L}, \boldsymbol{U}) f|$ and Theorem \ref{thm:Chanilloy-type} to conclude that 
\begin{align*}
\int_{\mathbb{R}^{n_1 + n_2}} |m(x, \boldsymbol{L}, \boldsymbol{U}) f(x)|^{p} \, w(x) \, dx & \leq \int_{\mathbb{R}^{n_1 + n_2}} \mathcal{M} (m(\cdot, \boldsymbol{L}, \boldsymbol{U}) f)(x)^{p} \, w(x) \, dx \\
& \lesssim \int_{\mathbb{R}^{n_1 + n_2}} \mathcal{M}^{\sharp} ( m(\cdot, \boldsymbol{L}, \boldsymbol{U}) f ) (x)^{p} \, w(x) \, dx \\
& \lesssim \int_{\mathbb{R}^{n_1 + n_2}} \mathcal{M}_{2} f (x)^{p} \, w(x) \, dx \\ 
& \lesssim \int_{\mathbb{R}^{n_1 + n_2}} |f(x)|^{p} \, w(x) \, dx,
\end{align*}
where the last inequality follows from the fact that $\mathcal{M}_{2}: L^{p}(w) \to L^{p}(w)$ is bounded for $w\in A_{p/2}$ with $p> 2$.

So, we are left with showing that $\|m(\cdot, \boldsymbol{L}, \boldsymbol{U}) f\|_{L^{p_0}(w)} \leq C_{m(\cdot, \boldsymbol{L}, \boldsymbol{U}), f, p_0, w} < \infty,$ for any compactly supported bounded function $f$. 

Let $f \in C^{\infty}_{c}(\mathbb{R}^{n_1 + n_2})$ be supported on $B(\mathfrak{z}, r)$. Without loss of generality we may assume that $r > 1$. Since $m(x, \boldsymbol{L}, \boldsymbol{U})$ is bounded on $L^2(\mathbb{R}^{n_1+n_2})$, we have $m(x, \boldsymbol{L}, \boldsymbol{U})f\in L^2(\mathbb{R}^{n_1+n_2})$ implying $\mathcal{M}(m(x, \boldsymbol{L}, \boldsymbol{U})f)\in L^2(\mathbb{R}^{n_1+n_2})$. Therefore, an application of \eqref{F-S} with $w=1$ and the pointwise domination $\mathcal{M}^{\sharp}(m(x, \boldsymbol{L}, \boldsymbol{U})f)(x)\lesssim \mathcal{M}_{2}f(x)$, we obtain $m(x, \boldsymbol{L}, \boldsymbol{U})f\in L^q(\mathbb{R}^{n_1+n_2})$ for all $q>2$.

Now, 
\begin{align}\label{eq:local-weighted-bound}
& \int_{B(\mathfrak{z}, 2 C_0 r)} |m(x, \boldsymbol{L}, \boldsymbol{U}) f(x)|^{p_0} w(x) \, dx \\
\nonumber & \leq \left(\int_{B(\mathfrak{z}, 2 C_0 r)} w(x)^{1 + \epsilon} \, dx \right)^{\frac{1}{1 + \epsilon}} \left(\int_{ B( \mathfrak{z}, 2 C_0 r)} |m(x, \boldsymbol{L}, \boldsymbol{U}) f(x)|^{p_{0} \frac{1 + \epsilon}{\epsilon}} \, dx \right)^{\frac{\epsilon}{1 + \epsilon}},
\end{align} 
where $\epsilon > 0$ is chosen such that the reverse H\"older's inequality is satisfied and the first term of right side of \eqref{eq:local-weighted-bound} is finite. Since $m(x, \boldsymbol{L}, \boldsymbol{U})f\in L^q(\mathbb{R}^{n_1+n_2})$ for all $q>2$, we conclude that the second term in the right side of \eqref{eq:local-weighted-bound} is finite. 

Next, in order to estimate 
$$\int_{\mathbb{R}^{n_1 + n_2} \setminus B(\mathfrak{z}, 2 C_0 r)} |m(x, \boldsymbol{L}, \boldsymbol{U}) f(x)|^{p_0} w(x) \, dx,$$ 
note first that $d(x, \mathfrak{z}) \geq 2 C_0 r$ and $ d(y, \mathfrak{z}) < r$ together imply $d(x,y) \geq \frac{1}{2 C_0} d(x, \mathfrak{z})$. Fix some $x$, and let $l$ be such that $d(x, \mathfrak{z}) \simeq 2^{l}r$. Now, if $2^{l} r \geq \frac{1}{2} |\mathfrak{z}'|$ then we proceed as follows: 
\begin{align*}
|m_j(x, \boldsymbol{L}, \boldsymbol{U}) f(x)| & \leq \int_{B(\mathfrak{z}, r)} |K_j(x, y)||f(y)| \, dy \\
& \leq \left(\int_{B(\mathfrak{z}, r)} |K_j(x, y)|^2  d(x, y)^{Q+1}\right)^{1/2} \left( \int_{B(\mathfrak{z}, r)} d(x, y))^{-(Q + 1)} |f(y)|^2 \, dy \right)^{1/2} \\
& \lesssim 2^{-\frac{j (1 - a)}{4}} \frac{|B(\mathfrak{z}, 2^l r)|^{1/2}}{(2^l r)^{(Q+1)/2}} \mathcal{M}_{2}f(x) \lesssim 2^{-\frac{j (1-a)}{4}} \mathcal{M}_{2}f(x).
\end{align*} 
Similarly, if $2^{l} r \leq \frac{1}{2} |\mathfrak{z}'|$, we just use $d(x, y)^{n_1+n_2+1}$ in place of $d(x, y)^{Q+1}$ to conclude that 
\begin{align*}
|m_j(x,\boldsymbol{L}, \boldsymbol{U})f(x)|\lesssim 2^{-\frac{j(1-a)}{4}} |\mathfrak{z}'|^{-\frac{\varkappa n_2}{2}} \frac{(2^{l}r)^{\frac{n_1 + n_2}{2}}}{(2^{l}r)^{(\frac{n_1 + n_2 + 1}{2})}} |\mathfrak{z}'|^{\frac{\varkappa n_2}{2}}  \mathcal{M}_{2} f(x) \lesssim_{T,s} 2^{-\frac{j(1-a)}{4}} 2^{-l/2} \mathcal{M}_{2} f(x).
\end{align*}

In view of the above estimates, we get 
\begin{align*}
\int_{d(x, \mathfrak{z})\geq 2C_0r} |m(x, \boldsymbol{L}, \boldsymbol{U})f(x)|^{p_0} w(x) \, dx & \lesssim \int_{d(x, \mathfrak{z})\geq 2C_0r} (\mathcal{M}_{2}f(x))^{p_0} w(x)   \, dx\\
& \lesssim \int |f(x)|^{p_0} w(x) \, dx \\ 
& \lesssim w(B(\mathfrak{z}, r)) \, \|f\|_{L^{\infty}},
\end{align*}
where again the last inequality follows from the fact that $\mathcal{M}_{2}: L^{p_0}(w) \to L^{p_0}(w)$ is bounded for $w\in A_{p_0/2}$ with $p_0 > 2$. 

This completes the proof of the Theorem \ref{thm:joint-fun-Channilo-type}.
\end{proof}


\section*{Acknowledgements} 
SB and RG were supported in parts from their individual INSPIRE Faculty Fellowships from DST, Government of India. RB was supported by the Senior Research Fellowship from CSIR, Government of India. AG was supported in parts by the  INSPIRE Faculty Fellowship of RG and institute postdoctoral fellowships from IISER Bhopal and Centre for Applicable Mathematics, TIFR.


\providecommand{\bysame}{\leavevmode\hbox to3em{\hrulefill}\thinspace}
\providecommand{\MR}{\relax\ifhmode\unskip\space\fi MR }
\providecommand{\MRhref}[2]{%
  \href{http://www.ams.org/mathscinet-getitem?mr=#1}{#2}
}
\providecommand{\href}[2]{#2}

\end{document}